\documentclass[11pt]{amsart} 
\usepackage{graphicx,amsfonts,psfrag,layout,appendix,subfigure}
\usepackage{fancyhdr}
  
\usepackage{amssymb,amsthm,amsmath, mathabx, mathrsfs}
\usepackage{amssymb}
\usepackage{amscd}
\usepackage{amsmath}
\usepackage{amsmath}
 \usepackage[all,cmtip]{xy}
\usepackage{color}
\usepackage{textcomp}

\theoremstyle{plain}

\newtheorem{theorem}{Theorem}[section]
\newtheorem{lemma}[theorem]{Lemma}
\newtheorem{corollary}[theorem]{Corollary}
\newtheorem{proposition}[theorem]{Proposition}
\newtheorem{remark}[theorem]{Remark}
\newtheorem{conjecture}[theorem]{Conjecture}

\newtheorem{example}[theorem]{Example}

\newtheorem{definition}[theorem]{Definition}

\font\russ=wncyr10  1
\def\sha{\hbox{\russ\char88}}

\newcommand{\crys}{\text{\rm crys}}
 \newcommand{\et}{\text{\rm\'et}}
\newcommand{\ket}{\text{\rm k\'et}}

\newcommand{\Cech}{\v{C}ech}
\newcommand{\Zp}{\mathbb{Z}_p}
\newcommand{\riso}{\xrightarrow{\sim}}

\DeclareMathOperator{\Spf}{Spf}

\DeclareMathOperator{\Gal}{Gal}

\DeclareMathOperator{\Hom}{Hom}

\DeclareMathOperator{\Spec}{Spec}

\DeclareMathOperator{\im}{im}

\DeclareMathOperator{\Fr}{Fr}

\DeclareMathOperator{\tr}{tr}

\newcommand{\C}{\mathbb{C}}

\newcommand{\CC}{\mathbb{C}}

\newcommand{\FF}{\mathbb{F}}

\newcommand{\QQ}{\mathbb{Q}}
\newcommand{\Q}{\mathbb{Q}}

\newcommand{\RR}{\mathbb{R}}

\newcommand{\ZZ}{\mathbb{Z}}
\newcommand{\Z}{\mathbb{Z}}
\newcommand{\R}{\mathbb{R}}
\newcommand{\bz}{\mathbb{Z}}

\newcommand{\F}{\mathbb{F}}

\newcommand{\cale}{\mathcal E}

\newcommand{\cala}{\mathcal A}

\newcommand{\calf}{\mathcal F}

\newcommand{\gotu}{\mathfrak{U}}
\newcommand{\gotx}{\mathfrak{X}}

\newcommand{\calA}{\mathcal{A}}

\newcommand{\cok}{\mathrm{cok}}

\newcommand{\tors}{\mathrm{tors}}

\newcommand{\Ru}[1]{Ru_{{#1}/K}}

\begin{document}

%%%%%%%%%%%%%%%%%%%%%%%%%%%%%%%
%%% TITLE, AUTHOR, ABSTRACT %%%
%%%%%%%%%%%%%%%%%%%%%%%%%%%%%%%
\title[]{\sc On a refinement of the \\ Birch and Swinnerton-Dyer Conjecture\\ in positive characteristic}

\author{David Burns, Mahesh Kakde, Wansu Kim}%

%\date{today}

\address{King's College London,
Department of Mathematics,
London WC2R 2LS,
U.K.}
\email{david.burns@kcl.ac.uk}

\address{Indian Institute of Science,
Department of Mathematics,
Bangalore 560012,
India}
\email{maheshkakde@iisc.ac.in}

\address{KAIST,
Department of Mathematical Sciences,
291 Daehak-ro, Yuseong-gu
Daejeon, 34141,
Republic of Korea}
\email{wansu.math@kaist.ac.kr}

\begin{abstract} We formulate a refined version of the Birch and Swinnerton-Dyer conjecture for abelian varieties over global function fields. This refinement incorporates both new families of algebraic relations between leading terms (at $s=1$) of Hasse-Weil-Artin $L$-series and restrictions on the Galois structure of Selmer complexes, and constitutes a natural analogue for abelian varieties over function fields of the equivariant Tamagawa number conjecture for abelian varieties over number fields. We provide strong supporting evidence for the conjecture including giving a full proof, modulo only the assumed finiteness of Tate-Shafarevich groups, in an important class of examples. \end{abstract}

\maketitle
\setcounter{tocdepth}{1}
\tableofcontents

\section{Introduction}\label{intro}

\subsection{}Let $A$ be an abelian variety that is defined over a function field $K$ in one variable over a finite field of characteristic $p$.

In \cite{tate} Artin and Tate formulated a precise conjectural formula for the leading term at $s=1$ of the Hasse-Weil $L$-series attached to $A$.

This formula constituted a natural `geometric' analogue of the Birch and Swinnerton-Dyer Conjecture for abelian varieties over number fields and was subsequently verified unconditionally by Milne \cite{milne68} in the case that $A$ is constant and by Ulmer \cite{ulmer2002} in certain other special cases. Further partial results have been obtained by many other authors and these efforts culminate in the main result of the seminal article of Kato and Trihan \cite{KT} which shows that the conjecture is valid whenever there exists a prime $\ell$ such that the $\ell$-primary component of the Tate-Shafarevich group of $A$ over $K$ is finite.

In this article we now formulate, and provide strong evidence for, a refined version of this conjecture that also incorporates new families of algebraic relations between the (suitably normalised) leading terms at $s=1$ of the Hasse-Weil-Artin $L$-series that are attached to $A$ and to irreducible complex characters (with open kernel) of the absolute Galois group of $K$. This conjecture is a natural analogue for abelian varieties over function fields of the equivariant Tamagawa number conjecture (`ETNC'), including the $p$-primary part, for the motive $h^1(A)(1)$ of abelian varieties $A$ over number fields.

To be a little more precise about our results we now fix a finite Galois extension $L$ of $K$ with group $G$.

Then, as a first step, we shall prove that the leading terms of the Hasse-Weil-Artin $L$-series that are attached to $A$ and to the irreducible complex characters of $G$ are interpolated by a canonical element of the Whitehead group $K_1(\R[G])$ of the group ring $\R[G]$. (This result is, a priori, far from clear and requires one to prove, in particular, that leading terms at irreducible symplectic characters are strictly positive.)

Our central conjecture is then a precise formula for the image of this element under the connecting homomorphism from $K_1(\R[G])$ to the relative algebraic $K_0$-group $K_0(\Z[G],\R[G])$ of the ring inclusion $\Z[G]\subset \RR[G]$.

The conjectural formula involves a canonical Euler characteristic element that is constructed by combining a natural `Selmer complex' of $G$-modules together with the classical N\'{e}ron-Tate height pairing of $A$ over $L$. This Selmer complex is constructed from the flat cohomology of the torsion subgroup scheme of the N\'eron model of $A$ over the projective curve $X$ with function field $K$ and, provided that the relevant Tate-Shafarevich groups are finite, is both perfect over $\ZZ[G]$ and has cohomology groups that are closely related to the classical Mordell-Weil and Selmer groups of $A^t$ and $A$ over $L$.

The formula also involves the Euler characteristic of an auxiliary perfect complex of $G$-modules that is constructed directly from the Zariski cohomology of an appropriate line bundle over $X$ and is necessary in order to compensate for certain choices of pro-$p$ subgroups that are made in the definition of the Selmer complex.

If $L=K$, then $K_0(\Z[G],\R[G])$ identifies with the quotient of the multiplicative group $\RR^\times$ by $\{\pm 1\}$ and we check that in this case our conjecture recovers the classical Birch and Swinnerton-Dyer conjecture for $A$ over $K$.

In the general case, the conjecture incorporates both a family of precise algebraic relations between the  normalised leading terms of Hasse-Weil-Artin $L$-series attached to $A$ and to characters of $G$ and also strong restrictions on the Galois structure of Selmer complexes (for more details see the discussion in \S\ref{cons2}).

To study the conjecture, we adapt (and, in some respects, clarify) certain constructions and arguments from \cite{KT} relating to syntomic cohomology complexes. In this way we are able to prove that, whenever there exists a prime $\ell$ such that the $\ell$-primary component of the Tate-Shafarevich group of $A$ over $L$ is finite, then our conjecture is valid modulo a certain finite subgroup $\mathcal{T}_{A,L/K}$ of $K_0(\Z[G],\R[G])$, the nature of which depends both on the reduction properties of $A$ and the ramification behaviour in $L/K$.

For example, if $A$ is semistable over $K$ and $L/K$ is tamely ramified, then $\mathcal{T}_{A,L/K}$ vanishes and so we obtain a full verification of our conjecture in this case.

In the worst case the group $\mathcal{T}_{A,L/K}$ coincides with the torsion subgroup of the subgroup $K_0(\Z[G],\Q[G])$ of $K_0(\Z[G],\R[G])$ and our result essentially amounts to proving a version of the main result of \cite{KT} for the leading terms of the Hasse-Weil-Artin $L$-series attached to $A$ and to each character of $G$.

However, even the latter result is new and of interest since, for example, it both establishes the `order of vanishing' part of the Birch and Swinnerton-Dyer conjecture for Hasse-Weil-Artin $L$-series and, in addition, plays a key role in a forthcoming complementary article that deals with, modulo the standard finiteness hypothesis on Tate-Shafarevich groups, the case of abelian varieties $A$ that are generically ordinary.

As a key step in the proof of our main result we shall combine Grothendieck's description of Hasse-Weil-Artin $L$-series in terms of the action of Frobenius on $\ell$-adic cohomology (for some prime $\ell\not= p$) together with a result of Schneider concerning the N\'eron-Tate height-pairing to show that our conjectural formula naturally decomposes as a sum of `$\ell$-primary parts' over all primes $\ell$.

It is thus of interest to note that in some related recent work Trihan and Vauclair \cite{trihanvauclair2} have adapted the approach of \cite{KT} in order to formulate and prove a natural main conjecture of ($p$-adic) non-commutative Iwasawa theory for $A$ relative to unramified $p$-adic Lie extensions of $K$ under the assumptions both that $A$ is semistable over $K$ and that certain Iwasawa-theoretic $\mu$-invariants vanish.

In addition, for each prime $\ell \neq p$, Witte \cite{Witte} has used techniques of Waldhausen $K$-theory to deduce an analogue of the main conjecture of non-commutative Iwasawa theory for $\ell$-adic sheaves over arbitrary $p$-adic Lie extensions of $K$ from Grothendieck's formula for the Zeta function of such sheaves.

It seems likely that these results can be combined with the descent techniques developed by Venjakob and the first author in \cite{BV2} and the explicit interpretation of height pairings in terms of Bockstein homomorphisms that we use below to give an alternative, although rather less direct, proof of the $\ell$-primary part of our main result for any $\ell\not =p$ and of the $p$-primary part of our main result in the special case that $L/K$ is unramified and suitable $\mu$-invariants vanish.

However, even now, there are still no ideas as to how one could formulate a main conjecture of (non-commutative) Iwasawa theory for $A$ relative to any general class of ramified $p$-adic Lie extensions of $K$.

It is thus one of the main observations of the present article that the techniques developed by Kato and Trihan in \cite{KT} are essentially themselves sufficient to prove refined versions of the Birch and Swinnerton-Dyer conjecture without the need to develop an appropriate formalism of non-commutative Iwasawa theory (and hence without the need to assume the vanishing of relevant $\mu$-invariants). 

This general philosophy also in fact underpins the complementary work of the first two authors regarding generically ordinary abelian varieties.

In a little more detail, the main contents of this article is as follows.

Firstly, in \S\ref{leading term section} we use the leading terms of the Hasse-Weil-Artin $L$-series attached to complex characters of $G$ to define a canonical element of $K_1(\RR[G])$. Then, in \S\ref{sect: arith}, we define a natural family of `Selmer complexes' of $G$-modules and establish some of its key properties.

In \S\ref{sect: statements} we formulate our main conjecture, describe some of its explicit consequences and state the main supporting evidence for the conjecture that we prove in later sections.

In \S\ref{sect: preliminary} we prove certain useful preliminary results including a purely $K$-theoretic observation that plays a key role in several subsequent calculations. We also show that our conjecture is consistent in some important respects and use a result of Schneider to give a useful reformulation of the conjecture.

In \S\ref{sec:syncom} and \S\ref{sect:tame} we investigate the syntomic cohomology complexes introduced by Kato and Trihan in \cite{KT}, with
a particular emphasis on understanding conditions under which these complexes can be shown to be perfect.

In \S\ref{sec:lfunctions} we analyse when certain morphisms of complexes that arise naturally in the theory are `semisimple' (in the sense of Galois descent) and deduce, modulo the assumed finiteness of Tate-Shafarevich groups, the order of vanishing part of the Birch and Swinnerton-Dyer conjecture for Artin Hasse-Weil $L$-series.

Then, in \S\ref{sect:proof}, we combine the results established in earlier sections to prove our main results.

There are also two appendices to this article. In the first of these, we show that coherent cohomology over a `separated' formal fs log scheme can be computed via the \v{C}ech resolution with respect to an affine Kummer-\'etale covering. (This result plays an important role in the arguments of \S\ref{sect:tame} and, whilst it is surely well-known to experts, we have not been able to find a good reference for it.) 

Then, in the second appendix, we extend the notion of overconvergent $\Lambda$-$F$-isocrystal for a finite extension $\Lambda$ of $\QQ_p$ whose residue field is not necessarily contained in the field of constants of the base curve, and also the Lefschetz trace formula for rigid cohomology with such coefficients. (This result is needed to obtain Theorem~\ref{order-leading} without further restriction, and the proof is a mere repetition of the proof of Etesse and Le Stum in \cite{ELS}.)

\subsection{} To end the introduction we collect together certain notation and conventions that are to be used in the sequel.

We fix a prime number $p$ and a function field $K$ in one variable over a finite field of characteristic $p$. We write $X$ for the proper smooth connected curve over $\F_p$ that has function field $K$.

Let $A$ be an abelian variety over $K$. Let $U$ be a dense open subset of $X$ such that $A/K$ has good reduction on $U$. We write $\mathcal{A}$ for the N\'eron model of $A$ over $X$.

Let $F$ be a finite extension of $K$. Let $X_F$ denote the proper smooth curve over $\FF_p$ that has function field $F$. We will denote the `base extension' of an object $\ast$ over either $K$ or $X$ to that over $F$ of $X_F$ by a subscript $\ast_{F}$. For example $A_F$ and $U_F$ denote $A \times_K F$,  $U \times_X X_F$ respectively. If there is no danger of confusion we often omit the subscript $F$.

If $M$ is an abelian group or complex of abelian groups, we denote its Pontryagin dual $ \Hom(M, \Q/\Z)$ by $M^\ast$.  If $W$ is a $\QQ_\ell$-module or complex of $\QQ_\ell$-modules for some prime $\ell$, we denote its linear dual $\Hom_{\QQ_\ell}(V,\QQ_\ell)$ by $V^\vee$ (regarding $\ell$ as clear from context). If either $M$ or $V$ has a left action of a group, then we endow $M^\ast$ and $V^\vee$ with the corresponding left contragredient action.  

We fix an algebraic closure $\QQ^c$ of $\Q$ and for every prime $\ell$ an algebraic closure $\QQ^c_{\ell}$ of $\Q_{\ell}$ and  the $\ell$-adic completion $\C_{\ell}$ of $\QQ^c_{\ell}$. For every prime $\ell$, we also fix an embedding  $\Q^c\to \Q^c_{\ell}$.

For each natural number $n$ the $n$-torsion subgroup of an abelian group $M$ is denoted by $M[n]$. The full torsion subgroup of $M$ is denoted by $M_{\rm tor}$ and, for each prime $\ell$, the $\ell$-primary part of $M_{\rm tor}$ is denoted by $M\{\ell\}$.

For a finite group $G$ we write ${\rm Ir}(G)$ for the set of its irreducible complex valued characters and ${\rm Ir}^{s}(G)$ for the subset of ${\rm Ir}(G)$ comprising characters that are symplectic. We write $\check\chi$ for the contragredient of each $\chi$ in ${\rm Ir}(G)$ and ${\bf 1}_G$ for the trivial character  of $G$. 

For any commutative ring $R$ we write $R[G]$ for the group ring of $G$ over $R$ and denote its centre by $\zeta(R[G])$. We identify $\zeta(\CC[G])$ with ${\prod}_{{\rm Ir}(G)}\CC$ in the standard way. \\

\noindent{\bf Acknowledgement:} We are very grateful indeed to both Fabien Trihan and Takashi Suzuki for several extremely helpful conversations, e-mail exchanges and remarks concerning  previous drafts of this article. We also thank the anonymous referee for the careful reading and helpful comments. The second named author is supported by DST-SERB grant SB/SJF/2020-21/11, SERB SUPRA grant SPR/2019/ 000422, DST FIST program - 2021 [TPN - 700661], and Tata Education and Development Trust.
The third named author was supported by the National Research Foundation of Korea (NRF) grant funded by the Korea government (MSIT) (No. RS-2023-00208018).

\section{Leading terms of Hasse-Weil-Artin $L$-series}\label{leading term section}

We fix a finite Galois extension $L/K$ with Galois group $G$, and choose $U$ not to contain any place that ramifies in $L/K$.
For each $\chi$ in ${\rm Ir}(G)$ we write $L_U(A, \chi, s)$ for the Hasse-Weil-Artin $L$-series of the pair $(A,\chi)$ that is truncated by removing the Euler factors for all places outside $U$.

We now show that there exists a canonical element of the Whitehead group $K_1(\RR[G])$ that naturally interpolates the leading terms  $L^*_U(A,\chi,1)$ at $s=1$ in the Taylor expansions of the functions $L_U(A,\chi,s)$ as $\chi$ ranges over ${\rm Ir}(G)$.

This `$K$-theoretical leading term' will then play an important role in the conjecture that we discuss in subsequent sections (but also see Remark \ref{positivity remark} in this regard).

To define the element we use the fact that the algebra $\mathbb{R}[G]$ is semisimple and hence that the classical reduced norm construction induces a homomorphism ${\rm Nrd}_{\RR[G]}$ of abelian groups from $K_1(\RR[G])$ to the subgroup $\zeta(\RR[G])^\times$ of ${\prod}_{{\rm Ir}(G)}\CC^\times$.

\begin{theorem}\label{lt def} There exists a unique element $L^*_{U}(A_{L/K},1)$ of $K_1(\RR[G])$ with the property that ${\rm Nrd}_{\RR[G]}(L^*_{U}(A_{L/K}, 1))_\chi = L^*_U(A,\chi, 1)$ for all $\chi$ in ${\rm Ir}(G)$.\end{theorem}

\begin{proof} Since the natural map $\RR[G]^\times \to K_1(\RR[G])$ is surjective, the Hasse-Schilling-Maass norm theorem implies both that ${\rm Nrd}_{\RR[G]}$ is injective and that its image is equal to the subgroup of ${\prod}_{{\rm Ir}(G)}\CC^\times$ comprising elements $(x_\chi)_\chi$ that satisfy both of the following conditions
\begin{equation*}\label{hsm} \begin{cases} x_{\tau\circ\chi} = \tau(x_\chi), &\text{ for all $\chi$ in ${\rm Ir}(G)$, and}\\
                                          x_\chi \in \mathbb{R} \,\text{ and }\, x_\chi > 0, &\text{ for all $\chi$ in ${\rm Ir}^{\rm s}(G)$,}\end{cases}\end{equation*}
where $\tau$ denotes complex conjugation. (For a proof of this result see \cite[Th. (45.3)]{curtisr}.)

The injectivity of ${\rm Nrd}_{\RR[G]}$ implies that there can be at most one element of $K_1(\RR[G])$ with the stated property and to show that such an element exists it is enough to show that the element $(L^*_U(A,\chi, 1))_\chi$ of ${\prod}_{{\rm Ir}(G)}\CC^\times$ satisfies the above displayed conditions.

This fact is established in Proposition \ref{prop:positive} below. \end{proof}

The following result extends an observation of Kato and Trihan from \cite[Appendix]{KT}.

\begin{proposition} \label{prop:positive}  The following claims are valid for every $\chi$ in ${\rm Ir}(G)$.
\begin{itemize}
\item[(i)] For every automorphism $\omega$ of $\CC$ one has $\tau(L^*_U(A,\chi, 1)) = L^*_U(A,\tau\circ\chi, 1)$. In particular, one has $L^*_U(A,\chi, 1)\in \RR$ if $\chi$ is real valued.

\item[(ii)] Write $\mathbb{F}_{q}$ and $\mathbb{F}_{q'}$ for the total field of constants of $K$ and $L$ respectively. Then if $\chi$ is both real valued and not inflated from a non-trivial one dimensional representation of $\Gal(\mathbb{F}_{q'}/\mathbb{F}_q)$, one has $L^*_U(A, \chi, 1)>0$.
\end{itemize}
\end{proposition}

\begin{proof} 
At the outset we fix a prime $\ell$ with $\ell\not= p$ and write $\QQ^c$ for the algebraic closure of $\QQ$ in $\CC$. We also fix an isomorphism $\CC\cong \CC_\ell$ that we suppress from the notation.

In particular, for each $\rho$ in ${\rm Ir}(G)$ we fix a realisation $V_\rho$ of $\rho$ over $\QQ^c$ and do not distinguish between it and the space $\QQ^c_\ell\otimes_{\QQ^c}V_\rho$.

Now for every $\rho$ in ${\rm Ir}(G)$ Grothendieck \cite{Gro} (see also the proof of \cite[Chap. VI, Th. 13.3]{milne:etale}) has proved that there is an equality of power series
\begin{equation}\label{deligne}
L_U(A, \rho, s) = {\prod}_{i=0}^{i=2}Q_{\rho,i}(q^{-s})^{(-1)^{i+1}},\end{equation}
where each $Q_{\rho,i}(u)$ is a polynomial in $u$ over $\QQ^c$ that can be computed as
\[
Q_{\rho,i}(u) := {\rm det}\bigl(1- u\cdot\varphi_q | H^i_{\et,c}\bigl(U^c, V_\rho\otimes_\QQ V_{\ell}(A)\bigr)(-1)\bigr)
\]
Here $U^c$ denotes $U\times_{\mathbb{F}_q}\mathbb{F}^c_q$, the vector space $V_{\ell}(A)$ is the $\QQ_\ell$-space spanned by the $\ell$-adic Tate module of $A$ and $\varphi_q$ the $q$-th power Frobenius map. We claim that, for each $i$ and every   automorphism $\omega$ of $\CC$, one has 
\begin{equation} \label{first inv} 
\omega(Q_{\rho,i}(u)) =  (Q_{\omega\circ\rho,i}(u)).
\end{equation}

In fact, since Grothendieck's result implies that the polynomial $Q_{\rho, i}(u)$ has coefficients in $\QQ^c$, it is enough to consider automorphisms $\omega$ of $\QQ^c$. Then, for each such $\omega$, the natural isomorphism of $(\QQ_\ell\otimes_\QQ\QQ^c)$-spaces 
\[
\QQ^c \otimes_{\QQ^c,\omega}(V_\rho\otimes_{\QQ}V_\ell(A)) \cong (\QQ^c \otimes_{\QQ^c,\omega} V_\rho)\otimes_\QQ V_\ell(A) \cong V_{\rho^\omega}\otimes_\QQ V_\ell(A)
\]
induces a similar isomorphism of the corresponding sheaves over $U^c$ and hence an isomorphism of cohomology groups
\[
\QQ^c \otimes_{\QQ^c,\omega} H^i_{\et,c}(U^c, V_{\rho}\otimes_\QQ V_\ell(A))\cong H^i_{\et,c}(U^c, V_{\rho^\omega}\otimes_\QQ V_\ell(A)) 
\]
under which $(1 \otimes \varphi_q )$ on the first space corresponds to $\varphi_q$ on the second space. This proves the claimed equality (\ref{first inv}).

The equalities (\ref{first inv}) (for each $i \in \{0,1,2\}$) can then be combined with (\ref{deligne}) to deduce that the orders of vanishing at $s=1$ of the series $L_U(A, \rho, s)$ and $L_U(A, \omega\circ\rho, s)$ are equal and moreover that
\[ \omega(L^*_U(A, \rho, 1)) = L^*_U(A, \omega\circ\rho, 1),\]
as required. The final assertion of claim (i) then follows immediately upon applying this equality with $\omega = \tau$. 

To prove claim (ii) we assume $\rho$ is real-valued and hence, by (\ref{first inv}) with $\omega=\tau$, that each polynomial $Q_{\rho,i}(u)$ belongs to $\mathbb{R}[u]$.

For each $i$ we set 
\[ d_{\rho,i} := {\rm dim}_{\QQ_\ell^c}(H^i_{\et,c}(U^c, V_{\rho}\otimes_\QQ V_\ell(A))(-1))\]
and write the eigenvalues, counted with multiplicity, of $\varphi_q$ on this space as $\{\alpha_{ia}\}_{1\le a\le d_{\rho,i}}$.

Now, since the weight on $U$ of $(V_{\rho}\otimes_\QQ V_{\ell}(A))(-1)$ is one, Deligne \cite{deligne:weil2} has shown that $|\alpha_{ia}| \leq q^{(i+1)/2}$ for all values of $i$ and $a$.

Further, as the space $H^2_{\et,c}(U^c, V_{\rho}\otimes_\QQ V_\ell(A))(-1)$ is dual to
$H^0_{\et}(U^c, V_{\check\rho}\otimes_\QQ V_\ell(A^t))(1)$, and the weight on $U$ of the representation $(V_{\check\rho}\otimes_\QQ V_\ell(A^t))(1)$ is $-3$, one has  $|\alpha_{2a}| = q^{\frac{3}{2}}$ for all $a$.
Therefore neither of the terms $Q_{\rho,0}(q^{-1})$ or $Q_{\rho,2}(q^{-1})$ vanish.

%Since $|\beta_{\rho,b}| = 1$ for all $b$ these observations combine with (\ref{explicit product}) to imply that neither of the terms $Q_{\rho,0}(q^{-1})$ or $Q_{\rho,2}(q^{-1})$ vanishes.

In particular, if $m$ denotes the order of vanishing of $L_U(A,\rho, s)$ at $s=1$, then one has
\begin{equation}\label{lt prod} L^*_U(A,\rho, 1) = Q_{\rho,0}(q^{-1})^{-1}Q_{\rho,2}(q^{-1})^{-1}\cdot \lim_{s \rightarrow 1} (s-1)^{-m}Q_{\rho,1}(q^{-s}).\end{equation}

To prove that this quantity is a strictly positive real number we shall split it into a number of subproducts and show that each separate subproduct is a strictly positive real number. 
%To do this we abbreviate each term $\alpha_{ia}\beta_{\rho b}$ to $\alpha_{iab}$.

At the outset we note that if an eigenvalue $\alpha_{ia}$ is not real, then (since $Q_{\rho,i}(u)$ belongs to $\mathbb{R}[u]$) there must exist an index $a\not= a'$ such that $\alpha_{ia'} = \tau(\alpha_{ia})$ and then  the product $(1- \alpha_{ia}q^{-1})(1- \alpha_{ia'}q^{-1})$ is a strictly positive real number.

We need therefore only consider eigenvalues $\alpha_{ia}$ that are real and to do this we define for each $i\in \{0,1,2\}$ sets of indices
\[ J'_i := \{ a: 1 \leq a \leq d_{\rho,i} \,\,\text{ with }\, \alpha_{ia} = q^{(i+1)/2} \}
 \subset J_i := \{a: 1 \leq a \leq d_{\rho,i} \text{ with } \alpha_{ia} \in \mathbb{R} \}.
 \]
Now if either $i = 0$ and $a \in J_0$ or if $1 \leq i \leq 2$ and $a \in J_i \setminus J'_i$, then one checks easily that
$(1- \alpha_{ia} q^{-1}) > 0$.

Furthermore, one has $m = |J'_1|$ and
\[
\lim_{s \rightarrow 1} (s-1)^{-m} {\prod}_{a \in J'_1} (1- \alpha_{1a}q^{-s}) = (\lim_{s \rightarrow 1} (s-1)^{-1}(1- q^{1-s}))^m = (\log(q))^m > 0
\]
is a strictly positive real number.

To prove the quantity in (\ref{lt prod}) is strictly positive we are therefore reduced to showing that the product 
\[ {\prod}_{a \in J'_2} (1-\alpha_{2a}q^{-1}) = {\prod}_{a \in J'_2} (1-q^{1/2})\]
is strictly positive, or equivalently that $|J'_2|$ is even.

To do this we set $\Delta := \Gal(L\mathbb{F}^c_q/ K\mathbb{F}^c_q)$ and recall that 
$H^2_{\et,c}(U^c, V_{\rho}\otimes_\QQ V_{\ell}(A))$ is dual to the $(1)$-twist of the space %
\begin{align*} 
H^0_{\et}(U^c, V_{\check\rho}\otimes_\QQ V_\ell(A^t)) & =\, 
\bigl(V_{\check\rho}\otimes_\QQ V_\ell(A^t)\bigr)^{\Gal(K^c/K\mathbb{F}^c_q)} \\ 
& =\, \bigl(V_{\check\rho}\otimes_\QQ V_{\ell}(A^t)^{\Gal(K^c/L\mathbb{F}^c_q)}\bigr)^\Delta \\
& \cong\, \bigl(V_{\check\rho}\otimes_\QQ V_{\ell}(B)\bigr)^\Delta \\
& =\, V_{\check\rho}^{\Delta}\otimes_\QQ V_{\ell}(B).
\end{align*}

Here the first equality is obvious and the second is true because the restriction of $\rho$ to $\Gal(K^c/L\mathbb{F}^c_q)$ is trivial, $B$ is the $L/\mathbb{F}_{q'}$ trace of $A^t$ (see \cite[Chap. VIII, \S 3, Th. 6]{lang} and note that $L/ \mathbb{F}_{q'}$ is primary i.e. the algebraic closure of $\mathbb{F}_{q'}$ in $L$ is purely inseparable extension of $\mathbb{F}_{q'}$) and the last equality is true because $B$ is defined over $\mathbb{F}_{q'}$.

In particular, if the representation $\check\rho^{\Delta}$ vanishes, then $|J'_2|= d_2 = 0$ is even and we are done.

We claim now that $\check\rho^{\Delta}$ does indeed vanish unless $\rho$ is trivial. To show this we note that $\Delta$ identifies with a normal subgroup of $G$ in such a way that the quotient is isomorphic to the cyclic group $H :=\Gal(\mathbb{F}_{q'}/\mathbb{F}_q)$.

Thus, if $\eta$ is any irreducible subrepresentation of ${\rm res}^G_\Delta(\check\rho)$, then Clifford's theorem (cf. \cite[Th. 11.1(i)]{curtisr}) implies that ${\rm res}^G_\Delta(\check\rho)$ is the direct sum of conjugates of $\eta$ and hence that ${\rm res}^G_\Delta(\check\rho)^\Delta$ does not vanish if and only if $\eta$ is trivial.

It follows that ${\rm res}^G_\Delta(\check\rho)^\Delta$ does not vanish if and only if ${\rm res}^G_\Delta(\check\rho)$ is trivial and this happens if and only if $\check\rho$, and hence also $\rho$ itself, is inflated from a representation of $H$.

Hence, since we have assumed that $\rho$ is both irreducible and not inflated from a non-trivial representation of $H$, the representation ${\rm res}^G_\Delta(\check\rho)^\Delta$ does not vanish if and only if $\rho$ is the trivial representation of $G$.

We have now verified the assertion of claim (ii) for all but the trivial representation of $G$ and in this case the claimed result is proved by Kato and Trihan in \cite[Appendix]{KT}. \end{proof}

\begin{remark}\label{contra just} {\em Fix a prime $\ell$ and an $\ell$-adic representation $V$ of $G_K$. Then the tensor product $\QQ_\ell^c[G]\otimes_{\QQ_\ell} V$ is a (left) module over $G\times G_K$ via the rule $(g,\sigma)(x\otimes v) := gx\overline{\sigma}^{-1}\otimes \sigma(v)$ for $g \in G, \sigma\in G_K, x \in \QQ_\ell^c[G]$ and $v \in V$, where $\overline{\sigma}$ is the image of $\sigma$ under the restriction map $G_K \to G$. In particular, if we fix $\chi$ in ${\rm Ir}(G)$ and a realisation $V_\chi$ over $\QQ_\ell^c$, then, with respect to this action, there is a canonical isomorphism of $\ell$-adic representations of $G_K$
\begin{equation*}\label{contra justification} V_\chi\otimes_{\QQ_\ell} V \cong \Hom_{\QQ_\ell^c[G]}(V_{\check\chi}, \QQ_\ell^c[G]\otimes_{\QQ_\ell}V),\end{equation*}
where $G_K$ acts diagonally on $V_\chi\otimes_{\QQ_\ell} V$ and on the Hom-group via only $\QQ_\ell^c[G]\otimes_{\QQ_\ell}V$. This isomorphism is induced by the canonical composite identification 
\begin{align*} H^0(G,\Hom_{\QQ_\ell^c}(V_{\check\chi}, \QQ_\ell^c[G]\otimes_{\QQ_\ell}V))\cong&\, H^0(G,\Hom_{\QQ_\ell^c}(V_{\check\chi},\QQ_\ell^c)\otimes_{\QQ_\ell^c}(\QQ_\ell^c[G]\otimes_{\QQ_\ell}V))\\
\cong & H_0(G,\Hom_{\QQ_\ell^c}(V_{\check\chi},\QQ_\ell^c)\otimes_{\QQ_\ell^c}(\QQ_\ell^c[G]\otimes_{\QQ_\ell}V)) \\
\cong&\, \Hom_{\QQ_\ell^c}(V_{\check\chi},\QQ_\ell^c)\otimes_{\QQ_\ell^c[G]}(\QQ_\ell^c[G]\otimes_{\QQ_\ell}V) \\
\cong&\, \Hom_{\QQ_\ell^c}(V_{\check\chi},\QQ_\ell^c)\otimes_{\QQ_\ell}V.\end{align*}
Here the second isomorphism is induced by the inverse of the canonical norm map (since the order of $G$ is invertible in $\QQ_\ell^c$), and all other isomorphisms are clear. }
\end{remark}

\section{Arithmetic complexes}\label{sect: arith}

In this section we construct certain canonical complexes of Galois modules whose Euler characteristics will occur in the formulation of our refined Birch and Swinnerton-Dyer conjecture.

In the sequel, for any noetherian ring $R$ we shall write $D^{\rm perf}(R)$ for the full triangulated subcategory of the derived category $D(R)$ of (left) $R$-modules comprising complexes that are `perfect' (that is, isomorphic in $D(R)$ to a bounded complex of finitely generated projective $R$-modules).

\subsection{Selmer groups} The Tate-Shafarevich group and, for any natural number $n$, the $n$-torsion Selmer group of $A$ over any finite extension $F$ of $K$ are respectively defined to be the kernels
\[
\sha(A/F) := \ker(H^1(F, A) \rightarrow {\bigoplus}_{v} H^1(F_v, A))
\]
and
\[
{\rm Sel}_n(A/F) := \ker(H^1(F, A[n]) \rightarrow {\bigoplus}_v H^1(F_v, A)).
\]
Here the groups $H^1(F, A), H^1(F, A[n])$ and $H^1(F_v, A)$ denote flat cohomology and in both cases $v$ runs over all places of $F$ and the arrow denotes the natural diagonal restriction map.

One then defines Selmer groups of $A$ over $F$ via the natural limits
\[
{\rm Sel}_{\Q/\Z}(A/F) := \varinjlim_n {\rm Sel}_{n}(A/F) \,\,\text{ and }\,\, {\rm Sel}_{\hat{\Z}}(A/F) := \varprojlim_n {\rm Sel}_{n}(A/F)
\]
and, for convenience, we write $X(A/F)$ for the Pontryagin dual of ${\rm Sel}_{\Q/\Z}(A/F)$.

\begin{remark}\label{eq:sel1}{\em We make much use in the sequel of the fact that the above definitions lead naturally to canonical exact sequences
\[ 0 \rightarrow A(F) \otimes_{\Z} \hat{\Z} \rightarrow {\rm Sel}_{\hat{\Z}}(A/F) \rightarrow \varprojlim_{n} \sha(A/F)[n] \rightarrow 0
\]
and
\[ 0 \rightarrow (\sha(A/F)_{\tors})^{\vee} \rightarrow X(A/F) \rightarrow \Hom_{\Z}(A(F), \hat{\Z}) \rightarrow 0.\]}
\end{remark}

\subsection{Arithmetic cohomology} For each place $v$ of $F$ outside $U_F$ we fix a pro-$p$ open subgroup $V_v$ of $A(F_v)$ and denote the family $(V_v)_{v \notin U_F}$ by $V_{U_F}$, or more simply by either $V_F$ or $V$ when the context is clear.

We then follow Kato and Trihan \cite{KT} in defining the `arithmetic cohomology' complex $R\Gamma_{ar,V}(U_F, \calA_{\tors})$ to be the mapping fibre of the morphism
\begin{equation}\label{fibre morphism}
R\Gamma_{\rm fl}(U_F, \mathcal{A}_{\tors}) \oplus ({\bigoplus}_{v \notin U_F} V_v \otimes_\Z^{\mathbb{L}} \Q/\Z)[-1]) \xrightarrow{(\kappa_1, \kappa_2)} {\bigoplus}_{v \notin U_F} R\Gamma_{\rm fl}(F_v, \calA_{\tors}).
\end{equation}
Here $\kappa_1$ denotes the natural diagonal localisation morphism in flat cohomology
and $\kappa_2$ the restriction of the morphism
\[
(A(F_v) \otimes_\Z^{\mathbb{L}} \Q/\Z)[-1] \rightarrow {\bigoplus}_{v \notin U_F} R\Gamma_{\rm fl}(F_v, \calA_{\tors})
\]
that is obtained by applying $- \otimes_\Z^{\mathbb{L}}\Q/\Z$ to the morphism $A(F_v)[-1] \rightarrow R\Gamma_{\rm fl}(F_v, \varprojlim_n A[n])$ in $D(\ZZ[G])$ induced by the fact that $H^0_{\rm fl}(F_v, \varprojlim_n A[n])$ vanishes whilst $A(F_v)$ is canonically isomorphic to a submodule of $H^1_{\rm fl}(F_v, \varprojlim_n A[n])$.

\begin{proposition} \label{prop:arithcoho} The complex $C^{\rm ar}_V := R\Gamma_{ar,V}(U_F, \calA_{\tors})$ is acyclic outside degrees $0, 1$ and $2$. In addition, there exists a canonical exact sequence
\[
0 \rightarrow H^0(C^{\rm ar}_V) \rightarrow A(F)_{\tors} \rightarrow {\bigoplus}_{v \notin U_F} A(F_v)/V_v \rightarrow H^1(C^{\rm ar}_V) \rightarrow {\rm Sel}_{\Q/\Z}(A/F) \rightarrow 0,
\]
and a canonical isomorphism $H^2(C^{\rm ar}_V) \cong \Hom_{\Z}({\rm Sel}_{\hat{\Z}}(A^t/F), \Q/\Z)$.
\end{proposition}

\begin{proof} This is proved in \cite[\S2.5]{KT}.\end{proof}

Since $R\Gamma_{ar,V}(U_F, \calA_{\tors})$ is a complex of torsion groups it decomposes naturally as a direct sum of $\ell$-primary component complexes $R\Gamma_{ar,V}(U_F, \calA_{\tors})_\ell$.

\begin{remark}\label{flat-etale}{\em For any prime $\ell\not= p$ the definition of $R\Gamma_{ar,V}(U_F, \calA_{\tors})_\ell$ via the morphism in (\ref{fibre morphism}) implies that it identifies with the compactly supported \'etale cohomology complex $R\Gamma_{\et,c}(U_F, \mathcal{A}\{\ell\})$ of the (\'etale) sheaf $\mathcal{A}\{\ell\}$ on $U_F$ comprising all  $\ell$-primary torsion in $\mathcal{A}$.}\end{remark}

\subsection{Pro-$p$ subgroups} To make the complex $R\Gamma_{ar,V}(U_F, \calA_{\tors})$ constructed above amenable for our purposes we need to make an appropriate choice of the family $V$. We now explain how to make such a choice following the approach of Kato and Trihan in  \cite[\S~6]{KT}.

To do this we fix a finite Galois extension $L/K$ and set $G:= \Gal(L/K)$. We let $X_L$ be the proper smooth curve with function field $L$, and let $U_L\subset X_L$ be the preimage of $U$ (and we will later `shrink' $U$ so that $L/K$ is unramified at places in $U$). For any place $w$ of $L$ we write $G_w$ for its decomposition subgroup in $G$.

We write $\cala'$ for the N\'eron model of $A_L$ over $X_L$, and $\cala_{X_L}$ for the pull back of $\cala$.

\begin{lemma}\label{lem:KT6.4} There exists a $G$-invariant divisor $E={\sum}_{w\notin U_L} n(w)w$ on $X_L$ with $\mathrm{supp}(E)=X_L\setminus U_L$ and for each place  $w\notin U_L$ over $v\notin U$ a $G_w$-stable pro-$p$ open subgroup $V'_w$ of $A_L(L_w)$ and an open $\mathcal{O}_v[G_w]$-submodule $W'_w$ of ${\rm Lie}(A_L(L_w))$  that satisfy all of the following properties.
\begin{enumerate}
\item For $w\notin U_K$, we have $\cala'(\mathfrak{m}_w^{2n(w)})\subset V'_w\subset \cala_{X_L}(\mathfrak{m}_w^{n(w)})$.
\item\label{lem:KT6.4:W} For $w\notin U_L$, we have ${\rm Lie}(\cala')(\mathfrak{m}_w^{2n(w)})\subset W'_w\subset
{\rm Lie}(\cala_{X_L})(\mathfrak{m}_w^{n(w)})$.
\item\label{lem:KT6.4:exp} For $w\notin U_L$, the canonical isomorphism
\[
\cala'(\mathfrak{m}_w^{n(w)})/\cala'(\mathfrak{m}_w^{2n(w)}) \cong
{\rm Lie}(\cala')(\mathfrak{m}_w^{n(w)})/{\rm Lie}(\cala')(\mathfrak{m}_w^{2n(w)})
\]
sends the image of $V_w'$ to $W_w'$.
\item\label{lem:KT6.4:CohTriv}  For each place $v$ outside $U$ the products ${\prod}_{w|v} V'_w$ and ${\prod}_{w|v} W'_w$ are stable under the action of $G$ and for each natural number $i$ the associated cohomology groups $H^i(G, {\prod}_{w|v}V_w')$ and $H^i(G, {\prod}_{w|v}W_w')$ vanish.
\end{enumerate}
We can furthermore require $E$ to be the pull back of some divisor $E_0$ of $X$.
\end{lemma}
In the application (\emph{cf.} \S~\ref{sect:tame}\emph{ff}) we need $E$ to be the pull back of a divisor $E_0$ of $X$.
\begin{proof} This result is only a slight adaptation of \cite[Lem. ~6.4]{KT} (see Remark \ref{KT issue} below). For this reason we only sketch the proof, following closely the argument of \cite[\S6]{KT}.

The key point is that it suffices to construct a divisor $E$ and a family of subgroups $\{W'_w\}_w$ with the properties stated in Lemma~\ref{lem:KT6.4}, since the family $\{W'_w\}_w$ uniquely determines the family $\{V'_v\}_v$ by property (\ref{lem:KT6.4:exp}) and then the latter family can be shown to satisfy property (\ref{lem:KT6.4:CohTriv}) by repeating the proof of \cite[Lem. ~6.2(2)]{KT}.

Now, by the argument of \cite[Lem. ~6.2(1)]{KT}, for each place $w$ of $L$ outside $U_L$ there exists a constant $c(w)$ such that for any integer $n\geqslant0$ there exists a $G_w$-stable $\mathcal{O}_v$-lattice $W'_w$ of ${\rm Lie}(\cala_{X_L})(\mathcal{O}_w)$ such that both
\[
{\rm Lie}(\cala_{X_L})(\mathfrak{m}_w^{n+c(w)})\subset W'_w \subset {\rm Lie}(\cala_{X_L})(\mathfrak{m}_w^{n}). 
\]
and the group $H^i(G_w,W'_w)$ vanishes for all $i\geqslant1$.

By the argument of \cite[Lem. ~6.3]{KT}, we may in addition assume that the subgroups $W'_w$ satisfy property (\ref{lem:KT6.4:W}), at least provided that $n(w)$ is sufficiently large and divisible by the ramification index $e(w|v)$. (We would like $n(w)$ to be divisible by $e(w|v)$ in general since we are only allowed to multiply $W'_w$ by $\mathcal{O}_v$-multiple; recall that $W'_w$ is only an $\mathcal{O}_v$-submodule, not an $\mathcal{O}_w$-submodule.)

To ensure that the product ${\prod}_{w\notin U_L} W'_w$ is stable under the action of $G$, we first fix a place $w$ over each $v\notin U$ and a subgroup  $W'_w$ that has property (\ref{lem:KT6.4:W}) and is such that $H^i(G_w,W'_w)$ vanishes for all $i\geqslant1$.

For each $\sigma$ in $G$, we then set $W'_{\sigma(w)}:= \sigma(W'_w) \subset {\rm Lie}(\cala_{X_L})(\mathcal{O}_{\sigma(w)})$ (which, we note, only depends on $\sigma(w)$). Then the collection of subgroups $\{W'_w\}_{w\notin U_L}$ clearly has both of the properties (\ref{lem:KT6.4:W}) and (\ref{lem:KT6.4:CohTriv}). 

To ensure that $E$ is a pull back of some divisor $E_0$ of $X$, we may replace $E$ with $\pi^*(\pi_*E)$ and replace $\{W'_w\}_{w|v}$ by some suitable power of uniformiser of $\mathcal{O}_w$.
\end{proof}

\begin{remark}\label{KT issue}{\em Lemma \ref{lem:KT6.4} only differs from \cite[Lem. ~6.4]{KT} in that we require each group $W'_w$ to be an open $\mathcal{O}_v[G_w]$-submodule of ${\rm Lie}(\cala')(\mathcal{O}_w)$ rather than an open $\mathcal{O}_w$-submodule as in loc. cit. In fact, in \cite[Lem.~6.2(1)]{KT}, it is claimed that $W'_w$ can be chosen as an $\mathcal{O}_w$-sublattice of ${\rm Lie}(\cala_{X_L})(\mathcal{O}_w)$, but the indicated proof seems only to show that it can be chosen as a $G_w$-stable $\mathcal{O}_v$-lattice.}\end{remark}

\begin{remark}\label{good unramified} {\em The proof of Lemma \ref{lem:KT6.4} shows that for any place $v$ of $K$ that is both unramified in $L$ and of  good reduction for $A$, the subgroup $V'_w$ can be chosen as $\mathcal{A}(\mathfrak{m}_w)$.}\end{remark}

\subsection{Selmer complexes}\label{selmer section} For each place $w$ outside $U_L$ we now fix a choice of subgroups $V'_w$ as in Lemma \ref{lem:KT6.4}. For any subgroup $J$ of $G$ and for any place $v$ outside $U_{L^J}$ we then define a group
\[ V_v :=\big({\prod}_{w|v}V'_w\big)^J\]
and we denote the associated families of subgroups $(V'_w)_{w \notin U_L}$ and $(V_v)_{v \notin U_{L^J}}$ by $V_L$ and $V_{L^J}$, respectively. We may occasionally write $V$ for $V_K$ when $J=G$.

In the following result we use these subgroups to construct a canonical `Selmer complex' ${\rm SC}_{V_L}(A, L/K)$ that will play a key role in the formulation of our refined Birch and Swinnerton-Dyer conjecture.

We also use the $G$-module $X_\ZZ(A/F)$ that is defined as the pre-image of $\Hom_\ZZ(A(F),\ZZ)$ under the natural surjective homomorphism $X(A/F) \to \Hom_\ZZ(A(F),\hat \ZZ)$ (see Remark \ref{eq:sel1}).

\begin{proposition}\label{prop:perfect} The following claims are valid.
\begin{itemize}
\item[(i)] $R\Gamma_{ar,V_L}(U_L, \calA_{\tors}')^*[-2]$ is an object of $D^{\rm perf}(\hat{\Z}[G])$ that is acyclic outside degrees $0, 1$ and $2$.
\item[(ii)] If the groups $\sha(A/L)$ and $\sha(A^t/L)$ are both finite, then there exists a complex ${\rm SC}_{V_L} := {\rm SC}_{V_L}(A, L/K)$ in $D^{\rm perf}(\Z[G])$ that is acyclic outside degrees $0, 1$ and $2$, is unique up to isomorphisms in $D^{\rm perf}(\Z[G])$ that induce the identity map in all degrees of cohomology and also has both of the following properties:
\begin{itemize}
\item[(a)] One has $H^0({\rm SC}_{V_L}) = A^t(L)$, $H^1({\rm SC}_{V_L})$ contains $X_\ZZ(A/L)$ as a submodule of finite index and $H^2({\rm SC}_{V_L})$ is finite.
\item[(b)] There exists a canonical isomorphism in $D^{\rm perf}(\hat{\Z}[G])$ of the form 
\[ \hat{\Z} \otimes_{\Z} {\rm SC}_{V_L}
\cong R\Gamma_{ar,V_L}(U_L, \calA_{\tors})^*[-2].\]
\end{itemize}
\end{itemize}
\label{cor:arithperf}
\end{proposition}

\begin{proof} For each subgroup $J$ of $G$ we set $C^{\rm ar}_{V,J} := R\Gamma_{ar,V_{L^J}}(U_{L^J}, \calA_{\tors})$ and we abbreviate $C^{\rm ar}_{V,J}$ to $C^{\rm ar}_{V}$ when $J$ is the trivial subgroup.

Then, since $H^i(C^{{\rm ar},*}_{V,J}[-2]) = H^{2-i}(C^{\rm ar}_{V,J})^*$ in all degrees $i$, the result of Proposition \ref{prop:arithcoho} implies that  each complex $C^{{\rm ar},*}_{V,J}[-2]$ is acyclic in all degrees outside $0, 1$ and $2$ and that its cohomology is finitely generated over $\hat \Z$ in all degrees.

By a standard criterion, it follows that $C^{{\rm ar},*}_V$, and hence also $C^{{\rm ar},*}_V[-2]$, belongs to $D^{\rm perf}(\hat\Z[G])$, and so claim (i) is valid, if for every subgroup $J$ of $G$ there is an isomorphism in $D(\hat\Z)$ of the form $\,\Z \otimes_{\Z[J]}^\mathbb{L} C^{{\rm ar},*}_V \cong C^{{\rm ar},*}_{V,J}$.

In view of the natural isomorphisms $\Z \otimes_{\Z[J]}^\mathbb{L} C^{{\rm ar},*}_V \cong R\Hom_{\ZZ[J]}(\ZZ,C^{{\rm ar}}_V)^*$ we are therefore reduced to showing the existence of isomorphisms in $D(\hat\Z)$ of the form
\begin{equation}\label{KT descent} R\Hom_{\ZZ[J]}(\ZZ,R\Gamma_{{\rm ar},V_L}(U_{L}, \calA_{\tors})) \cong R\Gamma_{{\rm ar},V_{L^J}}(U_{L^J}, \calA_{\tors})\end{equation}
and this is proved by Kato and Trihan in \cite[Lem. 6.1]{KT}.

Turning to claim (ii), we note that claim (i) combines with the general result of Lemma \ref{prop:zzhatcom} below to imply it suffices to show that, under the stated hypotheses, the group $H^{0}(C^{{\rm ar}}_V)^*$ is finite, the group $H^{2}(C^{{\rm ar}}_V)^*$ is naturally isomorphic to $\hat\ZZ\otimes_\ZZ A^t(L)$ and there exists a finitely generated $G$-module $M^1$ that contains $X_\ZZ(A/L)$ as a submodule of finite index and is such that there is a canonical isomorphism $\hat\ZZ\otimes_\ZZ M^1 \cong H^1(C^{{\rm ar}}_V)^*$ of $\hat\Z[G]$-modules.

In this direction, the exact sequence in Proposition \ref{prop:arithcoho} implies directly that $H^0(C^{{\rm ar}}_V)^*$ is finite.

In addition, since the limit $\varprojlim_{n} \sha(A^t/L)[n]$ vanishes under the assumption $\sha(A^t/L)$ is finite, the displayed isomorphism in Proposition \ref{prop:arithcoho} combines with the first exact sequence in Remark \ref{eq:sel1} to give a canonical isomorphism
\[ H^2(C^{{\rm ar}}_V)^* \cong (\hat\ZZ\otimes_\ZZ A^t(L))^{**} = \hat\ZZ\otimes_\ZZ A^t(L)\]
of the required form.

Next we note that, since $\sha(A/L)$ is assumed to be finite, the second exact sequence in Remark \ref{eq:sel1} implies $X_\ZZ(A/L)$ is finitely generated.

Thus, if we write $Y$ for the (finite) cokernel of the map $A(L)_{\tors} \rightarrow {\bigoplus}_{v \notin U_L} A(L_v)/V_v$ that occurs in Proposition \ref{prop:arithcoho}, then the natural map of finite groups
\[ {\rm Ext}^1_G(Y^*,X_\ZZ(A/L)) = \hat\ZZ\otimes_\ZZ{\rm Ext}^1_G(Y^*,X_\ZZ(A/L)) \to {\rm Ext}^1_{\hat\Z[G]}(Y^*,X(A/L))\]
is bijective and so there exists an exact commutative diagram of $G$-modules
 \[ \begin{CD}
 0 @> >> X_\ZZ(A/L) @> >> M @> >> Y^*@> >> 0\\
 @. @V VV @V VV @\vert\\
 0 @> >> X(A/L) @> >> H^{1}(C^{{\rm ar}}_V)^* @> >> Y^*@> >> 0\end{CD}\]
in which the first vertical arrow is the natural inclusion, and so induces an isomorphism $\hat\ZZ\otimes_\ZZ X_\ZZ (A/L) \cong X(A/L)$, and the lower row is induced by the Pontryagin dual of the long exact sequence in Proposition \ref{prop:arithcoho}.

In particular, from the upper row of the above diagram we can deduce that $M$ is finitely generated and hence is a suitable choice for the module $M^1$ that we seek.
\end{proof}

In the sequel, for a ring $\Lambda$ and integer $a$, we write $\tau_{\ge a}$ and $\tau_{\le a}$ for the truncation functors on $D(\Lambda)$ in degrees at least $a$ and at most $a$ respectively. 

We also recall that a $G$-module is said to be `cohomologically-trivial', or `c-t' for short in the sequel, if for every integer $i$ and every subgroup $J$ of $G$ the Tate cohomology group $\hat H^i(J,M)$ vanishes.

\begin{lemma} \label{prop:zzhatcom} Let $\hat{C}$ be a cohomologically-bounded complex of $\hat{\Z}[G]$-modules and assume to be given in each degree $j$ a finitely generated $G$-module $M^j$ and an isomorphism of $\hat \Z[G]$-modules of the form $\iota_j: \hat \Z\otimes_\Z M^j \cong H^j(\hat C)$.

Then there exists an object $C$ of $D(\Z[G])$ with all of the following properties:
\begin{itemize}
\item[(i)] $H^j(C) = M^j$ for all $j$.
\item[(ii)] There exists an isomorphism $\iota: \hat\Z\otimes_\Z C \cong \hat C$ in $D(\hat{\Z}[G])$ for which in each degree $j$ one has $H^j(\iota) = \iota_j$.
\item[(iii)] $C$ belongs to $D^{\rm perf}(\Z[G])$ if and only if $\hat C$ belongs to $D^{\rm perf}(\hat{\Z}[G])$.
\end{itemize}
Any such complex $C$ is unique to within an isomorphism $\kappa$ in $D(\ZZ[G])$ for which $H^j(\kappa)$ is the identity automorphism of $M^j$ in each degree $j$.
\end{lemma}

\begin{proof} We prove this by induction on the number of non-zero cohomology groups of $\hat C$.

If there is only one non-zero such group, in degree $d$ say, then $\hat C$ is isomorphic in $D(\hat \Z[G])$ to $(\hat\Z\otimes_\Z M^d)[-d]$ and we write $C$ for the complex $M^d[-d]$ in $D(\ZZ[G])$.

In this case claim (i) is clear and claim (ii) is true with $\iota$ the morphism induced by $\iota_d$. Finally, since any finitely generated module over either $\ZZ[G]$ or $\hat \ZZ[G]$ that is c-t has a finite projective resolution, $C$ belongs to $D^{\rm perf}(\ZZ[G])$ if and only if $M^d$ is c-t and $\hat C$ belongs to $D^{\rm perf}(\hat\ZZ[G])$ if and only if $\hat M^d$ is c-t. This implies claim (iii) since a finitely generated $G$-module $N$ is c-t if and only if 
$\hat \Z\otimes_\Z N$ is c-t as a consequence of the fact that in each degree $i$ and for each subgroup $J$ of $G$ the natural map $\hat H^i(J,N) \to \hat H^i(J,\hat\Z\otimes_\Z N)$ is bijective.

To deal with the general case we assume $\hat C$ is not acyclic and write $d$ for the unique integer such that $H^d(\hat C) \not= 0$ and $H^i(\hat C) = 0$ for all $i > d$. We then abbreviate the complexes $\tau_{\le d-1}\hat C$ and $\tau_{\ge d}\hat C$ to $\hat C_1$ and $\hat C_2$ and recall that there is a canonical exact triangle in $D(\hat \Z[G])$ of the form 
\[ \hat C_1 \to \hat C \to \hat C_2 \xrightarrow{\hat\theta} \hat C_1[1].\]
We note that this triangle induces an isomorphism $\kappa$ in $D(\hat \Z[G])$ between $\hat C$ and ${\rm Cone}(\hat \theta)[-1]$, where we write ${\rm Cone}(\alpha)$ for the mapping cone of a morphism $\alpha$.

Now, since $H^j(\hat C_1) = H^j(\hat C)$ for $j < d$ and $H^j(\hat C_1)=0$ for all $j \ge d$, the inductive hypothesis implies the existence of $C_1$ in $D(\ZZ[G])$ and an isomorphism $\iota_1: \hat \Z\otimes_\Z C_1 \cong \hat C_1$ in $D(\hat \Z[G])$ such that in each degree $j$ with $j < d$ one has $H^j(\iota) = \iota_j$.

In addition, since $\hat C_2$ is acyclic outside degree $d$, the argument given above shows the existence of a complex $C_2$ in $D(\ZZ[G])$ and an isomorphism $\iota_2: \hat \Z\otimes_\Z C_2 \cong \hat C_2$ in $D(\hat \Z[G])$ with $H^d(\iota) = \iota_d$.

Next we recall that the group $\Hom_{D(\hat \Z[G])}(\hat C_2,\hat C_1[1])$ is equal to $H^0(R\Hom_{\hat \Z [G]}(\hat C_2,\hat C_1[1]))$ and so can be computed by using the spectral sequence
\[ E_2^{p,q} = {\prod}_{a\in\ZZ}
{\rm Ext}^p_{G}(H^a(\hat C_2),H^{q+a}(\hat C_1[1])) \Rightarrow H^{p+q}(R\Hom _{\hat \Z[G]}(C_2,C_1[1]))\]
constructed by Verdier in \cite[III, 4.6.10]{verdier}. We also note that there is no degree in which the complexes $\hat C_2$ and $\hat C_1[1]$ have cohomology groups that are both non-zero and that any group of the form ${\rm Ext}^p_{G}(-,-)$ vanishes for $p < 0$ and is torsion for $p>0$. Given these facts, the above spectral sequence implies that $\Hom_{\!D(\hat \Z[G])}(\hat C_2,\hat C_1[1])$ is finite and hence that the diagonal localisation map  $\Hom_{\!D(\Z[G])}(C_2,C_1[1]) \to \Hom_{\!D(\hat \Z[G])}(\hat C_2,\hat C_1[1])$ is bijective.

We now write $\theta $ for the pre-image of $\hat \theta$ under the latter isomorphism and claim that the mapping fibre $C := {\rm Cone}(\theta)[-1]$ has all of the required properties.

Firstly, this definition implies directly that $H^j(C)$ is equal to $H^j(C_1)$ if $j < d$ and to $H^j(C_2)$ if $j \ge d$, and so claim (i) follows immediately from the given properties of $C_1$ and $C_2$. The definition also implies directly that $\hat \Z\otimes_\Z C$ is isomorphic to ${\rm Cone}(\hat\theta)[-1]$ and hence that $\kappa$ induces an isomorphism in $D(\hat \Z[G])$ between $\hat \Z\otimes_\Z C$ and $\hat C$ with the property described in claim (ii).

To prove claim (iii) it suffices to check that $C$ belongs to $D^{\rm perf}(\Z[G])$ if $\hat C$ belongs to $D^{\rm perf}(\hat \Z[G])$. To do this we can assume, by a standard resolution argument (as described, for example, in \cite[Rapport, Lem. 4.7]{del}), that $C$ is a bounded complex of finitely generated $G$-modules in which all but the first (non-zero) module, $M$ say, is free. If we then also assume that the complex $\hat C$ is isomorphic in $D(\hat \Z[G])$ to a bounded complex of finitely generated projective $\hat \Z[G]$-modules $Q$, then there exists a quasi-isomorphism $\pi: Q \to \hat\Z \otimes_\Z C$ of complexes of $\hat \Z[G]$-modules.

Now, since all terms of $Q$ and $\hat\Z \otimes_\Z C$ are projective $\hat \Z[G]$-modules, except possibly for $\hat \Z\otimes_\Z M$, the acyclicity of ${\rm Cone}(\pi)$ implies that the $\hat \Z[G]$-module $\hat \Z\otimes_\Z M$ is c-t. This in turn implies that $M$ is c-t and hence has a finite projective resolution. It follows that $C$ belongs to $D^{\rm perf}(\Z[G])$, as claimed. \end{proof}

\subsection{Coherent cohomology}\label{zariski section} The Selmer complex that is constructed in Proposition \ref{prop:perfect}(ii) depends on the choice of subgroups $V_L$. We shall therefore need to introduce an auxiliary perfect complex that will be used to compensate for this dependence in the formulation of our conjecture.

To do this for each place $v$ outside $U$ we choose a place $w$ of $X_L$ above $v$ and the $\mathcal{O}_v[G_w]$-submodule $W'_w$ of ${\rm Lie}(A_L(L_w))$ that corresponds in Lemma \ref{lem:KT6.4} to the subgroup $V'_w$ fixed at the beginning of \S\ref{selmer section}. For any other place $w'|v$ of $X_L$, we choose $\gamma\in G$ such that $w'=\gamma\cdot w$ and set $W'_{w'}$ denote the image of the isomorphism ${\rm Lie}(A_L(L_w))\xrightarrow{\cong}{\rm Lie}(A_L(L_{w'}))$ induced by $\gamma$. Note that $W'_{w'}$ does not depend on the choice of $\gamma$ as $W'_w$ is $G_w$-stable.

For any place $v$ outside $U$ we then set
\[ W_v:=\big({\prod}_{w|v}W'_w\big)^G \]
and we denote the associated families of subgroups $(W'_w)_{w \notin U_L}$ and $(W_v)_{v \notin U}$ by $W_L$ and $W_K$ respectively.

We then define $\mathcal{L}$ to be the coherent $\mathcal{O}_X$-submodule of ${\rm Lie}(\cala)$ that extends ${\rm Lie}(\cala|_U)$ and is such that $\mathcal{L}_v = W_v\subset {\rm Lie}(\cala)(\mathcal{O}_v)$ for each $v \notin U$.

We similarly define $\mathcal{L}_L$ to be the $G$-equivariant coherent $\mathcal{O}_X$-submodule of $\pi_*{\rm Lie}(\cala_{X_L})$ with $\mathcal{L}_{L,v} = {\prod}_{w|v}W'_w$ for each $v \notin U$, where we write $\pi:X_L\to X$ for the natural projection. 

\begin{lemma}\label{zar lemma} The complex $R\Gamma(X, \mathcal{L}_L)^*$ belongs to $D^{\rm perf}(\mathbb{F}_p[G])$, and hence to $D^{\rm perf}(\ZZ_p[G])$.
\end{lemma}

\begin{proof} For each subgroup $J$ of $G$ the complex $R\Gamma(X, (\mathcal{L}_{L})^J)^*$ is represented by a complex of finite-dimensional $\mathbb{F}_p$-vector spaces that is acyclic outside degrees $0$ and $1$.

By the same argument as used to prove Proposition \ref{prop:perfect}(i) we are therefore reduced to proving that for each $J$ there is a natural isomorphism in $D(\mathbb{F}_p)$ of the form
\begin{equation}\label{second KT descent}
R\Hom_{\mathbb{F}_p[J]}(\mathbb{F}_p,R\Gamma(X, \mathcal{L}_L)) \cong R\Gamma(X, (\mathcal{L}_L)^J)\end{equation}
and this is proved by Kato and Trihan in \cite[p. 585]{KT}. \end{proof}

\begin{remark}{\em In view of Remark \ref{KT issue}, we have here defined $\mathcal{L}_L$ to be a $\mathbb{F}_p[G]$-equivariant vector bundle over $X$ rather than a vector bundle over $X_L$, as in \cite[\S~6.5]{KT}. This means that various arguments in loc. cit. that rely on the `geometric $p$-adic cohomology theory' over $X_L$ and will be referred to in later sections must in our case be carried out over $X$ by using the relevant push-forward constructions. This, however, is a routine difference that we do not dwell on.} 
 \end{remark}

\section{Statements of the conjecture and main results}\label{sect: statements}

In this section we formulate our refinement of the Birch and Swinnerton-Dyer Conjecture, establish some basic properties of the conjecture and state the main supporting evidence for it that we will obtain in the rest of the article.

\subsection{Relative $K$-theory} Before stating our conjecture we quickly review relevant aspects of relative algebraic $K$-theory.

For a Dedekind domain $R$ with field of fractions $F$, an $R$-order $\mathfrak{A}$ in a finite dimensional separable $F$-algebra $A$ and a field extension $E$ of $F$ we set $A_E := E\otimes_F A$. 
 
\subsubsection{}We use the relative algebraic $K_0$-group $K_0(\mathfrak{A},A_E)$ of the ring inclusion $\mathfrak{A}\subset A_E$, as described explicitly in terms of generators and relations by Swan in \cite[p. 215]{swan}.

We recall that for any extension field $E'$ of $E$ there exists an exact commutative diagram
\begin{equation} \label{E:kcomm}
\begin{CD} K_1(\mathfrak{A}) @> >> K_1(A_{E'}) @> \partial_{\mathfrak{A},E'} >> K_0(\mathfrak{A},A_{E'}) @> \partial'_{\mathfrak{A},E'} >> K_0(\mathfrak{A})\\
@\vert @A\iota AA @A\iota' AA @\vert\\
K_1(\mathfrak{A}) @> >> K_1(A_E) @> \partial_{\mathfrak{A},E}  >> K_0(\mathfrak{A},A_E) @> \partial'_{\mathfrak{A},E}  >> K_0(\mathfrak{A})
\end{CD}
\end{equation}
in which the upper and lower rows are the respective long exact sequences in relative $K$-theory of the inclusions $\mathfrak{A}\subset A_E$ and $\mathfrak{A}\subset A_{E'}$ and both of the vertical arrows are injective and induced by the inclusion $A_E \subseteq A_{E'}$. (For more details see \cite[Th. 15.5]{swan}.)

We further recall that the Whitehead group $K_1(A_E)$ comprises (isomorphism classes of) pairs of the form $\langle W,\theta \rangle$ in which $\theta$ is an automorphism of the finitely generated projective $A_E$-module $W$. In particular, if $W$ is spanned by a (finitely generated) projective $\mathfrak{A}$-module $P$, then the connecting homomorphism $\partial_{\mathfrak{A},E}$ in (\ref{E:kcomm}) sends $\langle W,\theta\rangle$ to the element of $K_0(\mathfrak{A},A_E)$ that corresponds to the triple $(P,\theta,P)$. 

If $R = \ZZ$ and for each prime $\ell$ we set $\mathfrak{A}_\ell := \ZZ_\ell\otimes_\ZZ \mathfrak{A}$ and $A_\ell:=
\QQ_\ell\otimes _\QQ A$, then we regard each group $K_0(\mathfrak{A}_\ell,A_\ell)$ as a subgroup of $K_0(\mathfrak{A},A)$ by means of the canonical composite homomorphism
\begin{equation}\label{decomp}
{\bigoplus}_\ell K_0(\mathfrak{A}_\ell,A_\ell) \cong K_0(\mathfrak{A},A)\subset K_0(\mathfrak{A},A_\RR).
\end{equation}
Here $\ell$ runs over all primes, the isomorphism is as described in the discussion following \cite[(49.12)]{curtisr} and the inclusion is induced by the relevant case of the map $\iota'$ in (\ref{E:kcomm}).

For each element $x$ of $K_0(\mathfrak{A},A)$ we write $(x_\ell)_\ell$ for its image in ${\bigoplus}_\ell K_0(\mathfrak{A}_\ell,A_\ell)$ under the isomorphism in (\ref{decomp}).

\subsubsection{}\label{na dets section} We shall construct elements of $K_0(\mathfrak{A},A_E)$ by using the formalism of `non-abelian determinants' described by Fukaya and Kato in \cite[\S1]{FukayaKato:2006}. To recall the relevant facts we write $\Sigma$ for the category  $D^{\rm perf}(\mathfrak{A})$. 

Following \cite[Def. 1.3.2]{FukayaKato:2006}, one can define a localized $K_1$-group $K_1(A_E,\Sigma)$. This abelian group is generated by pairs $(C,h)$, where $C$ is an object of $\Sigma$ for which the Euler characteristic of $E\otimes_RC$ in $K_0(A_E)$ vanishes and $h$ is a morphism ${\rm Det}_{A_E}(E\otimes_R C) \to {\rm Det}_{A_E}(0)$ in the category $\mathcal{C}_{A_E}$ constructed in \cite[\S 1.2.1]{FukayaKato:2006}; the relations between these generators are then the obvious analogues of the relations (1), (2) and (3) given as part of \cite[Def. 1.3.2]{FukayaKato:2006}. These relations in turn ensure that there exists a canonical group homomorphism  
\[ \iota_{\mathfrak{A},E} : K_1(A_E,\Sigma) \to K_0(\mathfrak{A},A_E).\]
The approach of \cite[Th. 1.3.15]{FukayaKato:2006} proves the existence of an exact sequence relating $K_1(A_E,\Sigma)$ to $K_1(A_E), K_0(\Sigma) = K_0(\mathfrak{A})$ and $K_0(A_E)$, and by comparing this sequence to (\ref{E:kcomm}), one can deduce that $\iota_{\mathcal{A},E}$ is surjective (but we omit details as we make no use of this fact). 

For each generator $(C,h)$ of $K_1(A_E,\Sigma)$, we set 
\[ \chi_{\mathfrak{A},E}(C,h) := \iota_{\mathfrak{A},E}((C,h))\in K_0(\mathfrak{A},A_E).\]
If  $E \otimes_R C$ is acyclic, then we further set 
\[ \chi_{\mathfrak{A},E}(C,0) := \chi_{\mathfrak{A},E}(C,h_{\rm can}),\]
with $h_{\rm can}$ the canonical morphism 
${\rm Det}_{A_E}(E\otimes_R C) \to {\rm Det}_{A_E}(0)$ in $\mathcal{C}_{A_E}$ (from \cite[\S1.2.8]{FukayaKato:2006}).

\begin{example}{\em Fix a bounded complex $C^\bullet$ of finitely generated projective $\mathfrak{A}$-modules, and set $C^{\rm even} := {\bigoplus}_{i\in \ZZ}C^{2i}$ and $C^{\rm odd} := {\bigoplus}_{i\in \ZZ}C^{2i+1}$. Then, in this case, specifying a morphism $h: {\rm Det}_{A_E}(E\otimes_R C) \to {\rm Det}_{A_E}(0)$ in $\mathcal{C}_{A_E}$ is equivalent to specifying data as follows: for some finitely generated projective $\mathfrak{A}$-module $P$, one is given an isomorphism of $A_E$-modules 
$\theta: E\otimes_R (C^{\rm even} \oplus P) \cong  E\otimes_R (C^{\rm odd} \oplus P)$ that is unique up to pre-composition with an automorphism of $E\otimes_R (C^{\rm even} \oplus P)$ whose image in $K_1(A_E)$ is specified (and so depends only on $h$). Then, in terms of the standard presentation of $K_0(\mathfrak{A},A_E)$,  
the element $\chi_{\mathfrak{A},E}(C^\bullet,h)$ corresponds to the triple $(C^{\rm even}\oplus P, \theta, C^{\rm odd}\oplus P)$, with the defining relations of $K_0(\mathfrak{A},A_E)$ ensuring that  this element is indeed independent of both $P$ and the specific choice of $\theta$} \end{example}

We next record some general properties of the elements $\chi_\mathfrak{A}(C,h)$ that will be used frequently in the sequel (often without explicit comment). 

Firstly, if 
 $C_1 \to C_2 \to C_3 \to C_1[1]$ is an exact triangle in $D^{\rm perf}(\mathfrak{A})$ for which the complex $F\otimes_R C_3$ is acyclic, then each morphism $h : {\rm Det}_{A_E}(E\otimes_R C_1) \to {\rm Det}_{A_E}(0)$ in $\mathcal{C}_{A_E}$ 
combines with the given triangle to induce a morphism $h': {\rm Det}_{A_E}(E\otimes_R C_2) \to {\rm Det}_{A_E}(0)$ in $\mathcal{C}_{A_E}$. The same approach as used to prove \cite[Lem. 1.3.4]{FukayaKato:2006} then shows that 
\[ \chi_{\mathfrak{A},E}(C_2,h') = \chi_{\mathfrak{A},E}(C_1,h) + \chi_{\mathfrak{A},E}(C_3,0).\]

Secondly, if $h$ and $h'$ are any two morphisms ${\rm Det}_{A_E}(E\otimes_R C) \to {\rm Det}_{A_E}(0)$ in $\mathcal{C}_{A_E}$, then the (obvious analogue of the) defining relation (3) in 
\cite[Def. 1.3.2]{FukayaKato:2006} (with $C'$ taken to be $0$) implies that 
\[ \chi_{\mathfrak{A},E}(C,h') = \chi_{\mathfrak{A},E}(C,h) + \partial_{\mathfrak{A},E}(h'\circ h^{-1}).\]
Here the last term denotes the image under $\partial_{\mathfrak{A},E}$ of the unique element of $K_1(A_E)$ that is determined by the morphism $h'\circ h^{-1}: {\rm Det}_{A_E}(0) \to {\rm Det}_{A_E}(0)$ in $\mathcal{C}_{A_E}$. 

We next assume $\mathfrak{A} = R[G]$ for a finite group $G$, and write $\iota_{R[G]}^\#$ for the involutions on each of the groups $K_1(R[G]), K_1(F[G])$ and $K_1(R[G],F[G])$ that are induced by the $R$-linear  anti-involution on $R[G]$ that inverts elements of $G$. Then, if $M$ is any finite $R[G]$-module that is c-t, its Pontryagin dual $M^\ast$ (endowed with contragredient $G$-action) is also c-t, and %one has 
\begin{equation}\label{dual eq} 
\chi_{R[G],F[G]}(M^\ast[0],0) = \iota_{R[G]}^\#\bigl( \chi_{R[G],F[G]}(M[0],0)\bigr).\end{equation}
(By localisation, the verification of this equality reduces to the case that $R$ is a discrete valuation ring. In the latter case it then follows by explicit computation from the fact that a finite c-t $R[G]$-module has a free resolution of length one.) 

\begin{remark}{\em We often regard $E$ as clear from context and so write $\chi_\mathfrak{A}(-,-)$ in place of $\chi_{\mathfrak{A},E}(-,-)$. If $\mathfrak{A} = \ZZ[G]$, we further abbreviate $\chi_{\ZZ[G],E}(-,-)$  to $\chi_{G}(-,-)$, and the maps $\partial_{\ZZ[G],E}(-)$ and $\partial'_{\ZZ[G],E}$ to $\partial_{G}(-)$ and $\partial'_{G}(-)$ (again regarding $E$ as clear from context).}\end{remark}

\subsection{The refined Birch and Swinnerton-Dyer Conjecture}

\subsubsection{}In the sequel we write
 \[
h^{\rm NT}_{A,L}: A(L)\times A^t(L) \to \RR
\]
for the classical N\'eron-Tate height-pairing for $A$ over $L$.

This pairing is non-degenerate and hence, assuming $\sha(A/L)$ to be finite, combines with the properties of the Selmer complex ${\rm SC}_{V_L}(A, L/K)$ established in Proposition \ref{prop:perfect}(ii) to induce a canonical isomorphism of $\RR[G]$-modules
\begin{equation} \label{height triv}
h^{{\rm NT},{\rm det}}_{A,L}: {\rm Det}_{\RR[G]}(\RR \otimes_{\ZZ} {\rm SC}_{V_L}(A,L/K)) \cong {\rm Det}_{\RR[G]}(0).
\end{equation}

In particular, since ${\rm SC}_{V_L}(A,L/K)$ belongs to $D^{\rm perf}(\ZZ[G])$, we obtain an element of $K_0(\ZZ[G],\RR[G])$ by setting
\[
\chi_G^{\rm BSD}(A,V_L) := \chi_{G}({\rm SC}_{V_L}(A, L/K),h^{{\rm NT},{\rm det}}_{A,L}).
\]

Next we note that, since the complex $R\Gamma(X, \mathcal{L}_L)^*$ considered in Lemma \ref{zar lemma} belongs to $D^{\rm perf}(\mathbb{F}_p[G])$, it defines an object of $D^{\rm perf}(\ZZ[G])$ for which $\QQ\otimes_\ZZ R\Gamma(X, \mathcal{L}_L)^*$ is acyclic. The associated element 
\[ \chi_G^{\rm coh}(A,V_L) := \chi_{G}(R\Gamma(X, \mathcal{L}_L)^*,0)\]
therefore belongs to the image of the natural homomorphism
\begin{equation}\label{p-subgroup} K_0(\mathbb{F}_p[G]) \to K_0(\ZZ[G],\QQ[G]) \subset K_0(\ZZ[G],\RR[G]).\end{equation}

Finally, for each prime $\ell$, we shall use an explicit computation of Bockstein homomorphisms that arise naturally in arithmetic cohomology to define  a canonical, and computable, integer $a_\ell = a_{A,L,\ell}$ in $\{0,1\}$. We thereby obtain a canonical element of $K_0(\ZZ[G],\QQ[G])$ of order dividing two by setting
\[ \chi_G^{\rm sgn}(A) := {\sum}_{\ell}\partial_{G,\QQ}(\langle \QQ\cdot A^t(L),(-1)^{a_\ell}\rangle)_\ell,\]
where $\ell$ runs over all prime divisors of $|G|$. (Given the relatively minor role that this `sign-term' plays in our conjecture, and the involved nature of the relevant Bockstein homomorphisms, we prefer to delay giving explicit details regarding the integers $a_\ell$ until the respective computations are made in Proposition \ref{prop:htfrobcomp}(i) for $\ell =p$ and in equation (\ref{ell sign}) for $\ell\not=p$.)

\subsubsection{}We can now state our refined version of the Birch and Swinnerton-Dyer Conjecture for $A$ over $L$.

For each character $\psi$ in ${\rm Ir}(G)$, we fix an associated complex representation $V_\psi$ and  write %$\check\psi$ for the contragredient of $\psi$ and 
$e_\psi$ for the primitive idempotent $\psi(1)|G|^{-1}{\sum}_{g \in G}\psi(g^{-1})g$ of $\zeta(\CC[G])$. We then set 
\[ r_{\rm alg}(\psi) := \psi(1)^{-1}\cdot {\rm dim}_{\CC}(e_{\psi}(\CC\otimes_\ZZ A^t(L))) = {\rm dim}_{\CC}(\Hom_{\CC[G]}(V_{\psi},\CC\otimes_\ZZ A^t(L))),\]
%where $\check\psi$ denote the contragredient of $\psi$,
 and write 
\[ r_{\rm an}(\psi) := {\rm ord}_{s=1}L_U(A,\psi,s)\]
for the order of vanishing at $s=1$ of the series $L_U(A,\psi,s)$. We also use the `leading term' element $L^*_{U}(A_{L/K},1)$ of $K_1(\RR[G])$ that is defined in Theorem \ref{lt def}.

\begin{conjecture} \label{conj:ebsd} The following claims are valid.
\begin{itemize}
\item[(i)] For each character $\psi$ in ${\rm Ir}(G)$ one has $r_{\rm an}(\psi)= r_{\rm alg}(\psi)$.
\item[(ii)] The group $\sha(A/L)$ is finite.
\item[(iii)] Let $U$ be a dense open subset of $X$ comprising points at which both $L/K$ is unramified and $A/K$ has good reduction. Then, for every family of groups $V_L = V_{U_L}$ chosen as in \S\ref{selmer section}, there is an equality 
\[ \partial_{G}( L^*_{U}(A_{L/K}, 1)) = \chi_G^{\rm BSD}(A,V_{L}) - \chi_G^{\rm coh}(A,V_{L}) + \chi_G^{\rm sgn}(A)
\]
in $K_0(\ZZ[G],\RR[G])$.
\end{itemize}
\end{conjecture}

\begin{remark}\label{consistent lemma}{\em If $L=K$, then $K_0(\ZZ[G],\RR[G])$ identifies with the multiplicative group $\RR^\times/\{\pm 1\}$ and in   Proposition \ref{classical bsd} below we shall show that this case of Conjecture \ref{conj:ebsd} recovers the classical Birch and Swinnerton-Dyer conjecture for $A$. In \S\ref{consistent sect} we also show that the validity of the equality in Conjecture \ref{conj:ebsd}(iii) is independent of the choices of open set $U$ and  family of subgroups $V_L$}. \end{remark}

\begin{remark}\label{rank remark} {\em Since $\CC\otimes_\ZZ A^t(L)$ is the scalar extension of the finitely generated $\QQ[G]$-module $\QQ\otimes_\ZZ A^t(L)$ one has 
 $r_{\rm alg}(\psi) = r_{\rm alg}(\omega\circ \psi)$ for all $\psi \in {\rm Ir}(G)$ and all automorphims $\omega$ of $\CC$. Conjecture \ref{conj:ebsd}(i) therefore implies that  $r_{\rm an}(\psi) = r_{\rm an}(\omega\circ\psi)$ for all such $\psi$ and $\omega$. The validity of these equalities can be derived directly from the equalities (\ref{deligne}) and (\ref{first inv})  that played the  key role in the proof of  Proposition \ref{prop:positive}. } 
\end{remark}

\begin{remark}\label{positivity remark}{\em Theorem \ref{lt def} allows us to formulate Conjecture \ref{conj:ebsd} directly in terms of the connecting homomorphism $\partial_G$. % and also contributes to the derivation of consequences such as  Proposition \ref{cons prop}. 
However, without using Theorem \ref{lt def}, one could still formulate an analogue of Conjecture \ref{conj:ebsd} in terms of the image of the element $(L^*_U(A,\chi,1))_{\chi \in {\rm Ir}(G)}$ of $\zeta(\RR [G])^\times$ under the `extended boundary' homomorphism $\zeta(\RR[G])^\times \to K_0(\ZZ[G],\RR[G])$ constructed by Flach and the first author in \cite[\S4.2, Lem. 9]{BF_Tamagawa}. This observation provides the link between the formulation of Conjecture \ref{conj:ebsd}(iii) in terms of relative algebraic $K$-theory and the formalism of `equivariant Tamagawa number conjectures'  that is discussed in loc. cit. and later refined by Fukaya and Kato in \cite{FukayaKato:2006}.}\end{remark} 

\subsubsection{}\label{cons2} Conjecture \ref{conj:ebsd} entails a variety of explicit consequences concerning the structure of Selmer complexes and relations between leading terms of Hasse-Weil-Artin $L$-series. 
 To help provide context, we now state two concrete results in this direction (though, for convenience, the proof of these results is deferred to \S\ref{cons props proof}). 

We fix a prime $p$ and assume, for simplicity, that $G$ is of $p$-power order. We write $M_{(p)}$ for the $p$-localisation of a (complex of) abelian groups $M$.

The first result concerns the Galois structure of the complex ${\rm SC}_{V_L} := {\rm SC}_{V_L}(A,L/K)$.

\begin{proposition} \label{cons prop} If $G$ is a group of $p$-power order, then Conjecture \ref{conj:ebsd} implies the following restrictions on the complex ${\rm SC}_{V_L}$.   
\begin{itemize}
\item[(i)] ${\rm SC}_{V_L}$ is isomorphic in $D^{\rm perf}(\ZZ[G])$ to  a bounded complex of finitely generated free $G$-modules.
\item[(ii)] If $A(K)[p]$ and $A^t(K)[p]$ both vanish, then ${\rm SC}_{V_L,(p)}$ is isomorphic in $D^{\rm perf}(\ZZ_{(p)}[G])$ to a complex $\ZZ_{(p)}[G]^t \xrightarrow{\phi} \ZZ_{(p)}[G]^t$, 
where the first term is placed in degree one. 
\end{itemize}
\end{proposition}

The second result we record describes families of algebraic relations between suitable normalisations of the leading terms $L^*_U(A,\psi, 1)$ for varying characters $\psi$. 

To state this result, we assume the hypotheses of Proposition \ref{cons prop}(ii) and  fix a representative of ${\rm SC}_{V_L,(p)}$ of the specified form. We then consider the composite isomorphism 
\[ \iota_{A,L}^{\rm NT}: 
\RR[G]^t \cong (\RR\cdot \ker(\phi)) \oplus
(\RR \cdot {\rm im}(\phi)) \xrightarrow{ h_{A,L,\ast}^{\rm NT} \oplus {\rm id}} (\RR\cdot \cok(\phi))
\oplus (\RR\cdot {\rm im}(\phi)) \cong \RR[G]^t\]
of $\RR[G]$-modules. 
Here the first and third maps are induced by a choice of
 $\RR[G]$-equivariant sections to the surjective maps from 
 $\RR[G]^t$ to $\RR\cdot{\rm im}(\phi)$ and $\RR\cdot \cok(\phi)$ that are respectively induced by $\phi$ and by the tautological projection. In addition, $h_{A,L,\ast}^{\rm NT}$ denotes the composite 
 \[ \RR\cdot \ker(\phi) \cong \RR\cdot A^t(L) \cong \Hom_{\RR}\bigl(\RR\cdot\! A(L),\RR\bigr)   \cong \RR\cdot \cok(\phi),\]
 in which the second isomorophism is induced by the non-degenerate pairing $h_{A,L}^{\rm NT}$ and the first and third by Proposition \ref{prop:perfect}(ii)(a) and the fixed identifications of $\ker(\phi)$ and $\cok(\phi)$ with $H^0({\rm SC}_{V_L})_{(p)}$ 
 and $H^1({\rm SC}_{V_L})_{(p)}$. 
 
We write $\chi(\mathcal{L})$ for the integer obtained as the Euler characteristic in $K_0(\FF_p) \cong \ZZ$ of the complex  $R\Gamma(X, \mathcal{L})^*$ of $\mathbb{F}_p$-modules. 

For each character $\psi$ in ${\rm Ir}(G)$, we then normalise the leading term of the associated Hasse-Weil-Artin $L$-series by setting 
\begin{equation}\label{normalise def} \mathscr{L}(A,\psi) := \frac{ p^{\psi(1)\chi(\mathcal{L})}\cdot L^\ast_U(A,\psi,1)}{(-1)^{r_{\rm alg}(\psi)a_p}\cdot\det(\iota_{A,L,\psi}^{\rm NT})},
%\bigl(\Hom_{\CC[G]}(V_\psi,\CC\otimes_\RR\iota_{A,L}^{\rm NT})\bigr)} \in \CC^\times .
\end{equation}
where $\iota_{A,L,\psi}^{\rm NT}$ is the automorphism of $\Hom_{\CC[G]}(V_\psi,\CC[G]^t)$ induced by  $\iota_{A,L}^{\rm NT}$.
We also write $\QQ(\psi)$ for the field generated over $\QQ$ by the set $\{\psi(g): g \in G\}$.

\begin{proposition}\label{cons prop 2} Assume the hypotheses of Proposition \ref{cons prop}(ii) and the validity of  Conjecture \ref{conj:ebsd}(i) and (ii). Then the following claims are valid. \
\begin{itemize}
\item[(i)] For all $\psi$ in ${\rm Ir}(G)$ and $\alpha$ in $\Gal(\QQ(\psi)/\QQ)$ one has 
\[ \mathscr{L}(A,\psi) \in \QQ(\psi)\,\,\text{ and }\,\, \alpha(\mathscr{L}(A,\psi)) = \mathscr{L}(A,\alpha\circ\psi).\]
\end{itemize}
In the rest of the result we also assume the validity of the equality in Conjecture \ref{conj:ebsd}(iii).
\begin{itemize} 

\item[(ii)] For every abelian subquotient $Q = H/J$ of $G$, there is a containment 
\[ \mathscr{L}(A_{L^H},{\bf 1}_{Q})\in \ZZ_{(p)}^\times\]
and, for each $\gamma$ in $Q$, a congruence
\[ {\sum}_{\psi \in {\rm Ir}(Q)}\psi(\gamma)^{-1}\mathscr{L}(A_{L^H},\psi) \equiv 0 \,\,({\rm mod}\,\, |Q|\cdot\ZZ_{(p)}).\] 
\item[(iii)] If, for each subgroup $H$ of $G$, $\psi_H$ is an irreducible character of the abelianisation $H^{\rm ab}$ of $H$ and $m_H$ an integer such that the virtual character ${\sum}_{H \le G}m_H \cdot{\rm ind}_H^G({\rm inf}_{H^{\rm ab}}^H(\psi_H))$ vanishes, then one has   
\[ {\prod}_{H \le G} \mathscr{L}(A_{L^H},\psi_H)^{m_H} = 1.\]
\end{itemize}
\end{proposition}

\begin{remark}{\em By developing methods introduced by the second author in \cite{kakde}, the first two authors give an explicit description of the Whitehead group $K_1(\ZZ_p[G])$ in \cite[Th. 2.1]{bk}. In the setting of Proposition \ref{cons prop 2}, this result has the following explicit consequence. For every cyclic subgroup $C$ of $G$, the properties in  Proposition \ref{cons prop 2}(ii) combine to imply that %the element 
\[ \mathscr{L}(A,C) := |C|^{-1}{\sum}_{c \in C}{\sum}_{\psi \in {\rm Ir}(C)} \psi(c)^{-1}\mathscr{L}(A_{L^C},\psi)\cdot c\]
belongs to $\ZZ_{(p)}[C]^\times \subset \ZZ_{p}[C]^\times \cong K_1(\ZZ_{p}[C])$. Fix an embedding $j: \RR \to \CC_p$ and write $j_\ast$ for the induced embedding $K_0(\ZZ[G],\RR[G]) \to K_0(\ZZ_p[G],\CC_p[G])$ of relative $K_0$-groups. Then  \cite[Th. 2.1]{bk} implies that the validity of the image under $j_\ast$ of the equality in Conjecture \ref{conj:ebsd}(iii) is equivalent (under the hypotheses of Proposition \ref{cons prop}(ii)) to the validity of the family of equalities in Proposition \ref{cons prop 2}(iii) together with a single explicit congruence relation between the images of the individual elements $\mathscr{L}(A,C)$ under the respective induction maps $K_1(\ZZ_{p}[C]) \to K_1(\ZZ_p[G])$.}\end{remark}

\subsection{The main results}

In order to state our main result we must define the finite subgroup $\mathcal{T}_{A,L/K}$ of $K_0(\ZZ[G],\RR[G])$ that was discussed in the introduction.

If $\Xi$ is a quotient of a subgroup $\Delta$ of a finite group $\Gamma$, then we consider the composite homomorphism of abelian groups
\[ \pi^{\Gamma}_{\Xi}: K_0(\ZZ_p[\Gamma],\QQ_p[\Gamma]) \to K_0(\ZZ_p[\Delta],\QQ_p[\Delta]) \to K_0(\ZZ_p[\Xi],\QQ_p[\Xi]),\]
where the first map is restriction of scalars and the second is the natural coinflation homomorphism.

By the semistable reduction theorem, the set $\Sigma = \Sigma_{A,K}$ of finite Galois extensions of $K$ over which $A$ is semistable is non-empty. For each field $K'$ in $\Sigma$ we write $L'$ for the composite of $L$ and $K'$, set $G' := \Gal(L'/K)$ and $H' := \Gal(L'/K')$, write $P'$ for the normal subgroup of $H'$ that is generated by the Sylow $p$-subgroups of the inertia groups in $H'$ of each place in $K'$ and set $\pi_{K'} := \pi^{G'}_{H'/P'}$. We then define 
\begin{equation}\label{calT def} \mathcal{T}_{A,L/K} := K_0(\ZZ_p[G], \QQ_p[G])_{\rm tor} \cap \bigl( {\bigcap}_{K' \in \Sigma_1} \ker(\pi_{K'}) \bigr)\cap \bigl({\bigcap}_{K'\in \Sigma_2}\pi^{G'}_{G}(\ker(\pi_{K'}))\bigr),\end{equation}
where $\Sigma_1$ denotes the possibly empty subset of $\Sigma$ comprising fields $K'$ that are contained in $L$ (so that $G' = G$) and $\Sigma_2$ the possibly empty subset of $\Sigma$ comprising fields $K'$ that are not contained in $L$ but are such that $\sha(A/L')$ is finite. 

We recall that $K_0(\ZZ_p[G], \QQ_p[G])_{\rm tor}$ is finite and, in all cases, regard $\mathcal{T}_{A,L/K}$ as a subgroup of $K_0(\ZZ[G],\RR[G])$ via the natural embeddings
\[ K_0(\ZZ_p[G], \QQ_p[G])_{\rm tor} \subset K_0(\ZZ[G],\QQ[G])\subset K_0(\ZZ[G],\RR[G]).\]

We can now state the main evidence that we shall offer in support of Conjecture \ref{conj:ebsd}.

\begin{theorem} \label{thm:thm1} If the $\ell$-primary component of $\sha(A/L)$ is finite for some prime $\ell$, then the following claims are also valid.
\begin{itemize}
\item[(i)] Claims (i) and (ii) of Conjecture \ref{conj:ebsd} are valid.
\item[(ii)] The equality in Conjecture \ref{conj:ebsd}(iii) is valid modulo the finite subgroup $\mathcal{T}_{A,L/K}$ of $K_0(\ZZ[G],\RR[G])$.
\end{itemize}
\end{theorem}

\begin{remark}\label{sha finite rem}{\em  It is proved by Kato and Trihan in \cite{KT} that $\sha(A/L)$ is finite if and only if at least one of its $\ell$-primary components is finite. Thus, under the hypotheses of Theorem \ref{thm:thm1}, we can (and do) assume, without further comment, that $\sha(A/L)$ is finite (and hence that Conjecture \ref{conj:ebsd}(ii) is valid).}
\end{remark}

\begin{remark}{\em The main result that we prove here is, in principle, stronger than Theorem \ref{thm:thm1} but is more technical to state (for more details see Remark \ref{stronger main} below). One can also provide further evidence in support of Conjecture \ref{conj:ebsd} in the setting of generically ordinary abelian varieties, and we hope to discuss this elsewhere. }\end{remark}

\begin{remark}{\em Assume $G$ is a $p$-group, that the groups $A(K)[p]$ and $A^t(K)[p]$ both vanish and that some $\ell$-primary component of $\sha(A/L)$ is finite (where $\ell$ can be different from $p$). Then Theorem  \ref{thm:thm1} combines with Proposition \ref{cons prop 2} to imply the unconditional validity of the relations in Proposition \ref{cons prop 2}(i).}\end{remark}

In special cases it is possible to describe $\mathcal{T}_{A,L/K}$ explicitly and hence to make Theorem \ref{thm:thm1}(ii) much more concrete.

For example, if the sets $\Sigma_1$ and $\Sigma_2$ that occur in (\ref{calT def}) are both empty, then $\mathcal{T}_{A,L/K}$ is equal to $K_0(\ZZ_p[G], \QQ_p[G])_{\rm tor}$. On the other hand, if $A$ is semistable over $K$ and $L/K$ is tamely ramified, then the field $K'=K$ belongs to $\Sigma_1$ and is such that $G = G'= H'$ and $P'$ is trivial and so $\mathcal{T}_{A,L/K}$ vanishes. Hence, in the latter case, Theorem \ref{thm:thm1} has the following more explicit consequence.

\begin{corollary} Assume that $A$ is semistable, that $L/K$ is tamely ramified and that some $\ell$-primary component of $\sha(A/L)$ is finite. Then Conjecture \ref{conj:ebsd} is unconditionally valid. \end{corollary}

As far as we are aware, this result gives the first verification, modulo only the assumed finiteness of Tate-Shafarevich groups, of a refined version of the Birch-Swinnerton-Dyer conjecture in the context of ramified extensions.

\section{Preliminary results}\label{sect: preliminary}

In this section we first prove a purely algebraic result that is important for several subsequent arguments.

We then verify that the statement of Conjecture \ref{conj:ebsd} is consistent in certain key respects (as promised in Remark \ref{consistent lemma}).

Finally we use a result of Schneider to give a reinterpretation of the conjecture that plays an essential role in the proof of Theorem \ref{thm:thm1}.

\subsection{A result in $K$-theory} The following purely algebraic observation will underpin the proof of several subsequent results.

\begin{proposition}\label{prop:triangle} Let $R$ be a Dedekind domain with field of fractions $F$ and $\mathfrak{A}$ an $R$-order in a finite dimensional semisimple $F$-algebra $A$.

We suppose to be given exact triangles in $D^{\rm perf}(\mathfrak{A})$ of the form
\begin{equation}\label{triangles} C_\theta \to C_1 \xrightarrow{\theta} C_2 \to C_\theta[1]\,\,\,\text{ and }\,\,\, C_\phi \to C_1 \xrightarrow{\phi} C_2 \to C_\phi[1]\end{equation}
that satisfy all of the following conditions.
\begin{itemize}
\item[(a)] In each degree $i$ there are natural identifications $F\otimes_R H^i(C_1) = F\otimes_R H^i(C_2)$, with respect to which one has

\item[(b)] the composite tautological homomorphism of $A$-modules
\[ F\otimes_R\ker(H^i(\theta)) \subseteq F\otimes_R H^i(C_1) = F\otimes_R H^i(C_2)\to F\otimes_R{\rm cok}(H^i(\theta))\]
is bijective, and
\item[(c)] the map $H^i(\phi)$ induces the identity homomorphism on $F\otimes_R H^i(C_1) = F\otimes_R H^i(C_2)$.
\end{itemize}

Then the following claims are valid.

\begin{itemize}
\item[(i)] The bijectivity of the maps in (b) combines with the first triangle in (\ref{triangles}) to induce a canonical morphism
 $\tau_\theta: {\rm Det}_{A}(F\otimes_{R}C_\theta) \cong {\rm Det}_A(0)$ of (non-abelian) determinants.
\item[(ii)] In each degree $i$ the homomorphism $H^i(\theta)$ induces an automorphism $H^i(\theta)^{\diamond}$ of any $A$-equivariant complement to $F\otimes_R\ker(H^i(\theta))$ in $F\otimes_R H^i(C_1)$ in such a way that ${\rm Nrd}_{A}(H^i(\theta)^\diamond)$ is independent of the choice of complement.
\item[(iii)] The complex $F\otimes_RC_\phi$ is acyclic.
\item[(iv)]  In $K_0(\mathfrak{A},A)$ one has
\[ \chi_{\mathfrak{A}}(C_\theta,\tau_\theta) - \chi_{\mathfrak{A}}(C_\phi,0) = \partial_{\mathfrak{A},F}({\prod}_{i\in \ZZ}(H^i(\theta)_F^\diamond)^{(-1)^i}),\]
where we identify each automorphism $H^i(\theta)_F^\diamond$ with the associated element of $K_1(A)$.
\end{itemize}
\end{proposition}

\begin{proof} If $M$ denotes either an $R$-module or a complex of $R$-modules, then we abbreviate $F\otimes_RM$ to $M_F$.

To construct a morphism $\tau_\theta$ as in claim (i) we note first that the long exact cohomology sequence of the left hand exact triangle in (\ref{triangles}) gives in each degree $i$ a short exact sequence of $\mathfrak{A}$-modules 
\[  0 \to {\rm cok}(H^{i-1}(\theta)) \to H^i(C_\theta) \to \ker(H^i(\theta)) \to 0.\]

Then, upon tensoring these exact sequences with $F$ (over $R$), applying the determinant functor ${\rm Det}_A$ and then taking account of the isomorphisms given in (b) one obtains isomorphisms of (non-abelian) determinants
\begin{align}\label{passage} {\rm Det}_A(H^i(C_\theta)_F) &\cong {\rm Det}_A({\rm cok}(H^{i-1}(\theta))_F)\cdot {\rm Det}_A({\rm ker}(H^{i}(\theta))_F)\\
                            &\cong {\rm Det}_A({\rm cok}(H^{i-1}(\theta))_F)\cdot {\rm Det}_A({\rm cok}(H^{i}(\theta))_F).\notag\end{align}

We then define the morphism $\tau_\theta$ in claim (i) to be the composite
\begin{align*}  {\rm Det}_A((C_\theta)_F) &\cong {\prod}_{i \in \ZZ}{\rm Det}_A(H^i(C_\theta)_F)^{(-1)^i}\\
 &\cong {\prod}_{i \in \ZZ}[{\rm Det}_A({\rm cok}(H^{i-1}(\theta))_F)\cdot {\rm Det}_A({\rm cok}(H^{i}(\theta))_F)]^{(-1)^i}\\
 &\cong {\prod}_{i\in \ZZ}[{\rm Det}_A({\rm cok}(H^{i}(\theta))_F)^{-1}\cdot {\rm Det}_A({\rm cok}(H^{i}(\theta))_F)]^{(-1)^i}\\
 &\cong {\prod}_{i\in \ZZ}{\rm Det}_A(0)^{(-1)^i}\\
 &= {\rm Det}_A(0).\end{align*}
Here the first map is the canonical `passage to cohomology' map, the second is induced by the maps (\ref{passage}) in each degree $i$, the third by the obvious rearrangement of terms and the fourth from the canonical morphisms 
\[ {\rm Det}_A({\rm cok}(H^{i}(\theta))_F)^{-1}\cdot {\rm Det}_A({\rm cok}(H^{i}(\theta))_F) \cong {\rm Det}_A(0).\]

Claim (ii) is a straightforward consequence of the condition (b) and claim (iii) follows directly upon combining the long exact cohomology sequence of the second triangle in (\ref{triangles}) with the condition (c).

Finally, to prove claim (iv) we fix bounded complexes of finitely generated projective $\mathfrak{A}$-modules $P_1$ and $P_2$ that are respectively isomorphic in $D(\mathfrak{A})$ to $C_1$ and $C_2$. Then the morphisms $\theta$ and
$\phi$ are represented by morphisms of complexes of $\mathfrak{A}$-modules of the form $\theta':P_1\to P_2$ and $\phi': P_1\to P_2$.

The key to our argument is then to consider the exact triangle
\begin{equation}\label{better triangle} C_\theta\oplus {\rm Cone}(\phi') \xrightarrow{(\kappa,{\rm id})} P_1\oplus {\rm Cone}(\phi') \xrightarrow{(\kappa',0)} {\rm Cyl}(\theta')\to (C_\theta\oplus {\rm Cone}(\phi'))[1]\end{equation}
in $D(\mathfrak{A})$ where $\kappa$ is the morphism $C_\theta\to P_1$ induced by the first triangle in (\ref{triangles}) and $\kappa'$ the morphism $P_1 \to {\rm Cyl}(\theta')$ induced by $\theta'$ and the natural quasi-isomorphism ${\rm Cyl}(\theta') \cong P_2$.

We note first that this triangle satisfies the analogues of conditions (a) and (b) (with $C_1$, $C_2$ and $\theta$ replaced by $P_1\oplus {\rm Cone}(\phi')$, ${\rm Cyl}(\theta')$ and $(\kappa',0)$) and, in addition, that in each degree $i$ one has $(P_1\oplus {\rm Cone}(\phi'))^i = {\rm Cyl}(\theta')^i$.

%Further, the acyclicity of $F\otimes_R C_\phi$ is an immediate consequence of assumption (c) and implies that
%
Further, the acyclicity of $F\otimes_R C_\phi$ implies that
\[ \chi_{\mathfrak{A}}(C_\theta\oplus {\rm Cone}(\phi'),\tau_\theta) = \chi_{\mathfrak{A}}(C_\theta,\tau_\theta) + \chi_{\mathfrak{A}}(C_\phi[1],0) =  \chi_{\mathfrak{A}}(C_\theta,\tau_\theta) - \chi_{\mathfrak{A}}(C_\phi,0),\]
where the first equality is true because ${\rm Cone}(\phi')$ is isomorphic to $C_\phi[1]$.

In particular, after replacing the first triangle in (\ref{triangles}) by (\ref{better triangle}), we are reduced to proving that if $C_1$ and $C_2$ are represented by bounded complexes of finitely generated projective $\mathfrak{A}$-modules $P_1$ and $P_2$ with $P_1^i = P_2^i$ in each degree $i$, then the conditions (a), (b) and (c) combine to imply an equality
\begin{equation}\label{new version} \chi_{\mathfrak{A}}(C_\theta,\tau_\theta) = \delta_{\mathfrak{A}}({\prod}_{i \in \ZZ}{\rm Nrd}_{A}(H^i(\theta)_F^\diamond)^{(-1)^i}),\end{equation}
where $\delta_{\mathfrak{A}}$ denotes the composite $\partial_{\mathfrak{A},F}\circ ({\rm Nrd}_{A})^{-1}: \im({\rm Nrd}_{A}) \to K_0(\mathfrak{A},A)$.

To do this we note first that, under these conditions, an easy downward induction on $i$ (using hypothesis (c)) implies that in each degree $i$ the $F$-spaces spanned by the groups of boundaries $B^i(P_1)$ and $B^i(P_2)$ have the same dimension.

If necessary, we can then also change $\theta$ by a homotopy (without changing conditions (b)) in order to ensure that, in each degree $i$, the restriction of
$\theta^{i+1}$ is injective on $B^i(P_1)$ and hence induces an isomorphism $F\otimes_R B^i(P_1) \cong F\otimes_R B^i(P_2)$ (for details of such an argument see, for example, the proof of  \cite[Lem. 7.10]{ckps}).

Having made these constructions, one can then simply mimic the argument of \cite[Prop. 3.1]{Bu4} in order to prove the required equality (\ref{new version}) by using induction on the number of non-zero terms in $P_1$. \end{proof}

\subsection{Consistency checks}\label{consistent sect}

\begin{proposition}\label{classical bsd} If $L=K$, then Conjecture \ref{conj:ebsd} recovers the classical Birch and Swinnerton-Dyer conjecture for $A$.
 \end{proposition}

\begin{proof} We assume $\sha(A/K)$ is finite and abbreviate ${\rm SC}_{V_K}(A,K/K)$ to ${\rm SC}_{V_K}$.

Now, if $L=K$, then $G$ is the trivial group ${\rm id}$ and $K_0(\ZZ[G],\RR[G])$ identifies with the multiplicative group $\mathbb{R}^\times/\{\pm 1\}$. In addition, upon unwinding the definition of Euler characteristic one finds that, with respect to the latter identification, there is an equality
\begin{equation}\label{explicit trivial} \chi_{\rm id}^{\rm BSD}(A,V_K) \equiv {\rm disc}(h^{\rm NT}_{A,K})\cdot {\prod}_{i \in \ZZ}\#(H^i({\rm SC}_{V_K})_{\rm tor})^{(-1)^{i+1}} \,\,\text{(mod }\,\,\pm 1)\end{equation}
where ${\rm disc}(h^{\rm NT}_{A,K})$ denotes the discriminant of the pairing $h^{\rm NT}_{A,K}$.

To compute the above product we write $\theta$ for the natural map $A(K)_{\rm tor} \rightarrow {\bigoplus}_{v \notin U} A(K_v)/V_v$. Then, from Propositions \ref{prop:arithcoho} and \ref{prop:perfect}, one finds that there are equalities $H^0({\rm SC}_{V_K}) = A^t(K)$ and $H^2({\rm SC}_{V_K}) = \ker(\theta)^*$ and a short exact sequence of the form
\[
0 \rightarrow X_{\mathbb{Z}}(A/K) \rightarrow H^1({\rm SC}_{V_K}) \rightarrow {\rm cok}(\theta)^*\rightarrow 0.
\]

Upon combining these observations with the natural exact sequences
\[ 0 \to \ker(\theta) \to A(K)_{\rm tor} \rightarrow {\bigoplus}_{v \notin U} A(K_v)/V_v \to \cok(\theta) \to 0\]
and
\[
0 \rightarrow \sha(A/K)^* \rightarrow X_{\mathbb{Z}}(A/K) \rightarrow \Hom_{\mathbb{Z}}(A(K), \mathbb{Z}) \rightarrow 0
\]
one computes that
\begin{equation}\label{explicit trivial 2} {\rm disc}(h^{\rm NT}_{A,K})\cdot {\prod}_{i \in \ZZ}\#(H^i({\rm SC}_{V_K})_{\rm tor})^{(-1)^{i+1}} = \frac{\# \sha(A/K){\rm disc}
(h^{\rm NT}_{A,K})}{\#A(K)_{\rm tor}\#A^t(K)_{\rm tor}} {\prod}_{v \notin U} [A(K_v):V_v].\end{equation}

On the other hand, from \cite[3.7.3]{KT}, one finds that
\begin{equation}\label{explicit trivial 3} \chi^{\rm coh}_{\rm id}(V_K, K/K) \equiv \frac{\#H^1(X,\mathcal{L})}{\#H^0(X,\mathcal{L})} \equiv \frac{{\prod}_{v \notin U}[A(K_v):V_v]}{{\rm vol}({\prod}_{v \notin U}A(K_v))}\,\,\text{(mod }\,\,\pm 1),\end{equation}
where the `volume term' here is as defined in \cite[\S1.7]{KT}.

Thus, since $\chi^{\rm sgn}_{\rm id}(A)$ is clearly trivial, the expressions (\ref{explicit trivial}), (\ref{explicit trivial 2}) and (\ref{explicit trivial 3}) combine to show the equality in Conjecture \ref{conj:ebsd}(iii) is equivalent to an equality
\[ L^*_U(A,1) \equiv \pm\frac{\# \sha(A/K) {\rm disc}(h^{\rm NT}_{A,K})}{\#A(K)_{\rm tor}\#A^t(K)_{\rm tor}} {\rm vol}({\prod}_{v \notin U} A(K_v)).\]

Since $L^*_U(A,1)$ is known to be a strictly positive real number (by Proposition \ref{prop:positive}(ii)), this equality is precisely the form of the Birch and Swinnerton-Dyer Conjecture that is discussed in \cite[\S1.8]{KT}. \end{proof}

\begin{proposition}\label{V independent} The validity of Conjecture \ref{conj:ebsd}(iii) is independent of the choice of the family of subgroups $V_L$. \end{proposition}

\begin{proof} It is clearly enough to show that the difference $\chi_G^{\rm BSD}(A,V_L) - \chi^{\rm coh}_G(A,V_L)$ is independent of the choice of $V_L$.

In addition, it suffices to consider replacing $V_L$ by a family of subgroups $V_L' = (V'_w)_{w \notin U_L}$ that satisfies $V_w' \subseteq V_w$ for all $w \notin U_L$.

In this case, the definition of the complexes ${\rm SC}_{V_L'}(A, L/K)$ and ${\rm SC}_{V_L}(A,L/K)$ via the (dual of the) mapping fibre of the respective morphisms (\ref{fibre morphism}) leads naturally to an exact triangle in $D^{\rm perf}(\ZZ[G])$ of the form
%
%\[
%{\rm SC}_{V_L'}(A, L/K) \rightarrow {\rm SC}_{V_L}(A,L/K) \rightarrow Q_1^*[0] \rightarrow
%\]
\[
{\rm SC}_{V_L}(A, L/K) \rightarrow {\rm SC}_{V_L'}(A,L/K) \rightarrow Q_1^*[-1] \rightarrow
\]
with $Q_1:= {\bigoplus}_{w \notin U_L} (V_w/V_w')$, and hence to an equality in $K_0(\ZZ[G],\RR[G])$
%%
%\begin{equation}\label{first eq} 
%\chi^{\rm BSD}_G(A,V_L) - \chi^{\rm BSD}_G(A,V_L') = \chi_G(Q_1^*[0], 0).
%\end{equation}
%%
\begin{equation}\label{first eq} 
\chi^{\rm BSD}_G(A,V_L) - \chi^{\rm BSD}_G(A,V_L') = \chi_G(Q_1^*[0], 0).
\end{equation}

On the other hand, if $\mathcal{L}_L'$ and $\mathcal{L}_L$ are the coherent sheaves that correspond (as in \S\ref{zariski section}) to the collections $V_L'$ and $V_L$ respectively, then there is a natural short exact sequence 
\[ 0 \rightarrow \mathcal{L}' \rightarrow \mathcal{L} \rightarrow Q_2 \rightarrow 0\]
with $Q_2:= {\bigoplus}_{w \notin U_L} W_w/W'_w$. This sequence gives rise to an exact triangle in $D^{\rm perf}(\mathbb{F}_p[G])$ of the form
%%
%\[
%R\Gamma(X, \mathcal{L}_L')^* \rightarrow R\Gamma(X, \mathcal{L}_L)^* \rightarrow Q_2^*[0] \rightarrow
%\]
%
\[
R\Gamma(X, \mathcal{L}_L)^* \rightarrow R\Gamma(X, \mathcal{L}_L')^* \rightarrow Q_2^*[1] \rightarrow
\]
and hence to an equality
\begin{equation}\label{second eq} 
\chi^{\rm coh}_G(A,V_L) - \chi^{\rm coh}_G(A,V_L') = \chi_{\ZZ_p[G]}(Q_2^*[0], 0).
\end{equation}

Now, given the explicit construction of the groups $V_w$ and $V_w'$ from $W_w$ and $W_w'$, it is straightforward to show that, for both $i=1$ and $i=2$ there exists a (finite length) decreasing filtration $(Q_{i,j})_{j \ge 0}$ of the finite $\ZZ_p[G]$-module $Q_i$ such that each module $Q_{i,j}$ is c-t for $G$ and the graded modules ${\rm gr}(Q_i) := {\bigoplus}_{j \ge 0}(Q_{i,j}/Q_{i,j+1})$ are both c-t for $G$ and mutually isomorphic. This fact in turn implies that 
\begin{multline*} \chi_{\ZZ_p[G]}(Q_1^*[0], 0) = \iota_{\ZZ_p[G]}^\#\bigl(\chi_{\ZZ_p[G]}(Q_1[0], 0)\bigr) = 
\iota_{\ZZ_p[G]}^\#\bigl(\chi_{\ZZ_p[G]}({\rm gr}(Q_1)[0], 0)\bigr)\\ = \iota_{\ZZ_p[G]}^\#\bigl(\chi_{\ZZ_p[G]}({\rm gr}(Q_2)[0], 0)\bigr) = \iota_{\ZZ_p[G]}^\#\bigl(\chi_{\ZZ_p[G]}(Q_2[0], 0)\bigr) = \chi_{\ZZ_p[G]}(Q_2^*[0], 0)\end{multline*}
where the first and last equalities follow from the general result (\ref{dual eq}) and the second and fourth from a standard d\'evissage argument. These equalities then combine with (\ref{first eq}) and (\ref{second eq}) to imply the required result.
\end{proof}

\begin{proposition}\label{U independent} The validity of Conjecture \ref{conj:ebsd}(iii) is independent of the choice of $U$. \end{proposition}

\begin{proof} It suffices to fix $v_0$ in $U$ and consider the effect of replacing $U$ by the set $U' := U \setminus \{v_0\}$.

We fix a family $V_L = (V'_w)_{w \notin U'_L}$ of subgroups as in Lemma \ref{lem:KT6.4} and assume, following Remark \ref{good unramified}, that for each place $w$ above $v_0$ one has $V'_w = \mathcal{A}(\mathfrak{m}_w)$. We also write $V^\dagger_L$ for the associated family $(V'_w)_{w \notin U_L}$.

Then, setting $\mathcal{E}_{v_0} := L_U^*(A_{L/K}, 1)\cdot L_{U'}^*(A_{L/K}, 1)^{-1}$, it is enough for us to prove that 
%%
%\begin{equation}\label{needed eq} \partial_{G, \mathbb{R}}(\mathcal{E}_{v_0}) = (\chi_G^{\rm BSD}(A,V_L) - \chi^{\rm BSD}_G(A,V'_L))  -
%(\chi_G^{\rm coh}(A,V^\dagger_L)  - \chi_G^{\rm coh}(A,V_L)) \end{equation}
%
\begin{equation}\label{needed eq} 
\partial_{G, \mathbb{R}}(\mathcal{E}_{v_0}) = (\chi_G^{\rm BSD}(A, V_L^\dagger) - \chi^{\rm BSD}_G(A,V_L))  -  (\chi_G^{\rm coh}(A,V^\dagger_L)  - \chi_G^{\rm coh}(A,V_L)) 
\end{equation}
in $K_0(\ZZ[G],\RR[G])$. In addition, $\mathcal{E}_{v_0}$ belongs to the subgroup $K_1(\QQ[G])$ of $K_1(\RR[G])$ and we claim that ${\rm Nrd}_{\QQ[G]}(\mathcal{E}_{v_0})$ is equal to the evaluation at $u=1$ of the expression 
\begin{multline*}
 {\rm Nrd}_{\mathbb{Q}_{\ell}[G]}\big(1- u^{{\rm deg}(v_0)} \varphi_p^{{\rm deg}(v_0)}: \big(\mathbb{Q}_\ell[G]\otimes_{\ZZ_\ell}T_{\ell}(A)\big)^\vee \big) \\
= {\rm Nrd}_{\mathbb{Q}_p[G]}\big(1- u^{{\rm deg}(v_0)} \varphi^{{\rm deg}(v_0)} : \big(\mathbb{Q}_p[G]\otimes_{\QQ_p}H^0_{\rm crys}(k(v_0)/\ZZ_p, D_{v_0})\big)^\vee \big).\end{multline*}
Here $\ell$ is any choice of prime different from $p$ and $\varphi_p$ is the geometric $p$-th power Frobenius map on $T_{\ell}(A)$, the endomorphism $\varphi$ is such that $p\varphi$ is induced by the crystalline Frobenius on the fibre $D_{v_0}$ at $v_0$ of the covariant Dieudonn\'e crystal $D$ and the above displayed equation follows from the result \cite[Th.~1]{KM} of Katz and Messing. % and lastly, $V^\vee$ denotes the contragredient of a finitely generated $\Q_\ell[G]$- or $\Q_p[G]$- module $V$. 
 To verify this claim about ${\rm Nrd}_{\QQ[G]}(\mathcal{E}_{v_0})$ it is enough to fix an arbitrary $\chi\in{\rm Ir}(G)$, with a corresponding realisation $V_\chi$ over $\QQ_\ell^c$ with $\ell\ne p$, and then note that  
\begin{align*}
e_\chi\bigl({\rm Nrd}_{\QQ[G]}(\mathcal{E}_{v_0})\bigr)
&= {\rm det}\big(1-  \varphi_p^{{\rm deg}(v_0)} : V_\chi\otimes_{\bz_\ell} T_\ell(A) \big)\\
&= {\rm det}\big(1-  \varphi_p^{{\rm deg}(v_0)} : \Hom_{\QQ_\ell^c[G]}(V_{\check\chi}, \QQ_\ell^c[G]\otimes_{\bz_\ell}T_\ell(A) ) \big)\\
&= e_{\check\chi}\bigl({\rm Nrd}_{\Q_\ell[G]}\big(1-  \varphi_p^{{\rm deg}(v_0)} : \QQ_\ell[G]\otimes_{\bz_\ell}T_\ell(A) \big)\\
&= e_{\chi}\bigl({\rm Nrd}_{\Q_\ell[G]}\big(1-  \varphi_p^{{\rm deg}(v_0)} : (\QQ_\ell[G]\otimes_{\bz_\ell}T_\ell(A))^\vee \big).
\end{align*}
Here the second equality follows from Remark~\ref{contra justification} and all others are clear. 
%The argument is essentially same for the case when $\ell = p$. {\color{red}why do we need to explicitly mention the case of $p$ since we just proved what we want for $\ell$ and earlier claimed that the equality identifying the expressions for $\ell$ and $p$ follows from Katz-Messing?} 

In addition, our assumption that $V'_w = \mathcal{A}(\mathfrak{m}_w)$ for places $w$ above $v_0$ implies  there are exact triangles in $D^{\rm perf}(\ZZ[G])$ of the form
%%
%\[
%\begin{cases} & {\rm SC}_{V^\dagger_L}(A,L/K) \to {\rm SC}_{V_L}(A, L/K) \to {\bigoplus}_{w | v_0} A(k(w))^*[-1] \to \\
%&R\Gamma(X, \mathcal{L}_L)^* \rightarrow R\Gamma(X,\mathcal{L}^\dagger_L)^* \to {\bigoplus}_{w | v_0} {\rm Lie}(A)(k(w))[0] \to, \end{cases}
%\]
%%
%
\[
\begin{cases} &{\rm SC}_{V^\dagger_L}(A,L/K) \to {\rm SC}_{V_L}(A, L/K) \to {\bigoplus}_{w | v_0} A(k(w))^*[-1] \to \\
&R\Gamma(X, \mathcal{L}_L^\dagger)^* \rightarrow R\Gamma(X, \mathcal{L}_L)^* \to {\bigoplus}_{w | v_0} {\rm Lie}(A)(k(w))^*[-1] \to. \end{cases}
\]
%where $\mathcal{L}_L$ and $\mathcal{L}^\dagger_L$ are the coherent sheaves associated to the families $V_L$ and $V_L^\dagger$ (so that $\mathcal{L}_{v_0} = {\rm Lie}(\cala)(\mathcal{O}_{v_0})$).

These triangles in turn imply that there are equalities in $K_0(\ZZ[G],\QQ[G])$
%%
%\[ \begin{cases} &\chi^{\rm BSD}_G(A,V_L) - \chi_G^{\rm BSD}(A,V^\dagger_L) = \chi_{G}({\bigoplus}_{w | v_0} A(k(w))^*[-1],0) \\
%&\chi_G^{\rm coh}(A,V^\dagger_L)  - \chi_G^{\rm coh}(A,V_L) = \chi_{\ZZ_p[G]}({\bigoplus}_{w | v_0} {\rm Lie}(A)(k(w))[-1],0).
%\end{cases} \]
%%
%
\[ \begin{cases} &\chi^{\rm BSD}_G(A,V_L^\dagger) - \chi_G^{\rm BSD}(A,V_L) = -\chi_{G}({\bigoplus}_{w | v_0} A(k(w))^\ast[-1],0)\\
&\chi_G^{\rm coh}(A,V^\dagger_L)  - \chi_G^{\rm coh}(A,V_L)  = -\chi_{\ZZ_p[G]}({\bigoplus}_{w | v_0} {\rm Lie}(A)(k(w))^\ast[-1],0).
\end{cases} \]
To prove the required equality (\ref{needed eq}) it is thus enough to show 
\begin{equation}\label{ell case} 
\delta_{G,\ell}({\rm Nrd}_{\mathbb{Q}_{\ell}[G]}(1- \varphi_p^{{\rm deg}(v_0)}: \big(\mathbb{Q}_\ell[G]\otimes_{\ZZ_\ell}T_{\ell}(A)\big)^\vee))\! =\! \chi_{\ZZ_\ell[G]}({\bigoplus}_{w | v_0} A(k(w))\{\ell\}^\ast[-1],0)
\end{equation}
for every prime $\ell\not= p$, and also that %%and in addition that
\begin{multline}\label{p case} 
\delta_{G,p}({\rm Nrd}_{\mathbb{Q}_p[G]}(1- (p^{-1}\varphi)^{{\rm deg}(v_0)} : \big(\mathbb{Q}_p[G]\otimes_{\QQ_p}H^0_{\rm crys}(k(v_0), \overline{D})\big)^\vee)) \\
 = \chi_{\ZZ_p[G]}({\bigoplus}_{w | v_0} A(k(w))\{p\}^\ast[-1],0) - \chi_{\ZZ_p[G]}({\bigoplus}_{w | v_0} {\rm Lie}(A)(k(w))^\ast[-1],0).
\end{multline}
Here, for each prime $q$, we write $\delta_{G,q}$ for the composite homomorphism
\[ \partial_{\ZZ_q[G],\QQ_q}\circ ({\rm Nrd}_{\QQ_q[G]})^{-1}: \zeta(\QQ_q[G])^\times \to K_0(\ZZ_q[G],\QQ_q[G]).\]
%(Note also that we have $\chi_{G}(Q^*[-1],0) = \chi_{G}(Q[-1],0)$ for any c-t finite $G$-module $Q$.)

Now, if $\ell \neq p$, then the complex $R\Gamma(k(v_0), T_{\ell}(A) \otimes \mathbb{Z}[G]) \cong {\bigoplus}_{w\mid v_0}R\Gamma(k(w), T_{\ell}(A))$ is acyclic outside degree one and has cohomology ${\bigoplus}_{w|v_0} A(k(w))\{\ell\}$ in that degree. This gives rise to a short exact sequence of
$\ZZ_\ell[G]$-modules
%%
%\[
%0 \to T_{\ell}(A) \otimes_{\mathbb{Z}_{\ell}} \mathbb{Z}_{\ell}[G]  \xrightarrow{1- \varphi_p^{deg(v_0)}} T_{\ell}(A) \otimes_{\mathbb{Z}_{\ell}} \mathbb{Z}_{\ell}[G] \rightarrow {\bigoplus}_{w|v_0} A(k(w))\{\ell\}^* \to 0
%\]
%
\[
0 \to \Hom_{\bz_\ell}(T_{\ell}(A),\bz_\ell[G])  \xrightarrow{1- \varphi_p^{\deg(v_0)}}\Hom_{\bz_\ell}(T_{\ell}(A),\bz_\ell[G])  \rightarrow {\bigoplus}_{w|v_0} A(k(w))^\ast\{\ell\} \to 0,
\]
which leads directly to the equality (\ref{ell case}).

We next note that, by Kato-Trihan \cite[5.14.6]{KT}, for each $w$ dividing $v_0$ the complex $A(k(w))\{p\}[-1]$ identifies with $R\Gamma(k(w), \mathcal{S}_{D_w})$, where $\mathcal{S}_{D_w}$ is the syntomic complex over $k(w)$ (obtained as a fiber of the syntomic complex over $U$), and hence that there is an exact triangle in $D^{\rm perf}(\ZZ_p[G])$ of the form
\begin{multline*}
{\bigoplus}_{w|v_0} A(k(w))\{p\}[-1] \to  \mathbb{Z}_p[G]\otimes_{\ZZ_p}R\Gamma_{\rm crys}(k(v_0)/\ZZ_p, D_{v_0}^0) \\ \xrightarrow{1-\varphi^{{\rm deg}(v_0)}} \mathbb{Z}_p[G]\otimes_{\ZZ_p}R\Gamma_{\rm crys}(k(v_0)/\ZZ_p, D_{v_0}) \rightarrow .\end{multline*}

There is also a natural exact triangle in $D^{\rm perf}(\ZZ_p[G])$
\begin{multline*} {\bigoplus}_{w | v_0} {\rm Lie}(A)(k(w))[-1] \to \mathbb{Z}_p[G]\otimes_{\ZZ_p}R\Gamma_{\rm crys}(k(v_0)/\ZZ_p, D_{v_0}^0) \\ \xrightarrow{1}  \mathbb{Z}_p[G]\otimes_{\ZZ_p} R\Gamma_{\rm crys}(k(v_0)/\ZZ_p, D_{v_0})  \to .
\end{multline*}

The required equality (\ref{p case}) now follows directly upon applying Proposition \ref{prop:triangle} with $R = \ZZ_p[G]$ and the triangles in (\ref{triangles}) taken to be the images of the above two triangles under the exact linear duality functor $R\Hom_{\ZZ_p}(-,\ZZ_p)$ on $D^{\rm perf}(\ZZ_p[G])$.
(These triangles are easily seen to satisfy the hypotheses of Proposition \ref{prop:triangle} since the modules $A(k(w))\{p\}$ and ${\rm Lie}(A)(k(w))$ are both finite.)   \end{proof}

\begin{remark}{\em The results of Propositions \ref{V independent} and \ref{U independent} will play a key role in later arguments. In Proposition \ref{group independent} below we will also establish a further consistency property of Conjecture \ref{conj:ebsd} with respect to changes of field extension $L/K$. }\end{remark}

\subsection{A reformulation}\label{sect:reform} In this section we establish a useful reformulation of the equality in Conjecture \ref{conj:ebsd}(iii).

In \cite[p. 509]{schneider} Schneider shows that the pairing $h^{\rm NT}_{A,L}$ can be factored in the form
\begin{equation} \label{height triv rat} h^{\rm NT}_{A,L} = {\rm log}(p)\cdot h_{A,L}\end{equation}
for a certain non-degenerate skew-symmetric bilinear form $h_{A,L}: A(L)\times A^t(L) \to \QQ$.

We write
\begin{equation*}
h^{{\rm det}}_{A,L}: {\rm Det}_{\QQ[G]}(\QQ \otimes_{\ZZ} {\rm SC}_{V_L}(A, L/K)) \cong {\rm Det}_{\QQ[G]}(0).
\end{equation*}
for the isomorphism induced by $h_{A.L}$ and then define an element of $K_0(\ZZ[G],\Q[G])$ by setting
\[ \chi^{\rm BSD}_{G,\QQ}(A,V_L) := \chi_{G}({\rm SC}_{V_L}(A,L/K),h^{\rm det}_{A,L}).\]

For each $\chi$ in ${\rm Ir}(G)$ we define a function of the $t := p^{-s}$ by setting
\[ Z_U(A,\chi,t) := L_U(A,\chi,s)\]
and %write $Z^*_U(A,\chi,p^{-1})$ for the leading term of this function at $t=p^{-1}$.
 normalise its leading term at $t = p^{-1}$ as follows
\begin{equation}\label{eq:leading-term-Z}
Z^*_U(A,\chi,p^{-1}) := \lim_{t\to p^{-1}} (1-pt)^{-r_{\rm an}(\chi)}\cdot Z_U(A,\chi,t).
\end{equation}

\begin{proposition}\label{bsd-schneider} The following claims are valid.

\begin{itemize}
\item[(i)] There exists a unique element $Z^*_U(A_{L/K},p^{-1})$ of $K_1(\QQ[G])$ with the property that
${\rm Nrd}_{\QQ[G]}(Z^*_U(A_{L/K},p^{-1}))_\chi = Z^*_U(A,\chi,p^{-1})$ for all $\chi$ in ${\rm Ir}(G)$.

\item[(ii)] If claims (i) and (ii) of Conjecture \ref{conj:ebsd} are valid, then the equality in claim (iii) of Conjecture \ref{conj:ebsd} is valid if and only if in $K_0(\ZZ[G],\QQ[G])$ one has
\[ \partial_{G}(Z^*_{U}(A_{L/K},p^{-1})) = \chi^{\rm BSD}_{G,\QQ}(A,V_L) - \chi_G^{\rm coh}(A,V_L) + \chi^{\rm sgn}_G(A).\]
\end{itemize}
\end{proposition}

\begin{proof} The argument of Proposition \ref{prop:positive} implies, via the equalities (\ref{deligne}) and (\ref{first inv}), that one has $\omega(Z^*_U(A,\chi,q^{-1}))= Z^*_U(A,\omega\circ\chi,q^{-1})$ for all $\chi$ in ${\rm Ir}(G)$ and all automorphisms $\omega$ of $\CC$, and hence that the element $(Z^*_U(A,\chi,p^{-1}))_{\chi}$ of $\zeta(\CC[G])^\times = {\prod}_{\chi \in {\rm Ir}(G)}\CC^\times$ belongs to the subgroup $\zeta(\QQ[G])^\times$.

Given this, claim (i) follows from the Hasse-Schilling-Maass Norm Theorem and the fact the same proof also shows $Z^*_U(A,\chi,p^{-1})$ is a strictly positive real number for $\chi$ in ${\rm Ir}^{\rm s}(G)$.

To prove claim (ii) we set $r_\chi := r_{\rm an}(\chi)$ and $r'_\chi := r_{\rm alg}(\chi)$. Then the order of vanishing of $Z_{U}(A,\chi,t)$ at $t=p^{-1}$ is equal to $r_\chi$ and hence, since the leading term of $(1-p^{1-s})^{r_\chi}$ at $s=1$ is equal to $({\rm log}(p))^{r_\chi}$, it follows that 
\[Z^*_U(A,\chi,p^{-1}) = ({\rm log}(p))^{-r_\chi}\cdot L^*_U(A,\chi,1).\]

Thus, writing $\varepsilon_{L/K}$ for the unique element of $K_1(\RR[G])$ with 
\[ {\rm Nrd}_{\RR[G]}(\varepsilon_{L/K})_{\chi} = ({\rm log}(p))^{-r_\chi}\]
for all $\chi$ in ${\rm Ir}(G)$, one has
\[ Z^*_{U}(A_{L/K},p^{-1}) = \varepsilon_{L/K} \cdot L^*_{U}(A_{L/K},1).\]

On the other hand, the equality (\ref{height triv rat}) implies that
\[ \chi^{\rm BSD}_{G,\QQ}(A,V_L) = \chi_G^{\rm BSD}(A,V_L) + \partial_{G}(\varepsilon'_{L/K})\]
where $\varepsilon'_{L/K}$ is the element of $K_1(\RR[G])$ that is represented by the automorphism of the $\RR[G]$-module $ \RR\otimes_\ZZ H^0({\rm SC}_{V_L}(A,L/K)) = \RR\otimes_\ZZ A^t(L)$ given by multiplication by ${\rm log}(p)^{-1}$.

Given the last two displayed formulas, the claimed equivalence will follow if one can show that the assumed validity of Conjecture \ref{conj:ebsd}(i) implies $\varepsilon'_{L/K} = \varepsilon_{L/K}$. But this is true since, for every $\chi$ in ${\rm Ir}(G)$, one has
\begin{align*} {\rm Nrd}_{\RR[G]}(\varepsilon'_{L/K})_\chi =&\, {\rm det}_\CC\left({\rm log}(p)^{-1}\mid \Hom_{\CC[G]}(V_\chi, \CC\otimes_\ZZ A^t(L))\right)\\
=&\, ({\rm log}(p))^{-r'_\chi}\\
 =&\, ({\rm log}(p))^{-r_\chi}\\
 =&\, {\rm Nrd}_{\RR[G]}(\varepsilon_{L/K})_\chi .\end{align*}
Here the first equality follows directly from an explicit computation of reduced norm, the second from the fact $r_\chi'$ is (by its definition) equal to ${\rm dim}_\CC\bigl( \Hom_{\CC[G]}(V_\chi, \CC\otimes_\ZZ A^t(L))\bigr)$, the third from the assumption that Conjecture \ref{conj:ebsd}(i) is valid (and hence $r_\chi' = r_\chi$) and the last directly from the explicit definition of  $\varepsilon_{L/K}$ given above. \end{proof}

\section{Syntomic cohomology}\label{sec:syncom} In this section we recall relevant facts concerning the complexes of syntomic cohomology with compact supports that are constructed by Kato and Trihan in \cite{KT}.

At the outset we fix a finite Galois extension $K'$ of $K$ over which $A\otimes_KK'$ is semistable at all places, write $L'$ for the compositum of $L$ and $K'$ and set $G' := \Gal(L'/K)$. Taking advantage of Proposition \ref{U independent} we shrink $U$ (if necessary) in order to assume that no point on $U$ ramifies in $L'/K$.

We also fix a Galois extension of fields $F'/F$ with 
\[ K'\subseteq F \subseteq F' \subseteq L'\]
and set $Q := \Gal(F'/F)$. (Whilst the use of this auxiliary extension $F'/F$ adds a degree of notational complexity to the results in this section, it provides results that we can then directly apply in the proof of Theorem \ref{thm:thm1} given in \S\ref{sect:proof}.)

Then, with $N$ denoting either $F'$ or $F$ we set $A_N:= A\otimes_KN$ and write $X_N$ and $U_N$ for the integral closures of $X$ and $U$ in $N$ and $\mathcal{A}_N/X_N$ for the N\'{e}ron model of $A_N$ over $N$. Let $\pi_N:X_N\to X$ denote the natural map. Let $X^{\sharp}_{F'}$ be the log scheme with underlying scheme $X_{F'}$ equipped with the log structure associated to the divisor $X_{F'}-U_{F'}$, and we abbreviate to $\mathcal{O}_{\langle N\rangle}$ the structure sheaf $\mathcal{O}_{(X_N)^\sharp/\ZZ_p}$ for the small \'etale log crystalline topos $((X_{N})^{\sharp}/\mathbb{Z}_p)_{\rm crys}$.

Since $A$ is semistable over $N$, the construction in \cite[\S4.8]{KT} gives a Dieudonn\'e crystal 
\begin{equation}\label{crystal notation} D_N := D_{\rm log}(A_{N})\end{equation}
on $((X_{N})^{\sharp}/\mathbb{Z}_p)_{\rm crys}$. We then write $D_N^0$ for the kernel of the surjective morphism of sheaves  $D_N \rightarrow i_{(X_{N})^{\sharp}/\mathbb{Z}_p, *}({\rm Lie}(D_N))$  in  $((X_{N})^{\sharp}/\mathbb{Z}_p)_{\rm crys}$  described at the beginning of \cite[\S5.5]{KT}.

%Let $\pi_{F'/F}:X_{F'}\to X_{F}$ and $\pi_{F'}:X_{F'}\to X$ denote the natural maps, and we set $\pi_{F'} = \pi_{F}\circ\pi_{F'/F}: X_{F'}\to X$.
%
We fix a $\Gal(L'/K)$-equivariant $\mathcal{O}_X$-submodule $\mathcal{L}_{L'}$ of $\pi_{L',*}{\rm Lie}(D_{L'})$ that is associated to $(W'_w)_{w\notin U_{L'}}$ following \S\ref{zariski section}, and set $\mathcal{L}_{F'}:= (\mathcal{L}_{L'})^{\Gal(L'/F')}$, which is a $Q$-equivariant $\mathcal{O}_{X}$-submodule  of $\pi_{F',*}{\rm Lie}(D_{F'})$. For simplicity, we write 
\begin{equation}\label{cL' def} \mathcal{L}':=\mathcal{L}_{F'}.\end{equation}
(In the intended setting, we will assume that $\mathcal{L}_{L'}$, and hence $\mathcal{L}'$, satisfies the conclusion of Lemma~\ref{zar lemma}. As noted in Remark~\ref{KT issue}, it may not be possible to arrange $\mathcal{L}'$ to be the pushforward of a vector bundle on $X_{F'}$ or  $X_{F}$.) 

We furthermore assume that for some positive integers $n(w)$ for each place $w$ of $X_{F'}$ not in $U_{F'}$, we have 
\[ {\rm Lie}(D_{F'})(\mathfrak{m}_w^{2n(w)}) \subset W'_w \subset {\rm Lie}(D_{F'})(\mathfrak{m}^{n(w)}).\]
(This can be arranged by shrinking $W'_w$ if necessary.) We set  
\[E := {\sum}_{w\notin U_{F'}}n(w)w,\]
which turns out to be a $Q$-stable divisor of $X_{F'}$ since $n(w)=n(w')$ if $w$ and $w'$ are above the same place in $X_F$ by construction. Then by the condition on $(W'_w)_{w}$ we have
\[ \pi_{F',*}{\rm Lie}(D_{F'}))(-2E) \subset \mathcal{L}' \subset \pi_{F',*}{\rm Lie}(D_{F'}))(-E).\]

Let us write $\mathcal{O}_{\langle F'\rangle}(-E)$ for the crystal on $((X_{F'})^\sharp/\ZZ_p)_{\rm crys}$ that is obtained as the twist of $\mathcal{O}_{\langle F'\rangle}$ by $-E$ and then set $D(-E)_{F'} :=
D_{F'}\otimes_{\mathcal{O}_{\langle F' \rangle}}\mathcal{O}_{\langle F'\rangle}(-E)$. 
%Now we define a $Q$-equivariant $\mathcal{O}_{\langle F\rangle}$-submodule $D(-E)_{F'}^{(\mathcal{L}')}\subset D(-E)_{F'}$ that satisfies the following short exact sequence of $\mathcal{O}_{\langle F\rangle}$-modules:
%\[
%\xymatrix@1{
%0 \ar[r] & D(-E)_{F'}^{(\mathcal{L}')} \ar[r] & D(-E)_{F'} \ar[r] & {\rm Lie}(\mathcal{A}_{F'})(-E)/\mathcal{L}' \ar[r] &0.
%}.
%\]
%(In \cite[\S5.12]{KT} Kato and Trihan assumed  that $\mathcal{L}'$ is an $\mathcal{O}_{X_{F'}}$-submodule of $ {\rm Lie}(D_{F'}))(-E)$, but we may not be able to find such $\mathcal{L}'$ that satisfies Lemma~\ref{zar lemma}.)

By \cite[5.5.2]{KT}, we have a distinguished triangle of $Q$-equivariant (small) \'etale sheaves on $X_{F'}$:
\[Ru'_*D(-E)_{F'}^{(0)} \to Ru'_*D(-E)_{F'}\to {\rm Lie}(D_{F'})(-E) \to,\]
where $u': ((X_{F'})^{\sharp}/\mathbb{Z}_p)_\crys \rightarrow X_{F',\et}$ is the natural morphism of topoi. 

Now, we would like to modify $Ru'_*D(-E)_{F'}$ using the $Q$-equivariant $\mathcal{O}_X$-submodule $\mathcal{L}'$ of $\pi_{F',*}{\rm Lie}(D_{F'})(-E)$, and for this to make sense we need to apply the pushforward $\pi_{F',*}$ to the above distinguished triangle. 
To alleviate the notation, let us write
\[\Ru{F'} := \pi_{F',*}Ru'_*\]
sending a crystalline sheaf of $\mathcal{O}_{\langle F'\rangle}$-modules to a complex of \'etale sheaves on $X$ (viewed in a suitable derived category).

We can now define a complex $\Ru{F'}D(-E)_{F'}^{(\mathcal{L}')}$ viewed in the derived category of $Q$-equivariant \'etale sheaves on $X$ so that it fits in the following distinguished triangle
\[
\Ru{F'}D(-E)_{F'}^{(\mathcal{L}')} \to \Ru{F'}D(-E)_{F'} \to \pi_{F',*}{\rm Lie}(D_{F'})(-E)/\mathcal{L}' \to. \
\] 
(For technical reasons, we directly define $\Ru{F'}D(-E)_{F'}^{(\mathcal{L}')}$ via the distinguished triangle above without defining the crystalline subsheaf $D(-E)_{F'}^{(\mathcal{L}')}$ of $D(-E)_F$. Note that $D(-E)_{F'}^{(\mathcal{L}')}$ can be defined if $\mathcal{L}'$ is the pushforward of a vector bundle on $X_{F'}$, in which case the above construction recovers $\Ru{F'}(D(-E)_{F'}^{(\mathcal{L}')})$; \emph{cf.} \cite[\S5.12]{KT}.)

Following \cite[\S5.12]{KT} there are canonical morphisms of complexes of \'etale sheaves on $X$
\begin{multline*}
\Ru{F'}D(-E)_{F'}^{(0)} \xrightarrow{{\bf 1}} \Ru{F'}D(-E)_{F'}^{(\mathcal{L}')}\,\,\text{ and }\,\, 
\Ru{F'}D(-E)_{F'}^{(0)} \xrightarrow{\varphi} \Ru{F'}D(-E)_{F'}^{(\mathcal{L}')},
\end{multline*}
(In fact, since all the above objects can be explicitly represented by choosing \emph{good embeddings} locally, the argument in \cite[\S5.12]{KT} can be directly applied to these complexes of \'etale sheaves instead of crystalline sheaves on $((X_{F'})^\sharp/\ZZ_p)_{\rm crys}$.

%We can regard $Ru'_*(D(-E)_{F'}^{(0)})$ and $Ru'_*(D(-E)_{F'}^{(\mathcal{L}')})$ as objects in the derived category of $Q$-equivariant \'etale $\ZZ_p$-sheaves as follows. Since $\mathcal{A}_{F'}$ is defined over $X_F$ and $E$ is a $Q$-stable divisor, $Ru'_*(D(-E)_{F'})$ comes equipped with natural $Q$-action. Now, since $Ru'_*(D(-E)_{F'}^{(\mathcal{L}')})$ fits into the following natural distinguished triangle
%\[
%\xymatrix@1{
%Ru'_*(D(-E)_{F'}^{(\mathcal{L}')}) \ar[r] & 
%Ru'_*(D(-E)_{F'}) \ar[r]^-{(*)} &
%{\rm Lie}(\mathcal{A}_{F'})(-E)/\mathcal{L}' \ar[r]& (+1),
%}\]
%where ${\rm Lie}(\mathcal{A}_{F'})(-E)/\mathcal{L}'$ is a $Q$-equivariant coherent $\mathcal{O}_{X_{F}}$-module viewed as an \'etale sheaf. Then we can represent $Ru'_*(D(-E)_{F'}^{(\mathcal{L}')})$ as a complex of $Q$-equivariant \'etale $\ZZ_p$-sheaves via the mapping fibre of $(*)$ in the diagram (with respect to the ``natural'' choice of complex representing $Ru'_*(D(-E)_{F'})$). The same (and slightly simpler) argument also shows the same for $Ru'_*(D(-E)_{F'}^{(0)})$.

 They then define the syntomic complex with compact supports $\mathcal{S}_{D_{F'}^{(E,\mathcal{L}')}}$ to be the mapping fibre of the morphism
\begin{equation}\label{compact sheaf}
  \Ru{F'}D(-E)_{F'}^{(0)} \xrightarrow{{\bf 1}-\varphi} \Ru{F'}D(-E)_{F'}^{(\mathcal{L}')},
\end{equation}
which is an object in the derived category of $Q$-equivariant \'etale $\ZZ_p$-sheaves on $X$.

If furthermore $F'/M$ is Galois for some intermediate field $M$ of $F/K$, then by choosing $E$ to be $\Gal(F'/M)$-stable we may give a natural $\Gal(F'/M)$-action on $\Ru{F'}D(-E)_{F'}^{(0)}$, $\Ru{F'}D(-E)_{F'}^{(\mathcal{L}')}$, and $\mathcal{S}_{D_{F'}^{(E,\mathcal{L}')}}$. Recall that $\mathcal{L'} = (\mathcal{L}_{L'})^{\Gal(L'/F')}$ for some $\Gal(L'/K)$-equivariant $\mathcal{O}$-submodule of $\pi_{L',*}{\rm Lie}(D_{L'})$, so the $Q$-action on $\mathcal L'$ naturally extends to the action of $\Gal(F'/M)$.

If we have $\mathcal{L}' = \pi_{F',*}\widetilde{\mathcal{L}}'$ for some $Q$-equivariant $\mathcal O_{X_{F'}}$-submodule $\widetilde{\mathcal{L}}'$ of ${\rm Lie}(\mathcal{A}_{F'})$, then the above constructions can be carried out over $X'$ as in \cite[\S5.12]{KT}. To explain, we can define an $\mathcal{O}_{\langle F'\rangle}$-submodule $D(-E)_{F'}^{(\widetilde{\mathcal{L}}')}$ of $D(-E)_{F'}$, and define $\widetilde{\mathcal{S}}_{D_{F'}^{(E,\widetilde{\mathcal{L}}')}}$ to be the mapping fibre of 
\begin{equation}\label{compact sheaf KT}
 Ru'_*D(-E)_{F'}^{(0)}\xrightarrow{{\bf 1}-\varphi}Ru'_*D(-E)_{F'}^{(\widetilde{\mathcal{L}}')}.
 \end{equation} 
(See loc.~cit. for details.) Furthermore, we have a $Q$-equivariant quasi-isomorphism $\mathcal{S}_{D_{F'}^{(E,\mathcal{L}')}} = \pi_{F',*}\widetilde{\mathcal{S}}_{D_{F'}^{(E,\widetilde{\mathcal{L}}')}}$. On the other hand, in the presence of wild ramification it seems difficult to find $\mathcal{L}'$ coming from a $Q$-equivariant $\mathcal{O}_{X_{F'}}$-submodule that satisfies the conclusion of Lemma~\ref{zar lemma}.

\begin{lemma}\label{lem:KT5.13}
Let $E$, $\mathcal{L}'$ and $(W'_w)_{w\notin U_{F'}}$ be as above, and write $E = {\sum}_{w\notin U_{F'}}n(w)w$.
 For each $w\notin U_{F'}$ write $V'_w$ for the unique subgroup of $\mathcal{A}_{F'}(\mathcal{O}_w)$ with
$\mathcal{A}_{F'}(m_w^{2n(w)}) \subset V'_w \subset \mathcal{A}_{F'}(m_w^{n(w)})$ and whose image in $\mathcal{A}_{F'}(m_w^{n(w)})/\mathcal{A}_{F'}(m_w^{2n(w)}) \cong {\rm Lie}(\mathcal{A}_{F'})(m_w^{n(w)})/{\rm Lie}(\mathcal{A}_{F'}(m_w^{2n(w)}))$ coincides with the image of $W'_w$. Write $V'_{F'}$ for the family $(V'_w)_w$.

Then there are natural isomorphisms in $D(\ZZ_p[Q])$ of the form
\begin{equation} \label{eq:syn}
R\Gamma(X, \mathcal{S}_{D_{F'}^{(E,\mathcal{L}')}}\otimes^{\mathbb{L}} \Q_p/\Z_p) \cong
R\Gamma_{{\rm ar}, V'_{F'}}(U_{F'},\mathcal{A}_{\rm tor})_p.
\end{equation}
In addition, if $M$ is any intermediate field of $F/K$ over which $F'$ is Galois 
% and each of $E$, 
and $E$ is chosen to be $\Gal(F'/M)$-equivariant, then the above isomorphism is well-defined in $D(\ZZ_p[\Gal(F'/M)])$.
\end{lemma}
This lemma is a generalisation of  \cite[Prop. 5.13]{KT} in that the isomorphism~\eqref{eq:syn} is proven to be Galois equivariant and $\mathcal{L}'$ is not required to come from a vector bundle over $X_{F'}$.

\begin{proof}

Using the definition of $R\Gamma_{{\rm ar}, V_{F'}}(U_{F'},\mathcal{A}_{\rm tor})_p$ \eqref{fibre morphism} and \cite[Th.~1.1]{trihanvauclair1}, one can reduce the isomorphism \eqref{eq:syn} to the following local statement: \emph{For any $v\in X\setminus U$, we have a natural isomorphism
\begin{equation}\label{eq:synloc}
R\Gamma(\Spec \mathcal{O}_v,\mathcal{S}_{D_{F'}^{(E,\mathcal{L}')}})\cong {\prod}_{w|v}V'_w[-1]
\end{equation}
equivariant for the $Q$-action (respectively, for the $\Gal(F'/M)$-action if $F'/M$ is Galois for some intermediate extension $M$ of $F/K$).} 

Note that this local claim is a slight generalisation of  \cite[Lem. 5.14]{KT} in that the isomorphism \eqref{eq:synloc} is required to be Galois equivariant and $W'_w$ is not required to be an $\mathcal{O}_w$-module.

It remains to verify the local claim. Observe that for fixed $D$ the restriction of $\mathcal{S}_{D_{F'}^{(E,\mathcal{L}')}}$ to $\Spec \mathcal{O}_v$ only depends on $n(w)$ and $W'_w$ for $w|v$, and $n(w)$ is independent of $w|v$. So let us write
 \[
\mathcal{S}^{n,(W'_w)}_{D,v}:= \left . \mathcal{S}_{D_{F'}^{(E,\mathcal{L}')}} \right |_{\Spec O_v}
\]
where $n=n(w)$ for any $w|v$. To simplify the notation, for any positive integer $n$ we write 
\[W^{(n)\prime}_w := {\rm Lie}(\mathcal{A}_{F'})(\mathfrak{m}_{w}^n). \]
Note that the choice $(W^{(n(w))\prime}_w)_{w\notin U_{F'}}$ corresponds to  $\pi_{F',*}{\rm Lie}(\mathcal{A}_{F'})(-E)$, which contains $\mathcal{L}'$.

Let us first show that the local claim (\ref{eq:synloc}) is implied by the special case for $W'_w = W_w^{(n)\prime}$. %Assume that we have $W_w^{(2n)\prime}\subseteq W_w'\subseteq W_w^{(n)\prime} $ and $n=n(w)$. 
For this, we construct a distinguished triangle in the suitable derived category of equivariant \'etale $\ZZ_p$-sheaves on $\Spec \mathcal{O}_v$
\begin{equation}\label{eq:synloc:ind1}
\xymatrix@1{\mathcal{S}^{n,(W'_w)}_{D,v} \ar[r]&
\mathcal{S}^{n,(W^{(n)\prime}_v)}_{D,v} \ar[r] &
{\prod}_{w|v}(W^{(n)\prime}_w/W'_w)[-1]\ar[r] &
 .}
\end{equation}
Indeed, this can be obtained from the following commutative diagram where each row is a distinguished triangle
\[
\xymatrix{
\mathcal{S}^{n,(W'_w)}_{D,v} \ar[r]\ar[d]&
\Ru{F'}D(-E)^{(0)}|_{\mathcal{O}_v} \ar[r]^-{{\bf 1}-\varphi} \ar[d]^-{=}&
\Ru{F'}D(-E)^{(\mathcal{L}')}|_{\mathcal{O}_v} \ar[d] \ar[r]& \\
\mathcal{S}^{n,(W_w^{(n)\prime})}_{D,v} \ar[r]&
\Ru{F'} D(-E)^{(0)}|_{\mathcal{O}_v} \ar[r]^-{{\bf 1}-\varphi}&
\Ru{F'}D(-E)|_{\mathcal{O}_v} \ar[r]& ,}
\]
together with the fact that the mapping cone of the rightmost vertical arrow is isomorphic to $W^{(n)\prime}_w/W'_w$.

Let $\widehat{\mathcal{A}}_{F'}(\mathfrak{m}_v^n)\subset \mathcal{A}_{F'}(\mathcal{O}_w)$ denote the kernel of reduction modulo $\mathfrak{m}_w^n$. Then by \eqref{eq:synloc:ind1} and the natural isomorphism 
\begin{equation}\label{eq:synloc:ind2}
W^{(n)\prime}_w/W^{(n+1)\prime}_w\cong\widehat{\mathcal{A}}_{F'}(\mathfrak{m}_v^n)/\widehat{\mathcal{A}}_{F'}(\mathfrak{m}_v^{n+1}),
\end{equation}
the general case of the local claim is reduced to obtaining the Galois equivariant isomorphism \eqref{eq:synloc} when $W'_w = W^{(n(w))\prime}_w$ and $V'_w =  \widehat{\mathcal{A}}_{F'}(\mathfrak{m}_v^{n(w)})$.

We have thus reduced the proof of the lemma to the case when $\mathcal{L}' = \pi_{F',*}{\rm Lie}(\mathcal{A}_{F'})(-E)$. We will proceed by induction, for which it is convenient to allow $\mathcal{L}' = \pi_{F',*}{\rm Lie}(\mathcal{A}_{F'})(-E')$ where $E'$ is a $Q$-equivariant divisor such that $E'-E$ and $2E-E'$ are either effective or trivial. Since $\mathcal{L'}$ is the the pushforward of a $Q$-equivariant $\mathcal{O}_{X_{F'}}$-module $\widetilde{\mathcal{L}}':={\rm Lie}(\mathcal{A}_{F'})(-E')$, we also have a `syntomic complex' $\widetilde{\mathcal{S}}_{D_{F'}^{(E,\widetilde{\mathcal{L}}'})}$ over $X_{F'}$, constructed as the mapping fibre of the map (\ref{compact sheaf KT}). For any $w\in X_{F'}\setminus U_{F'}$, let us write
\[
\widetilde{\mathcal{S}}^{n,W^{(m)\prime}}_{D,w}:=\left. \widetilde{\mathcal{S}}_{D_{F'}^{(E,\widetilde{\mathcal{L}}')}} \right |_{\Spec \mathcal{O}_w},
\]
where $n=n(w)$ and $m$ are the coefficients of $w$ in $E$ and $E'$, respectively. As we have $\mathcal{S}_{D_{F'}^{(E,\mathcal{L}')}} = \pi_{F',*}\widetilde{\mathcal{S}}_{D_{F'}^{(E,\widetilde{\mathcal{L}}')}}$ the left hand side of (\ref{eq:synloc}) also decomposes in terms of $\widetilde{\mathcal{S}}^{n,W^{(m)\prime}}_{D,w}$. Therefore, to complete the proof, it suffices to show that for any positive integer $n$ and $w\in X_{F'}\setminus U_{F'}$  we have a natural isomorphism
\begin{equation}\label{eq:synloc2}
R\Gamma(\Spec \mathcal{O}_w,\widetilde{\mathcal{S}}^{n,W^{(n)\prime}}_{D,w})\cong \widehat{\mathcal{A}}(\mathfrak{m}_w^{n})
\end{equation}
equivariant for the $Q_w$-action (respectively, for the $\Gal(F'_w/M_v)$-action where $v$ is the place under $w$ if $F'/M$ is Galois for some intermediate extension $M$ of $F/K$).

Let us prove (\ref{eq:synloc2}) by induction on $n$. 
If $n=1$ then the isomorphism  \eqref{eq:synloc2} can be deduced by inspecting the distinguished triangle \eqref{fibre morphism} using Theorems 1.1 and 1.2 in \cite{trihanvauclair1}, where the Galois equivariance follows since the comparison maps in loc.~cit. are constructed naturally. Although the results were obtained only for $p>2$ in loc.~cit., there is an alternative proof that works for any $p$ via the prismatic Dieudonn\'e theory \cite{ALB23}. To give further details, \cite[Th.~1.2]{trihanvauclair1} holds for $p=2$ if \cite[Th.~1.1]{trihanvauclair1} does, and (the projective limit  of) \cite[Th.~1.1]{trihanvauclair1} can be deduced from \cite[Prop. ~4.83 and  Rem.~4.85]{ALB23} if we show that the prismatic and crystalline constructions of the syntomic complex in  \cite{ALB23} and \cite{trihanvauclair1} coincide. For this we may (and do) pass to some complete intersection semiperfect ring by $p$-power root extraction to represent the two syntomic complexes by explicit two-term complexes of modules, and the desired isomorphism follows from the comparison between crystalline and prismatic Dieudonn\'e theory over quasisyntomic bases in characteristic $p$ \cite[Th.~4.44]{ALB23} as well as the comparison of the Hodge and Nygaard filtrations.   (The Nygaard filtration for prismatic Dieudonn\'e crystals is defined at the beginning of \cite[\S4.8]{ALB23}, and the claimed compatibility with filtrations can be read off from the proof of \cite[Th.~4.44]{ALB23} using \cite[Lem.~4.40 and Lem.~4.43]{ALB23}. Note also that in \cite[Rem.~4.85]{ALB23} we have $\tilde\xi=p$ for characteristic $p$ base rings, so the comparison in \cite[Th.~4.44]{ALB23} is compatible with the divided Frobenius maps.) 

Assume that we have a natural isomorphism \eqref{eq:synloc2} for $n$, and we shall deduce \eqref{eq:synloc2} for $n+1$. By \eqref{eq:synloc:ind1} applied to $W'_w=W_w^{(n+1)\prime}$, we get
\begin{equation}\label{eq:synloc:ind3}
\xymatrix@1{\widetilde{\mathcal{S}}^{n,W^{(n+1)\prime}}_{D,w} \ar[r]&
\widetilde{\mathcal{S}}^{n,W^{(n)\prime}}_{D,w} \ar[r] &
{\rm Lie}(\mathcal{A}_{F'}) \otimes (\mathfrak{m}_w^{n}/\mathfrak{m}_w^{n+1})[-1]\ar[r] &
 .}
\end{equation}
%Indeed, this can be obtained from chasing the following commutative diagram where each row is a distinguished triangle:
%\[
%\xymatrix{
%\mathcal{S}^{n,W^{(n+1)\prime}}_{D,w} \ar[r]\ar[d]&
%Ru_*(D(-nw))_{\mathcal{O}_w}^{(0)} \ar[r]^-{{\bf 1}-\varphi} \ar[d]^-{=}&
%Ru_*(D(-nw))_{\mathcal{O}_w}^{(W_w^{(n+1)\prime})} \ar[d] \ar[r]&(+1) \\
%\mathcal{S}^{n,W^{(n)\prime}}_{D,w} \ar[r]&
%Ru_*(D(-nw))_{\mathcal{O}_w}^{(0)} \ar[r]^-{{\bf 1}-\varphi}&
%Ru_*(D(-nw))_{\mathcal{O}_w}^{(W_w^{(n)\prime})} \ar[r]&(+1)
%},
%\]
%and observing that the mapping cone of the rightmost vertical arrow is isomorphic to ${\rm Lie}(\mathcal{A}_{F'}) \otimes (\mathfrak{m}_w^{n}/\mathfrak{m}_w^{n+1})$.

We now claim that the  map 
\begin{equation}\label{eq:synloc:ind4}
\widetilde{\mathcal{S}}^{n+1,W^{(n+1)\prime}_w}_{D,w} \xrightarrow{\sim} \widetilde{\mathcal{S}}^{n,W^{(n+1)\prime}_w}_{D,w},
\end{equation}
induced by the natural map $D(-(n+1)w) \to D(-nw)$, is an isomorphism. To verify this assertion we may ignore the Galois action, which enables us to represent the syntomic complexes explicitly following \cite[5.14.1]{KT}. 

Let $S^{\sharp}$ be $\Spec\mathcal{O}_w$ equipped with the divisorial log structure for the closed point.
Choose an isomorphism $\mathcal{O}_w \cong k_w[[T]]$, and write $P^{\sharp}$ to be $\Spec W(k_w)[[T]]$ equipped with the divisorial log structure given by the ideal $(T)$. Then the natural closed immersion $S^{\sharp}\hookrightarrow P^{\sharp}$ is a \emph{good embedding} in the sense of \cite[\S~5.6]{KT}. Let $\sigma$ be the lift of Frobenius on $\mathcal{O}_{\widehat P}=W(k_w)[[T]]$ given by $\sigma(T):=T^p$.

For any  integers $n,m$ with $0<n\leqslant m <2n$, we can represent $\widetilde{\mathcal{S}}^{n,W^{(m)\prime}_w}_{D,w}$  as the total complex of the following double complex

\[
\xymatrix{
T^{n} D^{(0)}_{(S^{\sharp},P^{\sharp})} \ar[rr]^-{1-p^{-1}\Fr{\bf1}} \ar[d] &&
T^m D_{(S^{\sharp},P^{\sharp})}+ T^n D^{(0)}_{(S^{\sharp},P^{\sharp})} \ar[d]\\
\mathcal{O}_{\widehat P} \frac{dt}{t}\otimes T^{n} D_{(S^{\sharp},P^{\sharp})} \ar[rr]^{1-p^{-1}\sigma\otimes\Fr{\bf1}} &&
\mathcal{O}_{\widehat P} \frac{dt}{t}\otimes T^n D_{(S^{\sharp},P^{\sharp})}
}.
\]
This shows that the mapping cone of \eqref{eq:synloc:ind4} is quasi-isomorphic to the the total complex of the following double complex
\[\xymatrix{
\frac{T^{n} D^{(0)}_{(S^{\sharp},P^{\sharp})}}{T^{n+1} D^{(0)}_{(S^{\sharp},P^{\sharp})}} \ar[rrr]^-{1-p^{-1}\Fr{\bf1}}  \ar[d] &&&
\frac{T^{n+1} D_{(S^{\sharp},P^{\sharp})}+ T^n D^{(0)}_{(S^{\sharp},P^{\sharp})}}{T^{n+1} D_{(S^{\sharp},P^{\sharp})}} \ar[d]\\
\mathcal{O}_{\widehat P}\frac{dt}{t}\otimes\frac{T^{n} D_{(S^{\sharp},P^{\sharp})}}{T^{n+1} D_{(S^{\sharp},P^{\sharp})}} \ar[rrr]^-{1-p^{-1}\sigma\otimes\Fr{\bf1}} &&& 
\mathcal{O}_{\widehat P}\frac{dt}{t}\otimes\frac{T^{n} D_{(S^{\sharp},P^{\sharp})}}{T^{n+1} D_{(S^{\sharp},P^{\sharp})}}
}.
\]
It suffices to show that both horizontal maps are isomorphisms. Indeed, since $p^{-1}\Fr{\bf 1}$ takes $T^nD^{(0)}_{(S^{\sharp},P^{\sharp})}$ into $T^{np}D_{(S^{\sharp},P^{\sharp})}$, the top horizontal map coincides with the map induced by the natural inclusion ${\bf 1}$, which is an isomorphism since $T^{n+1} D^{(0)}_{(S^{\sharp},P^{\sharp})} = T^{n+1} D_{(S^{\sharp},P^{\sharp})} \cap T^n D^{(0)}_{(S^{\sharp},P^{\sharp})}$. Similarly, the bottom horizontal map coincides with the identity map.

Now combining \eqref{eq:synloc:ind2}, \eqref{eq:synloc:ind3} and \eqref{eq:synloc:ind4}, we verify the desired equivariant isomorphism \eqref{eq:synloc2}, which the lemma was reduced to.
\end{proof}

Recalling that $Q = \Gal(F'/F)$, we now define objects in $D(\ZZ_p[Q])$, respectively in $D(\ZZ_p[\Gal(F'/K)])$ if $F'/K$ is Galois, by setting
\[ I_{F'} := R\Gamma(X, \Ru{F'}D(-E)_{F'} ^{(0)} \otimes^{\mathbb{L}} \Q_p/\Z_p)^*[-2] \] 
and
\[ P_{F'} := R\Gamma(X, \Ru{F'}D(-E)_{F'}^{(\mathcal{L}')}\otimes^\mathbb{L} \Q_p/\Z_p)^*[-2],\]
where $\mathcal{L}'$ is as defined in (\ref{cL' def}). The following result describes the connection between these objects and the constructions made in earlier sections.

\begin{lemma}\label{KT triangles} 
Let $E$, $V_{F'}$ and $M$ be as in Lemma~\ref{lem:KT5.13}.
Then there are canonical exact triangles in $D^-(\ZZ_p[\Gal(F'/M)])$ of the form
\begin{equation} \label{eq:syntriangle}
P_{F'} \xrightarrow{{\bf 1}-\varphi} I_{F'} \xrightarrow{\theta} R\Gamma_{{\rm ar}, V_{F'}
}(U_{F'},\mathcal{A}_{\rm tor})^*_p[-2] \xrightarrow{\theta'} P_{F'}[1]
\end{equation}
and
\begin{equation} \label{eq:lietriangle}
P_{F'} \xrightarrow{{\bf 1}} I_{F'} \to R\Gamma(X, \mathcal{L}')^*[-2] \to P_{F'}[1].
\end{equation}
\end{lemma}

\begin{proof} 
By the result of Lemma~\ref{lem:KT5.13}, the triangle (\ref{eq:syntriangle}) is obtained by applying the exact composite  functor $R\Gamma(X, -\otimes^{\mathbb{L}} \Q_p/\Z_p)$ to the exact triangle (of complexes of sheaves) that results from the definition of $\mathcal{S}_{D_{F'}^{(E,\mathcal{L})}}$ as the mapping fibre of (\ref{compact sheaf}).

The exact triangle (\ref{eq:lietriangle}) results in a similar way by using the canonical exact triangle
\[ 
\xymatrix@1{
\mathcal{L}' \ar[r] & \Ru{F'}D(-E)_{F'}^{(0)}\otimes^{\mathbb{L}} \Q_p/\Z_p \ar[r]^{{\bf 1}}& \Ru{F'}D(-E)_{F'}^{(\mathcal{L}')}\otimes^{\mathbb{L}} \Q_p/\Z_p\ar[r]&}\]
described by Kato and Trihan in \cite[\S6.7]{KT}. 
\end{proof}

The complexes $I_{F'}$ and $P_{F'}$ are not known, in general, to belong to $D^{\rm perf}(\ZZ_p[\Gal(F'/F)])$ and hence, for our purposes, we must adapt the triangles (\ref{eq:syntriangle}) and (\ref{eq:lietriangle}), as per the following result.

\begin{proposition}\label{perfect adapt} Let $M$ be any extension of $K$ over which $F'$ is Galois.
% Let $J$ be a subgroup of $\Gal(F'/M)$ and $\mathfrak{N}$ an order in $\QQ_p[J]$ that contains $\ZZ_p[J]$ 
Set $J:=\Gal(F'/M)$, and let $\mathfrak{N}$ be an order in $\QQ_p[J]$ that contains $\ZZ_p[J]$ and is such that the complex $\tau_{\ge -1}(\mathfrak{N}\otimes_{\ZZ_p[J]}^\mathbb{L}I_{F'})$ can be represented by a bounded complex of projective $\mathfrak{N}$-modules.

Then the triangles (\ref{eq:syntriangle}) and (\ref{eq:lietriangle}) induce exact triangles in $D^{\rm perf}(\mathfrak{N})$ of the form
\[ \tau_{\ge -1}(\mathfrak{N}\otimes_{\ZZ_p[J]}^\mathbb{L}P_{F'}) \xrightarrow{{\bf 1}-\varphi} \tau_{\ge -1}(\mathfrak{N}\otimes_{\ZZ_p[J]}^\mathbb{L}I_{F'}) \to \mathfrak{N}\otimes_{\ZZ_p[J]}^\mathbb{L}R\Gamma_{{\rm ar},
V_{F'}}(U_{F'},\mathcal{A}_{\rm tor})^*_p[-2] \rightarrow \]
and

\[ \tau_{\ge -1}(\mathfrak{N}\otimes_{\ZZ_p[J]}^\mathbb{L}P_{F'}) \xrightarrow{{\bf 1}} \tau_{\ge -1}(\mathfrak{N}\otimes_{\ZZ_p[J]}^\mathbb{L}I_{F'}) \to \mathfrak{N}\otimes_{\ZZ_p[J]}^\mathbb{L}R\Gamma(X, \mathcal{L}')^*[-2]\rightarrow.\]
\end{proposition}

\begin{proof} The results of Proposition \ref{prop:perfect}(i) and Lemma \ref{zar lemma} imply that both of the complexes $C_1 := R\Gamma_{{\rm ar}, V_{F'}
}(U_{F'},\mathcal{A}_{\rm tor})^*_p[-2]$ and $C_2 := R\Gamma(X, \mathcal{L}')^*[-2]$ belong to $D^{\rm perf}(\ZZ_p[J])$ and are acyclic outside degrees $0$ and $1$ and $2$. 

In addition, a finitely generated, torsion-free, $\ZZ_p[J]$-module of finite projective dimension is itself projective (by \cite[Th. 8]{aw}). By a standard resolution argument (as in the proof of Lemma \ref{prop:zzhatcom}(iii)), it therefore follows that the complexes $C_1$ and $C_2$ are both   represented by complexes of finitely generated projective $\ZZ_p[J]$-modules, all terms of which are zero in every degree less than $-1$ and every degree greater than $2$.  

This in turn implies that $\mathfrak{N}\otimes_{\ZZ_p[J]}^\mathbb{L}C_1$ and $\mathfrak{N}\otimes_{\ZZ_p[J]}^\mathbb{L}C_2$ belong to $D^{\rm perf}(\mathfrak{N})$ and are both acyclic in all degrees less than $-1$.  

Given this last fact, one obtains exact triangles in $D^-(\mathfrak{N})$ of the stated form by simply applying the exact functor $\mathfrak{N}\otimes^{\mathbb{L}}_{\ZZ_p[J]}-$ to the triangles (\ref{eq:syntriangle}) and (\ref{eq:lietriangle}).

To prove that these respective triangles belong to $D^{\rm perf}(\mathfrak{N})$ (rather than just $D^-(\mathfrak{N})$) it is then enough, since $\mathfrak{N}\otimes_{\ZZ_p[J]}^\mathbb{L}C_1$ and $\mathfrak{N}\otimes_{\ZZ_p[J]}^\mathbb{L}C_2$ both belong to $D^{\rm perf}(\mathfrak{N})$, to prove that the complex $C := \tau_{\ge -1}(\mathfrak{N}\otimes_{\ZZ_p[J]}^\mathbb{L}I_{F'})$ also belongs to $D^{\rm perf}(\mathfrak{N})$.

To do this we note that, by assumption, $C$ is represented by a bounded complex of projective $\mathfrak{N}$-modules and, by \cite[Prop. 5.15(i)]{KT}, all cohomology groups of $C$ are finitely generated over $\mathfrak{N}$. Taken together, these facts combine with a standard construction of resolutions to imply $C$ belongs to $D^{\rm perf}(\mathfrak{N})$, as required.\end{proof}

Since $\tau_{\ge -1}(\mathfrak{N}\otimes_{\ZZ_p[J]}^\mathbb{L}I_{F'})$ is acyclic outside finitely many degrees  the stated condition in Proposition \ref{perfect adapt} is automatically satisfied if the order $\mathfrak{N}$ is hereditary (and hence, by \cite[Th. (26.12)]{curtisr}, if it is a maximal order).

With Theorem \ref{thm:thm1} in mind, in the next section we will show that, under suitable conditions on $A_{M}$ and $F'/M$  the condition in Proposition \ref{perfect adapt} can also be satisfied by orders that are not maximal.

Then, in \S\ref{sec:lfunctions}, we shall study in greater detail the long exact cohomology sequences of the exact triangles in Proposition \ref{perfect adapt}.

\section{Crystalline cohomology and tame ramification}\label{sect:tame}

We continue to use the general notation of \S\ref{sec:syncom}. We also assume that the extension $F'/F$ is tamely ramified and write $\pi:X_{F'}\to X_F$ for the corresponding cover of smooth projective curves. %For brevity we shall usually abbreviate $X_F$ and $U_F$ to $X$ and $U$ respectively.
We fix a log structure on $X_{F'}$ associated to the divisor $X_{F'}- U_{F'}$, write $X_{F'}^\sharp$ for the associated log scheme and note that the natural map $\pi^\sharp:X_{F'}^\sharp\to X_F^\sharp$ is Kummer-\'etale (in the sense of \cite[Def.~2.13]{Niziol:LogK-1}).

We write $u:(X_F^\sharp/\Zp)_\crys \to X_{F,\et}$ and $u':(X_{F'}^\sharp/\Zp)_\crys \to (X_{F'})_\et$ for the natural morphism of topoi.

%In this section, we fix a divisor $E_F$ of $X_F$ supported exactly at $X_F\setminus U_F$, and write $E:=\pi^*E_F$.
%In this section we shall construct certain complexes of $Q$-equivariant \'etale $\Zp$-modules that represent $Ru_*D(-E_F)^{(0)}$ and $Ru'_* D(-E)^{(0)}_{F'}$ and are useful for the proof of Theorem \ref{thm:thm1}. %(Note that $D(-E)^{(0)}_{F'} = \pi^{\sharp,*}(D(-E_F)^{(0)})$.)

In this section we shall construct certain complexes of $Q$-equivariant \'etale $\Zp$-modules that represent $Ru_*D(-E_F)^{(0)}$ and $Ru'_* D(-E)^{(0)}_{F'}$, where $E$ is the pull back of a suitable divisor $E_F$ of $X_F$ supported exactly at $X_F\setminus U_F$. This construction will play an important role in the proof of Theorem \ref{thm:thm1}.

\subsection{Digression on log de~Rham complexes} The main result of this section is the following general observation concerning crystalline sheaves.

\begin{proposition}\label{prop:perfection2} Let $\cale$ be a locally free crystal of $\mathcal{O}_{\langle F\rangle}$-modules (with $\mathcal{O}_{\langle F\rangle}:=\mathcal{O}_{X_F^\sharp/\Zp}$).
\begin{itemize}
\item[(i)] There exists a bounded below complex $C(\pi^{\sharp,*}\cale)$ of torsion free $\ZZ_p[Q]$-modules that has both of the following properties.
\begin{itemize}
\item[(a)] Each term of $C(\pi^{\sharp,*}\cale)$ is an \emph{induced} $\ZZ_p[Q]$-modules; in other words, in each degree $i$ there is an isomorphism $\ZZ_p[Q]$-modules
\[ C^i(\pi^{\sharp,*}\cale)\cong {\rm Ind}_{\{e\}}^Q(C^i(\pi^{\sharp,*}\cale)^Q),\]
where $e$ denotes the identity element of $Q$.
\item[(b)] For each normal subgroup $J$ of $Q$ there is an isomorphism in $D(\ZZ_p[Q/J])$
\[ \Hom_{\ZZ_p[J]}(\ZZ_p,C(\pi^{\sharp,*}\cale)) \cong R\Gamma_{\crys}(X_{F'^J}^\sharp/\Zp,\pi_J^{\sharp,*}\cale),\]
where $\pi^\sharp_J:X_{F'^J}^\sharp\to X^\sharp_F$ is the natural projection.
\end{itemize}
\item[(ii)] If there is a short exact sequence of sheaves 
\[ 0 \to \cale^0 \to \cale \to i_{X_F^\sharp/\Zp,*}\calf \to 0\]
for a vector bundle $\calf$ on $X_F$, then claim (i) is also true with $\cale$ replaced by $\cale^0$.
\end{itemize}
\end{proposition}

\subsubsection{}As preparation for the proof of this result we start with the following technical result.

\begin{lemma}\label{lem:lift} There exists a formal scheme $\gotx_F^\sharp$ over $\Z_p$ that is a smooth lift of $X_F^\sharp$. Furthermore, for any finite Kummer-\'etale covering $X_{F'}^\sharp \to X_F^\sharp$, there exists a finite Kummer-\'etale  covering $\tilde\pi^\sharp:\gotx_{F'}^\sharp\to\gotx_F^\sharp$ that lifts $\pi^\sharp:X_{F'}^\sharp\to X_F^\sharp$.
\end{lemma}

\begin{proof} This lemma is obtained from the infinitesimal deformation theory for smooth log schemes (cf. \cite[Prop.~3.14]{Kato:Fontaine-Illusie}). More precisely, if $\gotx_{F,n}^\sharp$ is a (flat) lift of $X_F$ over $\Z_p/p^n$, then it is easy to see that $\gotx_{F,n}^\sharp$ is log smooth over  $\Z_p/p^n$ (where $\Z_p/p^n$ is given the trivial log structure). To see this, one applies Kato's criterion \cite[Th.~3.5]{Kato:Fontaine-Illusie}. By \cite[Prop. ~3.14(4)]{Kato:Fontaine-Illusie}, the obstruction class for lifting $\gotx_{F,n}^\sharp$ over $\Z_p/p^{n+1}$ lies in $H^2(X_F,\omega_{X_F^\sharp}^\vee) = 0$, where $\omega_{X_F^\sharp}$ is the sheaf of  differentials with log poles at $X_F-U_F$. We write  $\gotx_F^\sharp$ for the natural inverse limit $\varprojlim_n\gotx_{F,n}^\sharp$.

Since the sheaf of relative log differentials $\omega_{X_{F'}^\sharp/X^\sharp}$ is trivial, it follows that the finite Kummer-\'etale  covering $\pi^\sharp:X_{F'}^\sharp\to X_F^\sharp$ canonically lifts to $\tilde\pi^\sharp:\gotx_{F',n}^\sharp\to\gotx_{F,n}^\sharp$ (cf. \cite[Prop.~3.14]{Kato:Fontaine-Illusie}). This produces the desired finite Kummer-\'etale  covering $\tilde\pi^\sharp:\gotx_{F'}^\sharp\to\gotx_F^\sharp$.
\end{proof}

We use this lemma to obtain some complexes representing $Ru_*\cale$ and $Ru'_*(\pi^{\sharp,*}\cale)$ for a  locally free crystal $\cale$  of $\mathcal{O}_{\langle F\rangle}$-modules. Given such $\cale$, we obtain a vector bundle $\cale_{\gotx_{F}^\sharp}$ that is equipped with an integrable connection with log poles $\nabla:\cale_{\gotx_{F}^\sharp}\to\cale_{\gotx_{F}^\sharp}\widehat\otimes_{\mathcal{O}_{{\gotx}_F}}\widehat\omega_{\gotx_F^\sharp}$.

Furthermore, since $X_F^\sharp\hookrightarrow \gotx_F^\sharp$ is a good embedding in the sense of \cite[\S~5.6]{KT}, it follows that $\cale$ is functorially determined by $(\cale_{\gotx_F^\sharp},\nabla)$ by \cite[Th.~6.2]{Kato:Fontaine-Illusie}. The same holds for any locally free crystal $\cale'$  of $\mathcal{O}_{\langle F'\rangle}$-modules, and the associated vector bundle with integrable connection with log poles  $(\cale'_{\gotx^\sharp_{F'}},\nabla)$.

Recall that the map $\tilde \pi:\gotx_{F'}\to\gotx_F$ is flat\footnote{It suffices to verify the flatness at the formal neighbourhood of any closed point. And by Abhyankar's lemma (cf. \cite[A.11]{FreitagKiehl}), the map of completed local rings induced by $\tilde\pi$ is of the form $W(\FF_q)[[t]]\to W(\FF_{q'}[[t^{1/e}]])$ for some $e$ not divisible by $p$.} and we have $\tilde\pi^*\omega_{\gotx^\sharp_F}\riso\omega_{\gotx_{F'}^\sharp}$ by \cite[Prop. ~3.12]{Kato:Fontaine-Illusie}, so we can define pull back and push forward by $\tilde\pi$ for vector bundles with connection with log poles (just as the unramified case).

Furthermore, by unwinding the proof of \cite[Th.~6.2]{Kato:Fontaine-Illusie}, one can see that the construction $\cale\rightsquigarrow (\cale_{\gotx_F^\sharp},\nabla)$ (and the same construction for $\cale'$) respects the pull back and push forward by $\pi^\sharp$ so that one has both $((\pi^{\sharp,*}\cale)_{\gotx^\sharp_{F'}},\nabla) = \tilde\pi^*(\cale_{\gotx_F^\sharp},\nabla)$ and $((\pi^{\sharp}_*\cale')_{\gotx_F^\sharp},\nabla) = \tilde\pi_*(\cale'_{\gotx^\sharp_{F'}},\nabla)$.

In particular, both $(\pi^{\sharp,*}\cale)_{\gotx^\sharp_{F'}}$ and $(\pi^{\sharp}_*\cale')_{\gotx_F^\sharp}$ have natural horizontal actions of $Q$.

Let $\gotx_{F,n}$ denote the closed subscheme of $\gotx_F$ cut out by the ideal generated by $p^n$. Then a coherent  $\mathcal{O}_{\gotx_{F,n}}$-modules $\calf_n$ can be seen as a torsion \'etale sheaf on $X_F$, where for any \'etale morphism $f:Y\to\gotx_{F,n}$ we have $\calf_n(Y) := \Gamma(Y, f^*\calf_n)$. Similarly, any coherent $\mathcal{O}_{\gotx_F}$-module $\calf$ can be viewed as a $\Zp$-\'etale sheaf on $X_F$; namely, the inverse system of torsion \'etale sheaves $\{\calf|_{\gotx_{n}}\}$.

Now, for any  locally free crystal $\cale$  of $\mathcal{O}_{\langle F\rangle}$-module, the complex $Ru_{*}\cale$ can be computed via the complex of $\Zp$-\'etale sheaves on $X_F$ given by $\cale_{\gotx_F^\sharp}\xrightarrow\nabla\cale_{\gotx_F^\sharp}\widehat\otimes_{\mathcal{O}_{{\gotx}_F}}\widehat\omega_{\gotx_F^\sharp},$ where the first term is placed in degree zero (cf. \cite[\S~5.6]{KT}). One also obtains a similar expression for $Ru'_*(\pi^{\sharp,*}\cale)$ as a complex of `$\Zp[Q]$-\'etale sheaves' on $X_{F'}$.

Given a short exact sequence
\[
0 \to \cale^0 \to \cale \to \calf \to 0,
\]
where $\calf$ is a vector bundle on $X_F$ viewed as a log crystalline sheaf, we have a short exact sequence $0\to Ru_{*}\cale^0\to Ru_{*}\cale \to \calf\to 0$, where $\calf$ is viewed as a torsion \'etale sheaf on $X_F$. Therefore, we may express
\begin{equation}
Ru_{*}\cale^0 =  [\cale^0_{\gotx_F^\sharp}\xrightarrow\nabla\cale_{\gotx_F^\sharp}\widehat\otimes_{\mathcal{O}_{{\gotx}_F}}\widehat\omega_{\gotx_F^\sharp}],
\end{equation}
where $\cale^0_{\gotx_F^\sharp}$ denotes the kernel of $\cale_{\gotx_F^\sharp} \twoheadrightarrow i_{X^\sharp_F/\Zp,*}\calf$.

\begin{remark} {\em Note that  $\gotx_F^\sharp$ can be obtained as a $p$-adic completion of a proper smooth log scheme $\widetilde X_F^\sharp$ over $\Z_p$, where the underlying scheme $\widetilde X_F$ is a smooth lift of $X_F$ and the log structure is given by relative divisor $\widetilde Z\subset \widetilde X_F$ smoothly lifting $Z:=|X_F-U_F|$.

Let us now give some examples of $Ru_{*}\cale$ for some $\cale$. When $\cale = \mathcal{O}_{\langle F\rangle}$ then $Ru_{*}\mathcal{O}_{\langle F\rangle}$ is the log de~Rham complex of $\gotx_F^\sharp$; that is, the $p$-adic completion of the de~Rham complex of $\widetilde X_F$ with log poles along $\widetilde Z$.

Given any divisor $E_F$ of $X_F$ supported in $Z$, one obtains a rank one locally free crystal of  $\mathcal{O}_{\langle F\rangle}$-modules $\cale:=\mathcal{O}_{\langle F\rangle}(E_F)$.

Let us now describe $Ru_{*}\mathcal{O}_{\langle F\rangle}(E_F)$. Viewing $\gotx_F^\sharp$ as the $p$-adic completion of the log scheme $\widetilde X_F^\sharp$ with divisorial log structure associated to $\widetilde Z$, we can find a relative divisor $\widetilde E_F$ of $\widetilde X_F$ that lifts $E_F$ and is supported in $\widetilde Z$.
 Then  from the definition of $\mathcal{O}_{\langle F\rangle}(E_F)$ (cf. \cite[\S~5.12]{KT}), one can check that 
\[
Ru_{*}\mathcal{O}_{\langle F\rangle}(E_F)  = [\mathcal{O}_{\widetilde X_F}(\widetilde E_F) \xrightarrow\nabla \mathcal{O}_{\widetilde X_F}(\widetilde E_F)\otimes \omega_{\widetilde X_F^\sharp} ] \otimes_{\mathcal{O}_{\widetilde X_F}} O_{\gotx_F},
\]
where $\nabla$ is induced by the universal derivation $d:\mathcal{O}_{\widetilde X_F}\to\omega_{\widetilde X_F^\sharp}=\Omega_{\widetilde X_F}(\log \widetilde Z)$. (Here, $\nabla$ is well defined since $\widetilde E_F$ is supported in $\widetilde Z$, where $\omega_{\widetilde X_F^\sharp}$ is allowed to have log poles.)}
\end{remark}

\subsubsection{}We are now ready to prove Proposition \ref{prop:perfection2}.

We shall, for brevity, only prove claim (ii) since this is directly relevant to the proof of Theorem \ref{thm:thm1} and claim (i) can be proved by exactly the same argument.

Our strategy is to use Proposition~\ref{prop:Cech} to construct a complex $C(\pi^{\sharp,*}\cale^0)$ of \emph{induced} $\ZZ_p[Q]$-modules that represents $R\Gamma_{\crys}(X_{F'}^\sharp/\Zp,\pi^{\sharp,*}\cale^0)$ in such a way that $C(\cale^0):=C(\pi^{\sharp,*}\cale^0)^Q$ is naturally isomorphic in $D(\ZZ_p)$ to $R\Gamma_{\crys}(X_{F}^\sharp/\Zp,\cale^0)$. (Since each term of $C(\pi^{\sharp,*}\cale^0)$ is  an induced $\ZZ_p[Q]$-module, the complex $C(\cale^0)$ of term-wise $Q$-invariants of $C(\pi^{\sharp,*}\cale^0)$ represents $R\Hom_{\ZZ_p[Q]}(\ZZ_p,C(\pi^{\sharp,*}\cale^0))$.)

We recall that $R\Gamma_{\crys}(X_{F}^\sharp/\Zp,\cale^0)$ identifies with $R\Gamma_{\et}(X_{F},Ru_*\cale^0)$ and that $Ru_*\cale^0$ is equal to the complex $\cale^0_{\gotx_F^\sharp}\xrightarrow\nabla\cale_{\gotx_F^\sharp}\otimes_{\mathcal{O}_{\gotx_{F}}}\widehat\omega_{\gotx_{F}^\sharp}$.
In particular, since all the terms of $Ru_*\cale^0$ are `coherent $\mathcal{O}_{\gotx_{F}}$-modules', we can compute $R\Gamma_{\et}(X_{F},Ru_*\cale^0)$ via Zariski topology on $\gotx_{F}$ (viewing  $Ru_*\cale^0$ as a complex of coherent $\mathcal{O}_{\gotx_{F}}$-modules with additive differential). Note that the same properties hold for  $Ru'_*(\pi^{\sharp,*}\cale^0)$ as well.

We now choose the disjoint union of some $Q$-stable finite affine open covering $\gotu_F^\sharp$ of $\gotx_{F'}^\sharp$ and regard it as a Kummer-\'etale covering of $\gotx_F^\sharp$. We then let $C(\cale^0)$ denote the total complex associated to the {\Cech} resolution of $Ru_*\cale^0$ with respect to $\gotu_F^\sharp$. Similarly, we let $C(\pi^{\sharp,*}\cale^0)$ denote the total complex associated to the {\Cech} resolution of $Ru'_*(\pi^{\sharp,*}\cale^0)$ with respect to the Kummer-\'etale covering $\gotu_F^\sharp\times_{\gotx_F^\sharp}\gotx_{F'}^\sharp$ of $\gotx_{F'}^\sharp$, which is a complex of $\ZZ_p[Q]$-modules where the $Q$-action is induced from the $Q$-action on $\gotx_{F'}$. Then, by Proposition~\ref{prop:Cech}, we know that $C(\cale^0)$ is isomorphic in $D(\ZZ_p)$ to $R\Gamma_{\crys}(X_F^\sharp/\Zp,\cale^0)$ and that $C(\pi^{\sharp,*}\cale^0)$ is isomorphic in $D(\ZZ_p[Q])$ to $R\Gamma_{\crys}(X_{F'}^\sharp/\Zp,\pi^{\sharp,*}\cale^0)$.

In addition, one has $\gotu_F^\sharp\times_{\gotx_F^\sharp}\gotx_{F'}^\sharp \cong \gotu_F^\sharp\times Q$ and so in each degree $i$ there is an isomorphism of $\ZZ_p[Q]$-modules
%$
\[ C^i(\pi^{\sharp,*}\cale^0)\cong \Hom_{\ZZ_p}(\ZZ_p[Q],C^i(\cale^0)) = {\rm Ind}_{\{e\}}^Q C^i(\cale^0),\]
where $C^i(\pi^{\sharp,*}\cale^0)$ and $C^i(\cale^0)$ denote the $i$-th term of $C(\pi^{\sharp,*}\cale^0)$ and $C(\cale^0)$, respectively. (Indeed, we have that  each term of $Ru'_{*}(\pi^{\sharp,*}\cale^0)$ is obtained by the pull back of the terms of $Ru_*\cale^0$ as coherent sheaves, using the isomorphism $\tilde\pi^*\widehat\omega_{\gotx_F^\sharp}\riso \widehat\omega^1_{\gotx_{F'}^\sharp}$  obtained in  \cite[Prop.~3.12]{Kato:Fontaine-Illusie}.) Therefore, we have $C(\cale^0) = C(\pi^{\sharp,*}\cale^0)^Q.$
(To see that the {\Cech} differentials on both sides match, we note that the {\Cech} resolution $C(\pi^{\sharp,*}\cale^0)$ is constructed with respect to the pull back $\gotu^\sharp_F\times_{\gotx_F}\gotx_{F'}$ of the Kummer-\'etale covering $\gotu^\sharp_F$ of $\gotx_F$, which was used for constructing the {\Cech} resolution $C(\cale^0)$.)

It remains to show that for any subgroup $J$ of $Q$ the complex $C(\pi^{\sharp,*}\cale^0)^J$ represents $R\Gamma_{\crys}(X^\sharp_{F'^J}/\Zp,\pi_J^{\sharp,*}\cale^0)$. Note that we have
\[\gotu^\sharp_{F}\times_{\gotx^\sharp_F} \gotx^\sharp_{F'^J} = \gotu^\sharp_{F}\times_{\gotx^\sharp_F} (\gotx^\sharp_{F'}/J) \cong \gotu^\sharp_{F} \times (Q/J),
\]
So it follows that $C(\pi^{\sharp,*}\cale^0)^J$ is the total complex of the {\Cech} resolution of $Ru_{F'^J,*}(\pi_J^{\sharp,*}\cale^0)$ with respect to the Kummer-\'etale covering $\gotu^\sharp_{F}\times_{\gotx^\sharp_F} \gotx^\sharp_{F'^J}$ of $\gotx^\sharp_{F'^J}$, and so $C(\pi^{\sharp,*}\cale^0)^J$ represents $R\Gamma_{\crys}(X^\sharp_{F'^J}/\Zp,\pi_J^{\sharp,*}\cale^0)$ by Proposition~\ref{prop:Cech}.

This completes the proof of Proposition \ref{prop:perfection2}.

\subsection{The complex $I_{F'}$}\label{subsec:pfperf2} The following consequence of Proposition \ref{prop:perfection2} regarding the complex $I_{F'}$ constructed in \S\ref{sec:syncom} will play an important role in the proof of Theorem \ref{thm:thm1}.

\begin{proposition}\label{I tame} If the extension $F'/F$ is tamely ramified, then $I_{F'}$ lies in $D^{\rm perf}(\ZZ_p[Q])$ and is acyclic in all degrees outside $0, 1$ and $2.$ \end{proposition}

\begin{proof}
Throughout this proof we use the notation introduced at the beginning of \S\ref{sec:syncom} with $K'=F$.  By applying  Proposition~\ref{prop:perfection2} to  $\cale^0 = D(-E_F)^{(0)}$ (so we have $\pi^{\sharp,*}\cale^0 = D(-E)_F^{(0)}$ as $E=\pi^*E_F$), we obtain a complex of torsion-free induced $\ZZ_p[Q]$-modules $C_{F'}$ representing
\[R\Gamma_{\crys}(X_{F'}^\sharp/\Zp,D(-E)_{F'}^{(0)})\cong R\Gamma_{\et}(X_{F'},Ru'_*(D(-E)_{F'}^{(0)}))\]
such that for any subgroup $J$ of $Q$ the complex $C_{F'}^J$ represents $R\Gamma_{\crys}(X_{F}^\sharp/\Zp,D(-E_{F'^J})_{F'^J}^{(0)})$ where $E_{F'^J}$ is the pull back of $E_F$ to $X_{F'^J}$. In particular,  in each degree $i$ there is an isomorphism of $\ZZ_p[Q]$-modules $C^i_{F'} \cong \Hom_{\ZZ_p}(\ZZ_p[Q],(C^i_{F'})^Q)$.

Since  $(C^i_{F'})^J$ is $\ZZ_p$-flat in all degrees $i$, for any normal subgroup $J$ of $Q$ there is an isomorphism in $D(\ZZ_p[Q/J])$
\[ I_{F'^J} \cong (C_{F'}^J\otimes_{\ZZ_p} \QQ_p/\ZZ_p)^*\]
where the complexes on the right hand side are defined by the term-wise operations.

If we set $I^i_{F'}:=(C^i_{F'}\otimes_{\ZZ_p}\QQ_p/\ZZ_p)^*$ and  $I^i_F:=((C^i_{F'})^Q\otimes_{\ZZ_p}\QQ_p/\ZZ_p)^*$ for any $i$, then we have
\[
I^i_{F'} \cong \ZZ_p[Q]\otimes_{\ZZ_p}I^i_F,
\]
which is a flat $\ZZ_p[Q]$-module. Therefore for any subgroup $J$ of $Q$ the derived coinvariants $\ZZ_p\otimes^{\mathbb{L}}_{\ZZ_p[J]} I_{F'}$ can be represented by the following complex defined by term-wise operations:
\[
\ZZ_p\otimes_{\ZZ_p[J]}(C_{F'}\otimes_{\ZZ_p} \QQ_p/\ZZ_p)^* \cong ((C_{F'})^J\otimes_{\ZZ_p}\QQ_p/\ZZ_p)^*.
\]
This implies, in particular, that $\ZZ_p\otimes^{\mathbb{L}}_{\ZZ_p[J]} I_{F'}$ is isomorphic in $D(\ZZ_p[Q/J])$ to $I_{F'^J}.$

Thus, since each complex $I_{F'^J}$ is acyclic outside degrees $0,1$ and $2$ and each cohomology group of $I_{F'}$ is finitely generated over $\ZZ_p$, a standard argument (as already used at the beginning of the proof of Proposition \ref{prop:perfect}) implies that $I_{F'}$ belongs to $D^{\rm perf}(\ZZ_p[Q])$, as claimed.
\end{proof}

\section{Crystalline cohomology, semisimplicity and vanishing orders} \label{sec:lfunctions}

As further preparation for the proof of Theorem \ref{thm:thm1}, in this section we establish a link between the long exact cohomology sequences of the exact triangles constructed in Lemma \ref{KT triangles} and the rational height pairing of Schneider and then use it to study the orders of vanishing of Hasse-Weil-Artin $L$-series.

Throughout we use the notation of Lemma \ref{KT triangles}. For convenienc, we also set 
\[ Q_M := \Gal(F'/M)\]
and $Y_{\QQ_p} := \QQ_p\otimes_{\ZZ_p}Y$ for each $\ZZ_p$-module $Y$.

\subsection{Height pairings and semisimplicity}\label{ssec:conslfun} At the outset we recall that, by the general discussion given at the beginning of \cite[\S4.3]{KT}, for each intermediate field $M$ of $L'/K$ the Dieudonn\'e isocrystal $D(\mathcal{A}_M|_{U_M})\otimes\mathbb{Q}_p$ on $(U_{M/\ZZ_p})_{\rm crys}$ comes from an overconvergent $F$-isocrystal on $U_M$ that we shall denote by 
\[ D^{\dagger}_M = D^{\dagger}(A_M).\]

%
%At the outset we recall that, by the general discussion given at the beginning of \cite[\S4.3]{KT}, for each intermediate field $M$ of $L'/K$ the Dieudonn\'e crystal {\color{yellow}to help the reader, is this the $D_M$ from (\ref{crystal notation})??}  on \wk{$(U_{M/\ZZ_p})_{\rm crys}$} that is associated to the $p$-divisible group over $U$ arising from $\mathcal{A}$ comes from an overconvergent $F$-isocrystal on $U_M$ that we shall denote by 
%%
%\[ D^{\dagger}_M = D^{\dagger}(A_M).\]  

We further recall that, by \cite[\S4.9 and Prop.~5.15]{KT}, there are natural identifications
\begin{equation}\label{missing complexes}
\Q_p \otimes_{\Z_p} I_{F'} = \Q_p \otimes_{\Z_p} P_{F'} = R\Hom_{\Q_p} \big(R\Gamma_{{\rm rig}, c}(U_{F'}, D^{\dagger}_{F'}),\Q_p\big)[-2] \end{equation}
with respect to which the morphism ${\bf 1}$ in the exact triangle (\ref{eq:lietriangle}) corresponds to the identity endomorphism on $R\Gamma_{{\rm rig}, c}(U_{F'}, D^{\dagger}_{F'})$. %Indeed, the second and the third terms are naturally identified with $\Q_p \otimes_{\Z_p}\wk{R\Gamma_{\crys,c}(X^{\sharp}_{F'},D_{F'})}$ by loc.~cit. {\color{yellow}what is $D$ in this last expression??}

Upon combining these identifications with the long exact cohomology sequence of the exact triangle (\ref{eq:syntriangle}) we obtain a composite homomorphism
\begin{equation}\label{semi second step}
 \beta_{A,F',p}: \QQ_p\otimes_\ZZ A^t(F') \xrightarrow{H^0(\theta')} H^1(P_{F'})_{\QQ_p} = H^1(I_{F'})_{\QQ_p}  \xrightarrow{H^1(\theta)}\QQ_p\otimes_\ZZ \Hom_\ZZ(A(F'),\ZZ).\end{equation}

We also write
\[ h_{A,F',p,*}: \QQ_p\otimes_\ZZ A^t(F') \to \QQ_p\otimes_\ZZ \Hom_\ZZ(A(F'),\ZZ)\]
for the isomorphism of $\QQ_p[Q_M]$-modules that is induced by the algebraic height pairing $h_{A,F'}$ that occurs in \S\ref{sect:reform}.

\begin{proposition}\label{prop:htfrobcomp} If $\sha(A/F')$ is finite, then the following claims are valid.
\begin{itemize}
\item[(i)] One has $\beta_{A,F',p} = (\pm 1)^{a_{A,F',p}}\times h_{A,F',p,*}$ for a computable integer $a_{A,F',p}$ in $\{0,1\}$.
\item[(ii)] The homomorphisms $H^i(\hat\varphi)_{\QQ_p}$ are bijective for all $i \not= 1$, where $\hat\varphi:= 1-\varphi$.
\item[(iii)] The $\QQ_p[Q_M]$-module $\ker(H^1(\hat\varphi))_{\QQ_p}$ is naturally isomorphic to $\QQ_p\otimes A^t(F')$.
\item[(iv)] The composite map $\ker(H^1(\hat\varphi))_{\QQ_p} \subseteq H^1(P_{F'})_{\QQ_p} = H^1(I_{F'})_{\QQ_p} \to {\rm cok}(H^1(\hat\varphi))_{\QQ_p}$ is bijective.
\end{itemize}
\end{proposition}

\begin{proof} Write $\mathfrak{C}$ for the quotient of the category of $\ZZ_p[Q_M]$-modules by the category of finite $\ZZ_p[Q_M]$-modules.

Then, since $\sha(A/F')$ is assumed to be finite, the (non-degenerate) pairing $h_{A,F'}$ induces an isomorphism in $\mathfrak{C}$ of the form
\begin{equation}\label{first map}
A(F') \otimes_\ZZ \QQ_p/\ZZ_p \rightarrow \Hom_\ZZ(A^t(F'), \QQ_p/\ZZ_p).\end{equation}

Next we set $C' := R\Gamma_{{\rm ar},V_{F'}}(U_{F'},\mathcal{A}_{\rm tor})$. Then, since the kernel of the homomorphism $
H^1(C') \to {\rm Sel}_{\QQ/\ZZ}(A_{F'})$ in Proposition \ref{prop:arithcoho} is finite the natural map $A(F') \otimes_\ZZ \QQ/\ZZ \to {\rm Sel}_{\QQ/\ZZ}(A_{F'})$ factors through a map  $A(F') \otimes_\ZZ \QQ/\ZZ \rightarrow H^1(C')$ in $\mathfrak{C}$. This homomorphism then gives rise to a composite homomorphism in $\mathfrak{C}$ of the form
\begin{multline}\label{second map} A(F') \otimes_\ZZ \QQ/\ZZ \rightarrow H^1(C')_p \rightarrow H^1(I^*_{F'}) \xrightarrow{\bf 1} H^1(P^*_{F'}) \\
 \rightarrow H^2(C')_p \rightarrow \Hom_\ZZ({\rm Sel}_{\hat{\ZZ}}(A^t), \QQ_p/\ZZ_p) \rightarrow \Hom_\ZZ(A^t(F'), \QQ_p/\ZZ_p),
\end{multline}
where the second and fourth maps are induced by the exact triangle (\ref{eq:syntriangle}) and the fifth by Proposition \ref{prop:arithcoho}.

To prove claim (i) it is sufficient, after taking Pontryagin duals, to show that the morphisms (\ref{first map}) and (\ref{second map}) in $\mathfrak{C}$ coincide up to a computable sign and this is precisely what is established by the argument of Kato and Trihan in \cite[3.3.6.2]{KT}.

To prove the other claims we note that the long exact cohomology sequence of the exact triangle (\ref{eq:syntriangle}) combines with the descriptions in Proposition \ref{prop:perfect}(ii) to imply that $H^i(\hat\varphi)_{\QQ_p}$ is bijective for all $i \notin \{0,1\}$, that ${\rm ker}(H^0(\hat\varphi))_{\QQ_p}$ and ${\rm cok}(H^2(\hat\varphi))_{\QQ_p}$ vanish and that there are exact sequences of $\QQ_p[Q_M]$-modules
\begin{equation}\label{semi first step} \begin{cases} &0 \to {\rm cok}(H^0(\hat\varphi))_{\QQ_p} \xrightarrow{H^0(\theta)} \QQ_p\otimes A^t(F') \xrightarrow{H^0(\theta')} {\rm ker}(H^1(\hat\varphi))_{\QQ_p} \to 0,\\
                 &0 \to {\rm cok}(H^1(\hat\varphi))_{\QQ_p} \xrightarrow{H^1(\theta)}\QQ_p\otimes_\ZZ\Hom_\ZZ(A^t(F'),\ZZ) \xrightarrow{H^1(\theta')}
                 {\rm ker}(H^2(\hat\varphi))_{\QQ_p} \to 0.\end{cases}\end{equation}

Now, since $h^*_{A,F'}$ is bijective, claim (i) implies the same is true of the map $\beta_{A,F}$ and this fact combines with the above exact sequences to imply that the spaces ${\rm cok}(H^0(\hat\varphi))_{\QQ_p}$ and ${\rm ker}(H^2(\hat\varphi))_{\QQ_p}$ vanish, as required to complete the proof of claim (ii), and hence that the upper sequence in (\ref{semi first step}) gives an isomorphism of the sort required by claim (iii).

Finally, claim (iv) is true because the bijectivity of $\beta_{A,F'}$ combines with the upper sequence in (\ref{semi first step}) to imply $\ker(H^1(\hat\varphi))_{\QQ_p}$ is disjoint from $\ker(H^1(\theta))_{\QQ_p}$ whilst the lower sequence in (\ref{semi first step}) implies that $\ker(H^1(\theta))_{\QQ_p}$ is equal to ${\rm im}(H^1(\hat\varphi))_{\QQ_p}$.
\end{proof}

\subsection{Orders of vanishing and leading terms} \label{ssec:interpolation} We now derive from Proposition \ref{prop:htfrobcomp} the following result about the order of vanishing $r_{A,M}(\chi)$ at $t=p^{-1}$ of the functions $Z_{U_M}(A_M,\chi,t)$ that are defined in \S\ref{sect:reform} for each character $\chi$ in ${\rm Ir}(Q_M)$.

We fix (and do not in the sequel explicitly indicate) an isomorphism of fields $\CC\cong \CC_p$ and hence do not distinguish between ${\rm Ir}(Q_M)$ and the set of irreducible $\CC_p$-valued characters of $Q_M$.

In particular, for $\chi$ in ${\rm Ir}(Q_M)$ we may then fix a representation
 $Q_M \rightarrow {\rm Aut}_{\CC_p}(V_{\chi})$ (that we also denote by $\chi$) of character $\chi$, where $V_\chi$ is a finite dimensional vector space over $\CC_p$.

If $R$ denotes either $\ZZ_p[Q_M]$ or $\QQ_p[Q_M]$, then for each finitely generated $R$-module $W$ and each $\chi$ in ${\rm Ir}(Q_M)$ we define a $\CC_p$-vector space by setting
\[ W^\chi:= \Hom_{\CC_p[Q_M]}(V_{\chi},\CC_p[Q_M]\otimes_{R}W).\]

\begin{theorem}\label{order-leading} For each $\chi$ in ${\rm Ir}(Q_M)$ the following claims are valid.
\begin{itemize}
\item[(i)] $r_{A,M}(\chi) = {\rm dim}_{\CC}((\CC\otimes_\ZZ A^t(F'))^{\chi}) = \chi(1)^{-1}\cdot{\rm dim}_{\CC}(e_{\chi}(\CC\otimes_\ZZ A^t(F')))$.
\item[(ii)] In each degree $i$ the homomorphism $H^i({\bf 1}-\varphi)$ induces an automorphism $H^i({\bf 1}-\varphi)_\chi^{\diamond}$ of any fixed complement to $\ker(H^i({\bf 1}-\varphi))^\chi$ in $H^1(P_{F'})^{\chi}$.
%$\wk{[H^i_{{\rm rig}, c}(U_{F'}, D^{\dagger}_{F'})^*]^\chi}$.
\item[(iii)] $Z^*_{U_M}(A_M,\chi,p^{-1}) = {\prod}_{i=0}^{i=2}\, {\rm det}(H^i({\bf 1}-\varphi)_{\chi}^{\diamond})^{(-1)^{i+1}}$,   where the leading term is normalised as in \eqref{eq:leading-term-Z}.
\end{itemize}
\end{theorem}

\begin{proof} 
We fix a finite Galois extension $\Lambda$ of $\QQ_p$ such that for any $\chi$ in ${\rm Ir}(Q_M)$ the $\CC_p[Q_M]$-module $V_\chi$ descends to a $\Lambda[Q_M]$-module $V_{\chi,\Lambda} $. We write $k_\Lambda$ for the residue field of $\Lambda$ and set $q :=\#(k_\Lambda)$. %$r:=[k_\Lambda:\FF_p]$. 
Then, for each $\chi$ in ${\rm Ir}(Q_M)$, we fix a $\Lambda[Q_M]$-module $V_{\chi,\Lambda}$ such that $\CC_p\otimes_\Lambda V_{\chi,\Lambda} \cong V_\chi$ and, for any $\QQ_p[Q_M]$-module $W$, we set
\[
W^\chi_\Lambda := \Hom_{\Lambda[Q_M]}(V_{\chi,\Lambda},\Lambda\otimes_{\QQ_p}W). 
\]

We now give an alternative description of $H^i_{{\rm rig}, c}(U_{F'}, D^{\dagger}_{F'})_\Lambda^\chi$ and $H^i({\bf 1}-\varphi)_{\chi,\Lambda}^{\diamond}$ in terms of the rigid cohomology of \emph{overconvergent $\Lambda$-$F$-isocrystal}; \emph{cf.} \cite[(7.1)]{tsuzuki98}. (In fact, we will work with overconvergent $\Lambda_0$-$F$-isocrystal for some suitable subfield $\Lambda_0$ of $\Lambda$.) 
We recall that an  $\Lambda$-$F$-isocrystal is, roughly speaking, an isocrystal with scalars in $\Lambda$ (instead of $\QQ_p$) equipped with $\Lambda$-linear $q$-Frobenius operator (denoted by $\varphi^{(\Lambda)}$). In particular, given an overconvergent $F$-isocrystal $(\mathcal{E},\varphi)$ over $U_{F'}$, one can `extend scalars' to obtain an overconvergent $\Lambda$-$F$-isocrystal $\mathcal{E}_\Lambda$ in the following way: if $F'$ contains $k_\Lambda$ and we set $r:=[k_\Lambda:\FF_p]$, then $(\mathcal{E},\varphi^r)$ is an overconvergent $\QQ_{q}$-$F$-isocrystal and so one can set $(\mathcal{E}_\Lambda, \varphi^{(\Lambda)}):=(\mathcal{E}\otimes_{\QQ_q}\Lambda, \varphi^r\otimes\Lambda)$. In addition, there is the following base change result (\emph{cf.} \cite[Th.~11.8.1]{CT:Descent}): for any overconvergent isocrystal $\mathcal{E}$ over $U_{F'}$, there is in each degree $i$ a natural isomorphism 
\[H^i_{{\rm rig},c}(U_{F'},\mathcal{E})\otimes_{\QQ_q}\Lambda \cong H^i_{{\rm rig},c}(U_{F'},\mathcal{E}_\Lambda ) \] 
(and similarly for the rigid cohomology without support condition).

Now, following the above discussion, if we are to construct overconvergent $\Lambda$-$F$-isocrystals, we should assume that the base field contains $k_\Lambda$. %To remove this restriction on $M$, we may need to work with $\Lambda'$-$F^{r'}$-isocrystals over $U_M$ for some suitable $\Lambda'$ with smaller residue field. To explain, let $k'$ be the intersection of $k_\Lambda$ and $M$, and set $r':=[k':\FF_p]$. We define a subfield $\Lambda'$ of $\Lambda$ as follows. If $k'=k_\Lambda$ (i.e., $M$ contains $k_\Lambda$) then we set $\Lambda':=\Lambda$. Otherwise we set $\Lambda':=W(k')_\QQ = \QQ_{p^{r'}}$. Then overconvergent $\Lambda'$-$F^{r'}$-isocrystals can be defined over $U_M$.
 To do this, we shall, if necessary, replace $F'$ by $F'':=F'\cdot \FF_{p^r}$. Then $F''/M$ is a Galois extension and, setting $Q_{M}'':=\Gal(F''/M)$, we regard ${\rm Ir}(Q_M)$ as a subset of ${\rm Ir}(Q_{M}'')$ in the natural way. Now, if $W''$ is a finitely generated module over either $\ZZ_p[Q_{M}'']$ or $\QQ_p[Q_{M}'']$  then $(W'')^\chi = W^\chi$ with $W:=(W'')^{\Gal(F''/F')}$. Hence, to prove the claimed result, we can  assume without loss of generality that $k_\Lambda\subset F'$.

%Applying \cite[7.5]{ELS2} to the scalar extension from $\QQ_p$ to  $\Lambda_0$, we obtain the following $Q_M$-equivariant isomorphisms between generalised $\lambda$-eigenspaces for $\varphi$: 
%\begin{align*}
%	H^i_{{\rm rig},c}(U_{F'},  D^{\dagger}_{F'})^{(\lambda)}\otimes_{\QQ_p}\Lambda_0
%&\cong ( H^i_{{\rm rig},c}(U_{F'},  D^{\dagger}_{F'})\otimes_{\QQ_p}\Lambda_0)^{(\lambda)}\\
%& \cong H^i_{{\rm rig},c}(U_{F'}\times_k k_\Lambda,  D^{\dagger}_{F'\otimes_k k_\Lambda})^{(\lambda)}.
%\end{align*}
%%&\cong H^i_{{\rm rig},c}(U_M,  D^{\dagger}_{M})_{(\varphi^r-\lambda^r){\rm -nilp}}.
In this case,  $H^i_{{\rm rig},c}(U_{F'},  D^{\dagger}_{F'})$ is a $\Lambda_0$-vector space and there is a natural isomorphism  
\[ H^i_{{\rm rig},c}(U_{F'},  D^{\dagger}_{F'})\otimes_{\QQ_p}\QQ_q \cong {\prod}_{\Gal(\QQ_q/\QQ_p)}H^i_{{\rm rig},c}(U_{F'},  D^{\dagger}_{F'}), 
 \] 
 with respect to which the endomorphism $\varphi\otimes\varphi$ of the left hand side corresponds to the following block matrix on the right hand side
\[
\left(
\begin{array}{ccccc}
	0 & {\bf1} & 0 & \cdots & 0\\
	0 & 0 & {\bf1} & \cdots & 0\\
	\vdots&\vdots&\ddots &\ddots &\vdots\\
	0&0& \cdots&0& {\bf1}\\
	\varphi^r&0& \cdots & 0& 0
\end{array} 
\right).
\]
One therefore obtains $Q_M$-equivariant isomorphisms
\[
{\prod}_{{\lambda'}^r=\lambda^r}H^i_{{\rm rig},c}(U_{F'},  D^{\dagger}_{F'})^{(\lambda')}\otimes_{\QQ_p}\QQ_q
\cong  H^i_{{\rm rig},c}(U_{F'},  D^{\dagger}_{F',\QQ_q})^{(\lambda^r)}, \]
where the left hand side is the product of generalised $\lambda'$-eigenspace for $\varphi$ %with $\lambda'^r = \lambda^r$   
and the right hand side is the  generalised $\lambda^r$-eigenspace for $\varphi^{(\QQ_q)}:=\varphi^r$. (Note that the underlying isocrystal for $D^{\dagger}_{F',\QQ_q}$ is $D^\dagger_{F'}$, equipped with the $\QQ_q$-linear $q$-Frobenius $\varphi^r$.)

Extending scalars from $\QQ_q$ to $\Lambda$, we now obtain an isomorphism 
\[
{\prod}_{\lambda'^r = \lambda^r} H^i_{{\rm rig},c}(U_{F'},  D^{\dagger}_{F'})^{(\lambda')}\otimes_{\QQ_p}\Lambda
\cong H^i_{{\rm rig},c}(U_{F'},  (D^{\dagger}_{F'})_\Lambda)^{(\lambda^r)} ,
\]
where the right hand side is the generalised $\lambda^r$-eigenspace for $\varphi^{(\Lambda)}:=\varphi^r\otimes\Lambda$, and so 
%\[
%{\rm det} _{\QQ_p} (H^i({\bf 1}-\varphi)) = {\rm det} _\Lambda(H^i({\bf 1}-\varphi^{(\Lambda)})).
%\]
%
\[
{\rm det} _{\QQ_p} (H^i({\bf 1}-t\cdot\varphi)) = {\rm det} _\Lambda(H^i({\bf 1}-t^{[k_\Lambda:\FF_p]}\cdot\varphi^{(\Lambda)})).
\]

Now, by the main theorem of Tsuzuki \cite[Th.~7.2.3]{tsuzuki98}, there exists an overconvergent unit-root $F$-isocrystal $\mathcal{O}^{\dagger}(\chi)$ over $U_{M}$ with monodromy given by $V_{\chi,\Lambda}$, viewed as a $\QQ_p[Q_M]$-module.  Furthermore, $\mathcal{O}^{\dagger}(\chi)$ has a natural action of $\Lambda$ commuting with the $p$-Frobenius operator $\varphi$ and the connection; that is, $\mathcal{O}^{\dagger}(\chi)$ is a $\QQ_p$-$F$-isocrystal with $\Lambda$-action in the sense of Definition~\ref{def:oc-Lambda} for $\Lambda_0=\QQ_p$. (Indeed, the $\Lambda$-action on the level of \emph{convergent} $\Lambda_0$-$F$-isocrystal is clear by construction since the $\Lambda$-action on $V_{\chi,\Lambda}$ commutes with the $\QQ_p[Q_M]$-action, and the $\Lambda$-action extends by the full faithfulness result \cite[5.1.1]{tsuzuki96}.)

We then obtain another $\QQ_p$-$F$-isocrystal $D^{\dagger}_{M}(\chi) := \mathcal{O}^{\dagger}(\chi)\otimes_{\QQ_p}D^\dagger_{M} $ with $\Lambda$-action and, in each degree $i$, we regard  
\[ H^i_{M}(\chi):= H^i_{{\rm rig},c}(U_{M}, D_{M}^{\dagger}(\chi)),\] 
as a $\Lambda$-vector space equipped with $\Lambda$-linear $p$-Frobenius operator $\varphi$.  
We then claim that there is an identity of functions
\begin{equation}\label{eq:lfun} 
Z_{U_M}(A_M,\chi, pt) = {\prod}_{i=0}^{i=2} {\rm det}_\Lambda(1-pt\cdot\varphi \mid H^i_{M}(\chi))^{(-1)^{i+1}}.
\end{equation}
Indeed, this identity is a standard consequence of Lefschetz trace formula for rigid cohomology of $\QQ_p$-$F$-isocrystals with $\Lambda$-action; \emph{cf.} Theorem~\ref{thm:Lefschetz}. (Its proof is a straightforward adaptation of the Lefschetz trace formula for $\Lambda$-$F$-isocrystals in \cite[Th.~6.3]{ELS}. In fact, in the special case that $M$ contains $k_\Lambda$, one can directly construct a $\Lambda$-$F$-isocrystal on $U_M$ that computes $Z_{U_M}(A_M,\chi, pt)$ via the more classical Lefschetz trace formula in \emph{loc. cit.})

Now, from Proposition \ref{prop:htfrobcomp}(ii) we know that, for both $i =0$ and $i =2$, the endomorphism $H^i(1-\varphi)$ is invertible on the $\Q_p$-linear dual $H^i_{M}(\chi)^\vee$ of $H^i_M(\chi)$. From the identity (\ref{eq:lfun}), we can therefore deduce that 
\begin{align}\label{key string}  r_{A,M}(\chi) =&\, {\rm dim}_{\Lambda}\bigl( \ker( 1-\varphi \mid H^1_{M}(\chi))\bigr) \\
=&\, {\rm dim}_{\Lambda}\bigl( \ker( 1-\varphi \mid H^1_{M}(\chi)^\vee )\bigr)\notag \\
 =&\, {\rm dim}_{\Lambda}\bigl( \ker( 1-\varphi \mid (H^1_{{\rm rig}, c}(U_{F'},(D_{F'}^{\dagger})_\Lambda)^\vee)^{\chi}\bigr)\notag\\
% =&\, {\rm dim}_{\Lambda}\bigl( \ker( 1-\varphi \mid (H^1(\Lambda\otimes_{\ZZ_p}P_{F'}^\ast)^\vee )^{\chi}\bigr)\notag
% \\
 =&\, {\rm dim}_{\Lambda}\bigl( \ker( 1-\varphi \mid H^1(\Lambda\otimes_{\ZZ_p}P_{F'})^{\chi}\bigr)\notag\\
 =&\, {\rm dim}_{\CC}(A^t(F')^{\chi}).\notag
 %=&\, \chi(1)^{-1}\cdot{\rm dim}_{\CC}(\wk{e_{\chi}}(\CC\otimes_\ZZ A^t(F'))).\notag
\end{align} 
Here the second equality is clear, the third follows from the isomorphism in Lemma \ref{lem:rigdes} below, the fourth from (\ref{missing complexes}) and the first and fifth from Proposition \ref{prop:htfrobcomp}(iii) and (iv). This proves claim (i). 
%{\color{red} in the above displayed formula, I changed $H^1_{{\rm rig}, c}(U_{F'},(D_{F'}^{\dagger})_\Lambda)^\ast$ to $H^1_{{\rm rig}, c}(U_{F'},(D_{F'}^{\dagger})_\Lambda)^\vee$}

Claim (ii) follows directly from Proposition \ref{prop:htfrobcomp}(iv) and the fact (already noted above) that $H^i(1-\varphi)$ is invertible on $H^i_{M}(\chi)^\vee $ for $i=0$ and $i =2$.

Next we note the equality (\ref{key string}) implies that 
\[ {\rm det}\bigl(1 - pt\cdot \varphi \mid \ker( 1-\varphi \mid H^1(P_{F'})^{\chi})\bigr) = (1-pt)^{r_{A,M}(\chi)}.\] 
Given this equality, and our chosen normalisation of leading terms, the formula in claim (iii) follows directly upon combining claim (ii) with the identity (\ref{eq:lfun}). 
\end{proof}

\begin{lemma} \label{lem:rigdes} For every absolutely irreducible representation $\chi: Q_M \rightarrow {\rm Aut}_{\Lambda}(V_{\chi,\Lambda})$ as above, and every degree $i$, there is a natural $\Lambda$-linear, Frobenius equivariant isomorphism 
\[ 
H^i_{{\rm rig}, c}(U_M, D_M^{\dagger}(\chi))^\vee \cong (H^i_{{\rm rig}, c}(U_{F'}, (D_{F'}^{\dagger})_\Lambda)^\vee)^{\chi}.
\]
%
%where $(-)^\vee$ denotes the $\Lambda$-linear dual on the left hand side and the contragredient on the right hand side.}{\color{red}what does this last assertion (about contragradient) mean? - can we perhaps just delete the phrase after the display??}
\end{lemma}

\begin{proof} All isomorphisms in the proof below can be checked to be Frobenius equivariant. Poincar\'e duality identifies the $\Lambda$-modules $H^i_{{\rm rig}, c}(U_M, D_M^{\dagger}(\chi))^\vee$ and $H^{i}_{{\rm rig}, c}(U_{F'}, (D_{F'}^{\dagger})_\Lambda)^\vee$ with $H^{2-i}_{\rm rig}(U_M, D_M^{\dagger}(\chi)^{\vee})$ and $H^{2-i}_{\rm rig} (U_{F'}, (D_{F'}^{\dagger,\vee})_\Lambda)$ respectively, where $D_M^{\dagger}(\chi)^{\vee}$ and $(D_{F'}^{\dagger,\vee})_\Lambda$ denote the dual as an overconvergent $F$-isocrystal and an overconvergent $\Lambda$-$F$-isocrystal respectively. It therefore suffices to prove there exists a natural isomorphism 
\begin{equation}\label{post pd}
%H^i_{\rm rig}(U_{F'}, (D_{F'}^{\dagger,\vee})_{\Lambda})^{\wk\chi} \cong (H^i_{\rm rig}(U_{F'}, (D_{F'}^{\dagger,%
%\vee})_{\Lambda})\otimes_{\Lambda}V_{\check\chi,\Lambda})^{Q_M} \cong  H^i_{\rm rig}(U_M, D_M^{\dagger}(\chi)^{\vee}).
H^i_{\rm rig}(U_{F'}, (D_{F'}^{\dagger,\vee})_{\Lambda})^{\chi} \cong H^i_{\rm rig}(U_M, D_M^{\dagger}(\chi)^{\vee}).
\end{equation}

To show this we use the canonical isomorphism $\pi_{F'/M}^*(D^{\dagger}_{M}) \cong D^{\dagger}_{F'}$, where $\pi_{F'/M}$ denotes the natural morphism $X_{F'} \rightarrow X_M$. We also note that the proof of \cite[Prop. 1.3]{CrewKloosterman} implies the overconvergent vector bundle $D^{\dagger}_{F'}$ has a natural $Q_M$-action that commutes with its natural Frobenius operator. (To see this, note that the natural $Q_M$-action and the Frobenius commute on the log Dieudonn\'e crystal $D_{F'}^{\log}$, and so the same must be true on the associated \emph{convergent} isocrystal. Then one need only note that, by \cite[5.1.1]{tsuzuki96}, the category of overconvergent $F$-isocrystals on $U_{F'}$ is naturally a \emph{full subcategory} of the category of convergent $F$-isocrystals on $U_{F'}$.)

Now, by construction of $D_M^{\dagger}(\chi)$, there is a natural isomorphism of $Q_M$-equivariant overconvergent $F$-isocrystals $\,\pi_{F'/M}^*(D_M^{\dagger}(\chi)^\vee) \cong V_{\check\chi,\Lambda}\otimes_{\Q_p}D_{F'}^{\dagger,\vee},$ 
 where $V_{\check\chi,\Lambda}$ is viewed as a $\Q_p[Q_M]$-module and $Q_M$ acts diagonally on the tensor product.  This isomorphism also respects the natural $Q_M$- and $p$-Frobenius equivariant $\Lambda$-actions on both sides.
 Hence, since the underlying overconvergent isocrystal for $V_{\check\chi,\Lambda}\otimes_{\Q_p}D_{F'}^{\dagger,\vee}$ coincides with that of $V_{\check\chi,\Lambda}\otimes_{\Lambda}(D_{F'}^{\dagger,\vee})_{\Lambda}$, one obtains the required isomorphism (\ref{post pd}) via the induced composite isomorphisms 
\begin{align*} H^i_{\rm rig}(U_M, D_M^{\dagger}(\chi)^{\vee})\xrightarrow\sim&\, H^i_{\rm rig}(U_{F'}, \pi_{F'/M}^\ast(D_M^{\dagger}(\chi)^{\vee}))^{Q_M}\\
\xrightarrow\sim&\, H^i_{\rm rig}(U_{F'}, V_{\check\chi,\Lambda}\otimes_{\Lambda}(D_{F'}^{\dagger,\vee})_\Lambda)^{Q_M}\\
%\xrightarrow\sim&\, \bigl( V_{\chi,\Lambda}\otimes_{\Lambda}H^i_{\rm rig}(U_{F'},(D_{F'}^{\dagger,\vee})_\Lambda)\bigr)^{Q_M}\\
\xrightarrow\sim&\, H^i_{\rm rig}(U_{F'},(D_{F'}^{\dagger,\vee})_\Lambda)^{\chi}.\end{align*}
Here the first map is induced by Shapiro's Lemma and its bijectivity is proved in \cite[Prop.~4.6]{trihan02} and the change from $\check\chi$ to $\chi$ that occurs in the third isomorphism is for the reason outlined in Remark \ref{contra just}. 
\end{proof}

\section{Proof of the main result}\label{sect:proof}  In this section we use results from earlier sections to obtain a proof of Theorem \ref{thm:thm1}. At the outset we note that Theorem \ref{thm:thm1}(i) is proved by Theorem \ref{order-leading}(i) and that Remark \ref{sha finite rem} allows us to assume that $\sha(A/L)$ is finite. We therefore focus on establishing the validity of the equality in Conjecture \ref{conj:ebsd}(iii).

For convenience, for each Galois extension $F'/M$ (as in Proposition \ref{perfect adapt}) we define an element of $K_0(\ZZ[Q_M],\QQ[Q_M])$ by setting
\begin{multline*} \chi (A,F'/M) := \partial_{Q_M}(Z^*_{U_M}((A_M)_{F'/M},p^{-1})) -\chi^{\rm BSD}_{Q_M,\QQ}(A_M,V_{F'})\\ +
\chi_{Q_M}^{\rm coh}(A_M,V_{F'}) - \chi_{Q_M}^{\rm sgn}(A_M),\end{multline*}
where, we recall, the leading term element is normalised via \eqref{eq:leading-term-Z}.

\subsection{A first reduction step} For a finite group $\Gamma$, a prime number $\ell$ and an element $x$ of $K_0(\ZZ[\Gamma],\QQ[\Gamma])$ we write $x_\ell$ for the image of $x$ in $K_0(\ZZ_\ell[\Gamma],\QQ_\ell[\Gamma])$ under the canonical decomposition (\ref{decomp}).

\begin{proposition}\label{reduction lemma} Assume $\sha(A/L)$ is finite. Then the statement of Theorem \ref{thm:thm1} is valid if and only if the following conditions are satisfied.
\begin{itemize}
\item[(i)] If $\mathfrak{M}_p$ is any given maximal $\ZZ_p$-order in $\QQ_p[G]$ that contains $\ZZ_p[G]$, then $\chi (A,L/K)_p$ belongs to the kernel of the homomorphism $K_0(\ZZ_p[G],\QQ_p[G]) \to K_0(\mathfrak{M}_p,\QQ_p[G])$.
\item[(ii)] Assume that the set $\Sigma_1\cup\Sigma_2$ in (\ref{calT def}) is non-empty. Fix a field $K' \in \Sigma_1\cup\Sigma_2$, set $L' = LK'$ and write $P'$ for the normal subgroup of $H' := \Gal(L'/K')$ that is generated by the Sylow $p$-subgroups of the inertia groups of all places that ramify in $L'/K'$. Then $\chi(A, (L')^{P'}/K')_p$ vanishes.
\item[(iii)] For each prime $\ell\not= p$ one has
\[ \partial_{G,\QQ}(Z^*_{U}(A_{L/K},p^{-1}))_\ell = \chi^{\rm BSD}_{G,\QQ}(A,V_L)_\ell - \chi_G^{\rm sgn}(A)_\ell .\]
\end{itemize}
\end{proposition}

\begin{proof} It suffices to check that the stated conditions are equivalent to the validity of the equality in Conjecture \ref{conj:ebsd}(iii).

Thus, after taking account of Proposition \ref{bsd-schneider}, the decomposition (\ref{decomp}) combines with the explicit definition of the subgroup $\mathcal{T}_{A,L/K}$ to reduce us to showing that the stated conditions imply the validity of each of the following assertions:
\begin{itemize}
\item[(C$_1$)] $\chi(A,L/K)_p$ has finite order;
\item[(C$_2$)] for every field $K'$ that belongs to either $\Sigma_1$ or $\Sigma_2$, $\chi(A,L/K)_p$ is  the image under $\pi^{G'}_{G}$ of an element of $K_0(\ZZ_p[G'],\QQ_p[G'])$ that belongs to $\ker(\pi^{G'}_{H'/P'})$;
\item[(C$_3$)] $\chi(A,L/K)_\ell$ vanishes if $\ell \not= p$.
\end{itemize}

To check this, we first recall (from \cite[\S4.5, Lem. 11(d)]{BF_Tamagawa}) that  $K_0(\ZZ_p[G],\QQ_p[G])_{\rm tor}$ is equal to the kernel of the scalar extension homomorphism $K_0(\ZZ_p[G],\QQ_p[G]) \to K_0(\mathfrak{M}_p,\QQ_p[G])$. Given this fact,  condition (i) directly implies that $\chi(A,L/K)_p$ has finite order, and hence verifies (C$_1$). 

Next, we note that, for any $K' \in \Sigma_1\cup\Sigma_2$, the result of Proposition \ref{group independent} below implies (in terms of the notation of condition (ii)) that one has both $\chi(A,L/K)_p = \pi^{G'}_{G}(\chi(A,L'/K)_p)$ and $\pi^{G'}_{H'/P'}(\chi(A,L'/K)_p) = \chi(A,(L')^{P'}/K')_p$. In particular, in this case, condition (ii) implies $\chi(A,L'/K)_p$ belongs to $\ker(\pi^{G'}_{H'/P'})$, as required to verify (C$_2$). 

Finally, to verify (C$_3$) we note that if $\ell\not= p$, then $\chi^{\rm coh}_G(A,V_L)_\ell$ vanishes. Thus, in this case, the vanishing of the image in $K_0(\ZZ_\ell[G],\QQ_\ell[G])$ of the equality in Conjecture \ref{conj:ebsd}(iii) is clearly equivalent to the equality stated in condition (iii).
\end{proof}

Before stating the next result we note that if $J$ is a normal subgroup of a subgroup $H$ of $G$, and we set $Q := H/J$, then there is a natural commutative diagram
\begin{equation}\label{func diagram}
\xymatrix{
K_1(\mathbb{Q}[G]) \ar[r]^{\theta^1_{G,H}} \ar[d]_{\partial_{G,\QQ}} & K_1(\mathbb{Q}[H]) \ar[r]^{\theta^1_{H,Q}} \ar[d]^{\partial_{H,\QQ}} & K_1(\mathbb{Q}[Q]) \ar[d]^{\partial_{Q,\QQ}}\\
K_0(\mathbb{Z}[G], \mathbb{Q}[G]) \ar[r]^{\theta^0_{G,H}} & K_0(\mathbb{Z}[H], \mathbb{Q}[H]) \ar[r]^{\theta^0_{H,Q}} & K_0(\mathbb{Z}[Q], \mathbb{Q}[Q])}
\end{equation}
where $\theta^i_{G,H}$ and $\theta^i_{H,Q}$ are the natural restriction and coinflation homomorphisms.

\begin{proposition}\label{group independent} If $J$ is a normal subgroup of a subgroup $H$ of $G$, with $Q = H/J$, then the composite homomorphism $\theta^0_{H,Q}\circ \theta^0_{G,H}$ sends $\chi(A,L/K)_p$ to $\chi(A_{L^H},L^J/L^H)_p$.
\end{proposition}

\begin{proof} We set $\theta^i_{G,Q} := \theta^i_{H,Q}\circ \theta^i_{G,H}$, $E:= L^H$ and $F := L^J$.

At the outset we note that, by a standard argument using the Artin formalism of $L$-functions, one finds that $\theta^1_{G,Q}(Z^*_{U}(A_{L/K},p^{-1})) = Z^*_{U_E}((A_E)_{F/E},p^{-1})$ and so the commutative diagram (\ref{func diagram}) implies
\begin{equation}\label{artin}\theta^0_{G,Q}(\partial_{G,\Q}(Z^*_{U}(A_{L/K},p^{-1}))) = \partial_{Q,\QQ}(Z^*_{U_E}((A_E)_{F/E},p^{-1})).\end{equation}

It is also clear that $\theta^1_{G,Q}(\langle \QQ\cdot A^t(L),-1\rangle) = \langle  \QQ\cdot A^t(F),-1\rangle$ and, given this, an explicit comparison of the equalities in Proposition \ref{prop:htfrobcomp}(i) with $F'$ equal to $L$ and $F$ implies
\begin{equation}\label{signs} \theta^0_{G,Q}(\chi^{\rm sgn}_G(A)_p) = \chi^{\rm sgn}_Q(A_E)_p.\end{equation}

To proceed we write $\pi, \pi'$ and $\pi''$ for the natural morphisms  $X_L \rightarrow X$, $X_L \rightarrow X_E$ and $X_E \to X$. We fix families of subgroups $V_L$ and $W_L$ for the extension $L/K$ as in \S\ref{selmer section} (the choice of which is, following Proposition \ref{V independent}, unimportant) and write $\mathcal{L}_L$ for the associated coherent $\mathcal{O}_X[G]$-submodule of
$\pi_*{\rm Lie}(\mathcal{A}_{X_L})$. In the same way we fix families of subgroups $V'_L$ and $W'_L$ for the extension $L/E$ and write $\mathcal{L}_L'$ for the associated coherent $\mathcal{O}_{X_E}[H]$-submodule of $\pi'_*{\rm Lie}(\mathcal{A}_{X_L})$.

We assume, as we may, that $V_L \subseteq V_L'$, and hence also $W_L \subseteq W_L'$. This implies that there are exact triangles in $D^{\rm perf}(\ZZ[H])$ of the form
%%
%\[{\rm SC}_{V_L}(A,L/K) \to {\rm SC}_{V_L'}(A_E, L/E) \to (V_L'/V_L)[0] \to \]
%%
%and
%%
%\[ R\Gamma(X, \mathcal{L}_L) \rightarrow R\Gamma(X_{E}, \mathcal{L}_L') \to (W_L'/W_L)[0] \to ,\]
%%
%
\[{\rm SC}_{V_L'}(A,L/E) \to {\rm SC}_{V_L}(A_E, L/K) \to (V_L'/V_L)^*[-1] \to \]
and
\[ R\Gamma(X_E, \mathcal{L}_L')^* \rightarrow R\Gamma(X, \mathcal{L}_L)^* \to (W_L'/W_L)^*[1]\to ,\]
where, in the latter case, we have used the fact that the complexes $R\Gamma(X, \pi''_*\mathcal{L}_L')$ and $R\Gamma(X_E, \mathcal{L}_L')$ are canonically isomorphic since $\pi_*''$ is exact. These triangles in turn give rise to equalities in $K_0(\ZZ[H],\QQ[H])$
%%
%\begin{align}\label{restrict} \theta^0_{G,H}(\chi^{\rm BSD}_{G,\QQ}(A,V_L) - \chi^{\rm coh}_G(A,V_L))
%= \,\, &\chi^{\rm BSD}_{H,\QQ}(A,V'_L) - \chi_{\ZZ_p[H]}((V_L'/V_L)[0],0) \\
%&\hskip 0.2truein - (\chi^{\rm coh}_{H}(A,V'_L) - \chi_{\ZZ_p[H]}((W_L'/W_L)[0],0))\notag\\
%= \,\, &\chi^{\rm BSD}_{H,\QQ}(A,V'_L)) - \chi_H^{\rm coh}(A,V'_L),\notag\end{align}
%%
%
\begin{align}\label{restrict} 
\theta^0_{G,H}(\chi^{\rm BSD}_{G,\QQ}(A,V_L) - \chi^{\rm coh}_G(A,V_L))
= \,\, &\chi^{\rm BSD}_{H,\QQ}(A,V'_L) + \chi_{\ZZ_p[H]}((V_L'/V_L)^*[-1],0) \\
&\hskip 0.2truein - (\chi^{\rm coh}_{H}(A,V'_L) + \chi_{\ZZ_p[H]}((W_L'/W_L)^*[-1],0))\notag\\
= \,\, &\chi^{\rm BSD}_{H,\QQ}(A,V'_L)) - \chi_H^{\rm coh}(A,V'_L),\notag\end{align}
where the last equality is valid since $\chi_{\ZZ_p[H]}((V_L'/V_L)^*[-1],0) = \chi_{\ZZ_p[H]}((W_L'/W_L)^*[-1],0)$ (by the same argument as used in the proof of Proposition \ref{V independent}).

Upon combining the equalities (\ref{artin}), (\ref{signs}) and (\ref{restrict}) one finds that the proof is reduced to showing that there are equalities
\[ \begin{cases} &\theta^0_{H,Q}(\chi^{\rm BSD}_{H,\QQ}(A_E,V'_L)) = \chi_{Q,\QQ}^{\rm BSD}(A_E,(V'_L)^J),\\
&\theta^0_{H,Q}(\chi^{\rm coh}_H(A_E,V'_L)) = \chi_Q^{\rm coh}(A_E,(V'_L)^J).\end{cases}\]

These equalities follow directly from the isomorphisms in $D^{\rm perf}(\ZZ[Q])$
\begin{equation}\label{key descent}\begin{cases} &\mathbb{Z}[Q]\otimes_{\mathbb{Z}[H]}^{\mathbb{L}} {\rm SC}_{V'_L}(A_E, L/E) \cong {\rm SC}_{(V'_L)^J}(A_E, F/E),\\
&\mathbb{Z}[Q]\otimes_{\mathbb{Z}[H]}^{\mathbb{L}} R\Gamma(X_E, \mathcal{L}'_L)^* \cong R\Gamma(X_E, (\mathcal{L}'_L)^J)^*,\end{cases}
\end{equation}
that are respectively used in the proofs of Proposition \ref{prop:perfect} and Lemma \ref{zar lemma}.\end{proof}

\subsection{The case $\ell = p$}\label{p section} In this section we verify that the conditions (i) and (ii) in Proposition \ref{reduction lemma} are satisfied.

The key observation we shall use in this regard is provided by the following result. In this result we use the notation and hypotheses of Proposition \ref{reduction lemma}(ii).

\begin{lemma}\label{stronger prop} Fix a field $K'$ in $\Sigma_1\cup \Sigma_2$ (so that, by assumption, $\sha(A/L')$ is finite) and a Galois extension of fields $M_2/M_1$ with  $K\subseteq M_1\subseteq M_2\subseteq L'$. Set $J:=\Gal(M_2K'/M_1)$ and $Q := \Gal(M_2/M_1)$.
Also fix a $\ZZ_p$-order  $\mathfrak{N}$ in $\QQ_p[J]$ as in Proposition \ref{perfect adapt} with $F'= M_2K'$ and $M= M_1$, and write $\overline{\mathfrak{N}} $ for the image of $\mathfrak{N}$ in $\QQ_p[Q]$.

Then $\chi (A,M_2/M_1)_p$ belongs to the kernel of the natural scalar extension homomorphism $K_0(\ZZ_p[Q],\QQ_p[Q]) \to K_0(\overline{\mathfrak{N}},\QQ_p[Q])$.  
\end{lemma}

\begin{proof} Under the present hypotheses, the exact triangles in Proposition \ref{perfect adapt} lie in $D^{\rm perf}(\mathfrak{N})$. Hence, after taking account of the relevant cases of the isomorphisms (\ref{key descent}), the exact functor $\Delta(-) := \overline{\mathfrak{N}} \otimes_{\mathfrak{N}}^{\mathbb{L}} -$ takes these triangles to exact triangles in $D^{\rm perf}(\overline{\mathfrak{N}})$ of the form 
\[ \Delta\bigl(\tau_{\ge -1}(\mathfrak{N}\otimes_{\ZZ_p[J]}^\mathbb{L}P_{F'})\bigr) \xrightarrow{{\bf 1}-\varphi} \Delta\bigl(\tau_{\ge -1}(\mathfrak{N}\otimes_{\ZZ_p[J]}^\mathbb{L}I_{F'})\bigr) \to \overline{\mathfrak{N}}\otimes_{\ZZ_p[Q]}^\mathbb{L}R\Gamma_{{\rm ar},V_{M_2}}(U_{M_2},\mathcal{A}_{\rm tor})^*_p[-2] \rightarrow \]
\[ \Delta\bigl(\tau_{\ge -1}(\mathfrak{N}\otimes_{\ZZ_p[J]}^\mathbb{L}P_{F'})\bigr) \xrightarrow{{\bf 1}} 
\Delta\bigl(\tau_{\ge -1}(\mathfrak{N}\otimes_{\ZZ_p[J]}^\mathbb{L}I_{F'})\bigr) \to \overline{\mathfrak{N}}\otimes_{\ZZ_p[Q]}^\mathbb{L}R\Gamma(X_{M_1}, (\mathcal{L}')^{H})^*[-2]\rightarrow,\]
with $H := \Gal(M_2K'/M_2)$. These triangles satisfy all of the conditions  (a), (b) and (c) of Proposition \ref{prop:triangle}: in fact, the only condition that is not straightforward to check in this case is (b) and this follows from the results of Proposition \ref{prop:htfrobcomp}(ii) and (iv).

In particular, by applying Proposition \ref{prop:triangle} in this context, and taking account of the equality in Proposition \ref{prop:htfrobcomp}(i), one finds that the image of $\chi(A,M_2/M_1)_p$ in $K_0(\overline{\mathfrak{N}},\QQ_p[Q])$ is equal to the image under the natural connecting homomorphism $K_1(\QQ_p[Q])\to K_0(\overline{\mathfrak{N}},\QQ_p[Q])$ of the product element 
\begin{equation}\label{prod element} Z^*_{U_{M_1}}((A_{M_1})_{M_2/M_1},p^{-1}) \cdot {\prod}_{i=0}^{i=2} (H^i({\bf 1}-\varphi_{M_2/M_1})_{\QQ_p}^{\diamond})^{(-1)^i} \in K_1(\QQ_p[Q]).\end{equation}
Here we write ${\bf 1}-\varphi_{M_2/M_1}$ for the morphism denoted by ${\bf 1}-\varphi$ in the first of the exact triangles displayed above, and identify each automorphism $H^i({\bf 1}-\varphi_{M_2/M_1})_{\QQ_p}^{\diamond}$ with the induced element of $K_1(\QQ_p[Q])$.  

It is thus enough to prove that the element (\ref{prod element}) vanishes, or equivalently, that its image under the  (injective) map ${\rm Nrd}_{\QQ_p[Q]}: K_1(\QQ_p[Q])\to \zeta(\QQ_p[Q])^\times$ is trivial. In addition, given the  characterisation of 
$Z^*_{U_{M_1}}((A_{M_1})_{M_2/M_1},p^{-1})$ in Proposition \ref{bsd-schneider}(i), the required triviality is deduced  directly from the formula of Theorem \ref{order-leading}(iii) (with $F'/M$ replaced by $M_2/M_1$) for every $\chi\in {\rm Ir}(Q)$ and the fact that, in terms of the notation of the corresponding case of Theorem \ref{order-leading}(ii), for every $i \in \{0,1,2\}$ and $\chi\in {\rm Ir}(Q)$ one has 
\[ {\rm Nrd}_{\QQ_p[Q]}(H^i({\bf 1}-\varphi_{M_2/M_1})_{\QQ_p}^{\diamond})_\chi = {\rm det}(H^i({\bf 1}-\varphi_{M_2/M_1})_\chi^{\diamond}).\]
\end{proof}

Turning now to consider the conditions in Proposition \ref{reduction lemma}, we first fix a maximal $\ZZ_p$-order $\mathfrak{N}$ in $\QQ_p[G']$ that contains $\ZZ_p[G']$. Then $\mathfrak{N}$ is regular and so satisfies the conditions of Proposition \ref{perfect adapt} with $F'=L'$ and $M = K$ (so $J= G'$). From Lemma \ref{stronger prop} (with $M_2=L$ and $M_1 = K$, so $Q=G$), it therefore follows that $\chi (A,L/K)_p$ belongs to the kernel of the scalar extension $K_0(\ZZ_p[G],\QQ_p[G]) \to K_0(\mathfrak{M},\QQ_p[G])$, where $\mathfrak{M}$ denotes the image of $\mathfrak{N}$ in $\QQ_p[G]$. In particular, since $\mathfrak{M}$ is a maximal $\ZZ_p$-order in $\QQ_p[G]$ that contains $\ZZ_p[G]$, this shows that the condition of Proposition \ref{reduction lemma}(i) is satisfied.

Next we consider condition (ii) of Proposition \ref{reduction lemma}. To do this we note that, by our assumption on $K'$, the group $\sha(A/L')$ is finite. In addition, the field $F' := (L')^{P'}$ is a tamely ramified Galois extension of $K'$ and so Proposition \ref{I tame} implies that the conditions of Proposition \ref{perfect adapt} are satisfied by the data $J = \Gal(F'/K')$ and $\mathfrak{N} = \ZZ_p[J]$. In this case, therefore, Lemma \ref{stronger prop} implies that $\chi(A,(F')^{P'}/K')_p$ vanishes, and hence that condition (ii) of Proposition \ref{reduction lemma} is satisfied.

\begin{remark}\label{stronger main}{\em A close reading of the above argument shows that we actually prove a (possibly) finer version of Theorem \ref{thm:thm1}(ii). Specifically, the validity of the equality in Conjecture \ref{conj:ebsd}(iii) is proved modulo the subgroup of $\mathcal{T}_{A,L/K}$ that is obtained by replacing the group $K_0(\ZZ_p[G],\QQ_p[G])_{\rm tor}$ in the intersection (\ref{calT def}) by its subgroup  
\[ \ker \bigl(K_0(\ZZ_p[G],\QQ_p[G])_{\rm tor} \xrightarrow{ (\lambda_{\mathfrak{N}})_{\mathfrak{N}}} {{\bigoplus}}_{\mathfrak{N}} K_0(\overline{\mathfrak{N}},\QQ_p[G])\bigr).\]
Here in the intersection $\mathfrak{N}$ runs over all $\ZZ_p$-orders of $\QQ_p[G']$ that contain $\ZZ_p[G']$ and satisfy the hypotheses of Proposition \ref{perfect adapt} (with $F' = L'$ and $M = K$), $\overline{\mathfrak{N}}$ is the image of $\mathfrak{N}$ in $\QQ_p[G]$ and each $\lambda_{\mathfrak{N}}$ is the scalar extension map that arises from the inclusion $\ZZ_p[G]\subseteq \overline{\mathfrak{N}}$. We recall that the hypotheses of Proposition \ref{perfect adapt} are automatically satisfied if the order $\mathfrak{N}$ is hereditary but that, aside from this, finding other interesting, and explicit, examples of such orders (beyond those that are used in the above argument) seems  difficult.}\end{remark}

\subsection{The case $\ell \ne p$.} In this section we verify that condition (iii) in Proposition \ref{reduction lemma} is satisfied, and thereby complete the proof of Theorem \ref{thm:thm1}.

To do this we fix a prime $\ell\not= p$, write $T_\ell(\mathcal{A})$ for the $\ell$-adic Tate module of $\mathcal{A}$ and set $V_\ell(\mathcal{A}) = \QQ_\ell\otimes_{\ZZ_\ell}T_\ell(\mathcal{A})$. We also write $\F^c_p$ for its algebraic closure of $\F_p$ and $\varphi_p$ for the Frobenius automorphism at $p$ and set $U^c_L := U_L\times_{\F_p} \F^c_p$.

For each $\chi\in {\rm Ir}(G)$ we fix an associated representation space $V_\chi$ over $\CC_\ell$. For each finitely generated $\Q_\ell[G]$-module $W$, we set
\[
W^\chi := \Hom_{\Q_\ell[G]}(V_\chi,\CC_\ell[G]\otimes_{\Q_\ell[G]}W).
\] 
Then by repeating the proof of Lemma~\ref{lem:rigdes} for $\ell$-adic cohomology in place of rigid cohomology, we obtain isomorphisms 
\[ H^i_{\et , c}(U^c,V_{\chi}\otimes_{\QQ_\ell}V_\ell(\mathcal{A}))^\vee 
\cong \big( H^i_{\et , c}(U_L^c,V_\ell(\mathcal{A}))^\vee\big)^\chi \cong 
H^i_{\et}(U_L^c,V_\ell(\mathcal{A}^t))^\chi ,
\]
where the second isomorphism is induced by the Poincar\'{e} duality theorem (as stated, for example, in \cite[Chap. VI, Cor. 11.2]{milne:etale}). Therefore the identity (\ref{deligne}) implies that 
\begin{align}\label{groth} Z_U(A, \chi, p^{-1}t) = &{\prod}_{i\in \bz}
{\rm det}\big(1-\varphi_p\cdot t: H^i_{\et , c}(U^c,V_{\chi}\otimes_{\QQ_\ell}V_\ell(\mathcal{A}))\big)^{(-1)^{i+1}}\\
= &{\prod}_{i\in \bz}
{\rm det}\big(1-\varphi_p\cdot t: H^i_{\et , c}(U^c,V_{\chi}\otimes_{\QQ_\ell}V_\ell(\mathcal{A}))^\vee \big)^{(-1)^{i+1}}\notag\\
= &{\prod}_{i\in \bz}{\rm det}\big(1-\varphi_p\cdot t:  H^i_{\et}(U_L^c,V_\ell(\mathcal{A}^t))^\chi\big)^{(-1)^{i+1}}.\notag
\end{align}

We now set ${\rm SC}_\ell := \ZZ_\ell\otimes_\ZZ {\rm SC}_{V_L}(A,L/K)$. Then the result of Proposition \ref{prop:perfect}(ii)(b) combines with Remark \ref{flat-etale} and the Artin-Verdier duality theorem to imply there are natural isomorphisms
\[ {\rm SC}_\ell\cong \ZZ_\ell\otimes_\ZZ R\Gamma_{{\rm ar},V_L}(U_L, \calA\{\ell\})^*[2] \cong R\Gamma_{\et,c}(U_L, \calA\{\ell\})^*[2] \cong R\Gamma_{\et}(U_L, T_\ell(\mathcal{A}^t))\]
and hence also a natural exact triangle in $D^{\rm perf}(\ZZ_\ell[G])$ of the form
\begin{equation}\label{semi tri}  {\rm SC}_\ell \to R\Gamma_{\et}(U^c_L, T_{\ell}(\mathcal{A}^t)) \xrightarrow{1-\varphi_p} R\Gamma_{\et}(U^c_L, T_{\ell}(\mathcal{A}^t)) \to {\rm SC}_\ell[1].\end{equation}

We consider the composite homomorphism
\[ \beta_{A,L,\ell}: \QQ_\ell\otimes_\ZZ A^t(L) \cong H^0({\rm SC}_\ell)_{\QQ_\ell} \to H^0_{\et}(U^c_L, V_{\ell}(\mathcal{A}^t))
\to H^1({\rm SC}_\ell)_{\QQ_\ell} \cong \QQ_\ell\otimes_\ZZ \Hom_\ZZ (A(L),\ZZ),\]
where the isomorphisms are from Proposition \ref{prop:perfect}(ii)(a) and the other maps are induced by the long exact cohomology sequence of (\ref{semi tri}).

Then it is shown by Schneider in \cite{schneider} (and also noted at the beginning of \cite[\S6.8]{KT}) that there exists a computable integer $ a_{A,L,\ell}\in \{0,1\}$ such that
\begin{equation}\label{ell sign} \beta_{A,L,\ell} = (-1)^{a_{A,L,\ell}}\cdot h_{A,L,\ell,*}\end{equation}
where $h_{A,L,\ell,*}$ is the isomorphism $\QQ_\ell\otimes_\ZZ A^t(L)\cong \QQ_\ell\otimes_\ZZ \Hom_\ZZ (A(L),\ZZ)$ induced by the height pairing $h_{A,L}$.

Taken in conjunction with the same argument used in Proposition \ref{prop:htfrobcomp} this observation implies firstly that the endomorphism $H^i(1-\varphi_p)_{\QQ_\ell}$ is bijective for $i \not= 1$, secondly that (\ref{semi tri}) satisfies all of the hypotheses  of Proposition \ref{prop:triangle} (with $\mathfrak{A} = \ZZ_\ell[G])$ regarding the left hand triangle in (\ref{triangles}), and thirdly (in view of (\ref{groth})) that 
\[ {\rm ord}_{t = p^{-1}}\bigl(Z_U(A,\chi, t)\bigr) = {\rm dim}_{\CC_\ell}\bigl(\ker \bigl(H^1(1-\varphi_p)\mid \Hom_{\CC_\ell[G]}(V_\chi,\CC_\ell\otimes_{\QQ_\ell} H^i_{\et}(U_L^c,V_\ell(\mathcal{A}^t))\bigr)\bigr).\]
 %
 %that is induced by $H^1(1-\varphi_p)$.

By applying Proposition \ref{prop:triangle} with the left and right hand triangles in (\ref{triangles}) taken to be (\ref{semi tri}) and the zero triangle respectively we can therefore deduce that
\begin{align*} \iota_{G,\ell}(\chi^{\rm BSD}_{G,\QQ}(A,V_L) + \chi^{\rm sgn}_G(A)) &= \chi_{\ZZ_\ell[G]}({\rm SC}_\ell,h^{\rm det}_{A,L,\ell,*}) + \partial_{\ZZ_\ell[G],\QQ_\ell}(\beta_{A,L,\ell}\circ h^{-1}_{A,L,\ell,*})\\
&= \chi_{\ZZ_\ell[G]}({\rm SC}_\ell,\tau_{1-\varphi_p}) \\
&= \partial_{\ZZ_\ell[G],\QQ_\ell}(H^1(1-\varphi_p)_{\QQ_\ell}^\diamond)\\
&= \partial_{\ZZ_\ell[G],\QQ_\ell}(({\rm Nrd}_{\QQ_\ell[G]})^{-1}((Z^\ast_U(A,\chi, p^{-1}))_{\chi\in {\rm Ir}(G)} ))\\
&= \partial_{\ZZ_\ell[G],\QQ_\ell}(Z^*_{U}(A_{L/K},p^{-1}))\\
 &= \iota_{G,\ell}(\partial_{G,\QQ}(Z^*_{U}(A_{L/K},p^{-1}))).
\end{align*}
Here the first equality follows directly from the definition of $\chi^{\rm sgn}(A,L/K)_\ell$ in terms of the integer $a_{A,L,\ell}$, the equality (\ref{ell sign}) and the result of Lemma \ref{sign lemma}. In addition, the fourth equality follows from (\ref{groth}), the fifth directly from the definition of the term $Z^*_{U}(A_{L/K},p^{-1})$ and all remaining equalities are clear.

This argument completes the proof that condition (iii) in Proposition \ref{reduction lemma} is satisfied and hence also, when combined with the observations made in \S\ref{p section}, completes the proof of Theorem \ref{thm:thm1}.

\subsection{The proof of Propositions \ref{cons prop} and \ref{cons prop 2}}\label{cons props proof}

Throughout this section, we shall use the notation of \S\ref{cons2}. 

\subsubsection{The proof of Proposition \ref{cons prop}} As a first step, we recall that Proposition \ref{prop:perfect}(i) implies ${\rm SC}_{V_L}$ belongs to $D^{\rm perf}(\ZZ[G])$ and is acyclic outside degrees $0, 1$ and $2$. In this case, therefore, the construction of resolutions used in the proofs of Lemma \ref{prop:zzhatcom}(iii) and Proposition \ref{perfect adapt} implies ${\rm SC}_{V_L}$ is isomorphic in $D(\ZZ[G])$ to a complex 
\begin{equation}\label{explicit comp} P_{-1}\xrightarrow{d^{-1}} P_0 \xrightarrow{d^{0}} P_1 \xrightarrow{d^1} P_2\end{equation}
in which $P_{-1}$ is a finitely generated projective $\ZZ[G]$-module that is placed in degree $-1$ and all other modules $P_i$ are finitely generated and free. By taking the direct sum with complexes of the form $\ZZ[G]\xrightarrow{1}\ZZ[G]$, with the first term placed in appropriate degrees, one can also assume that the $G$-rank ${\rm rk}_G(P_{i})$ of $P_{i}$ is greater than $1$ for every $i$. 

To prove claim (i) it is therefore enough to show that the $G$-module $P_{-1}$ is free, or equivalently (by the Bass Cancellation Theorem \cite[Th. (41.20)]{curtisr}, since ${\rm rk}_G(P_{-1}) > 1$) that the Euler characteristic $\chi_G({\rm SC}_{V_L})$ of ${\rm SC}_{V_L}$ in $K_0(\ZZ[G])$ vanishes. In addition, writing $\partial'_{G}$ for the connecting homomorphism  $K_0(\ZZ[G],\RR[G]) \to K_0(\ZZ[G])$, the (assumed) equality in Conjecture \ref{conj:ebsd}(iii) implies that   
\begin{align*}\chi_G({\rm SC}_{V_L}) =&\,\, \partial_G'\bigl(\chi_G^{\rm BSD}(A,V_{L}))\\
                           =&\,\, \partial_G'\bigl(\partial_{G}( L^*_{U}(A_{L/K}, 1))\bigr) + \partial_G'\bigl(\chi_G^{\rm coh}(A,V_{L})\bigr) - 
                           \partial_G'\bigl(\chi_G^{\rm sgn}(A)\bigr)\\
                           =&\,\, \partial_G'\bigl(\chi_G^{\rm coh}(A,V_{L})\bigr) - 
                           \partial_G'\bigl(\chi_G^{\rm sgn}(A)\bigr),\end{align*}
where the final equality follows directly from the fact that $\partial'_G\circ\partial_G$ is the zero map.

To prove claim (i), we are therefore reduced to showing that if $G$ has $p$-power order, then the last two terms in the above expression vanish. However, in this case, every finite projective $\mathbb{F}_p[G]$-module is free so that the image of the homomorphism (\ref{p-subgroup}) belongs to the kernel of $\partial'_{G}$ and hence $\partial_G'\bigl(\chi_G^{\rm coh}(A,V_{L})\bigr)$ automatically vanishes. In addition, the term $\partial_G'\bigl(\chi_G^{\rm sgn}(A)\bigr)$ vanishes since Lemma \ref{sign lemma} below implies that $\chi_G^{\rm sgn}(A)$ is equal to 
\begin{equation}\label{sign p equality} \partial_{G,\QQ}(\langle \QQ\cdot A^t(L),(-1)^{a_p}\rangle)_p = \partial_{G,\QQ}(\langle \QQ\cdot A^t(L),(-1)^{a_p}\rangle) = \partial_{G}(\langle \RR\cdot A^t(L),(-1)^{a_p}\rangle).\end{equation} 
This proves claim (i). 
 
To prove claim (ii) we note that $A(K)[p] = A(L)[p]^G$ and $A^t(K)[p] = A^t(L)[p]^G$. Hence, if $A(K)[p]$ and $A^t(K)[p]$ vanish, then $A(L)[p]$ and $A^t(L)[p]$ also vanish since $G$ is a $p$-group. In this case, therefore, Proposition \ref{prop:perfect}(i) implies ${\rm SC}_{V_L,(p)}$ is acyclic outside degrees $0$ and $1$ and $H^0({\rm SC}_{V_L,(p)})$ is torsion-free. This in turn implies that ${\rm SC}_{V_L,(p)}$ is isomorphic in $D(\ZZ_{(p)}[G])$ to a complex of projective $\ZZ_{(p)}[G]$-modules of the form (\ref{explicit comp}) in which $P_2$ vanishes and so there are exact sequences of $\ZZ_{(p)}[G]$-modules 
\[ 0 \to P_{-1} \to P_0 \to \cok(d^{-1})_{(p)} \to 0\]
and
\[ 0 \to H^0({\rm SC}_{V_L,(p)}) \to {\rm cok}(d^{-1})_{(p)} \xrightarrow{d^0} P_{1} \to H^1({\rm SC}_{V_L,(p)}) \to 0.\]
The first of these sequences implies $\cok(d^{-1})_{(p)}$ is a c-t $G$-module and the second implies it is torsion-free. These two properties combine to imply ${\rm cok}(d^{-1})_{(p)}$ is a projective $\ZZ_{(p)}[G]$-module (by \cite[Th. 8]{aw}). 

At this stage we therefore know that ${\rm SC}_{V_L,(p)}$ is isomorphic in $D(\ZZ_{(p)}[G])$ to a complex of $\ZZ_{(p)}[G]$-modules $\cok(d^{-1})_{(p)}\to P_1$ in which the first term is projective and the second is free (and of rank greater than $1$). To see that this is a complex of the required form it is then enough to note that, since the Euler characteristic of ${\rm SC}_{V_L,(p)}$ in $K_0(\ZZ_{(p)}[G])$ vanishes, the Bass Cancellation Theorem implies that the module $\cok(d^{-1})_{(p)}$ is isomorphic to $P_1$. 

This completes the proof of Proposition \ref{cons prop}. 

\begin{lemma}\label{sign lemma} If $\ell$ is any prime that does not divide $\#G$, then $\partial_{G,\QQ}(\langle \QQ\cdot A^t(L),-1\rangle)_\ell$ vanishes. \end{lemma}

\begin{proof} If $\ell$ does not divide $\# G$, then the $\ZZ_\ell$-order $\ZZ_\ell[G]$ is maximal and so $\QQ_\ell\otimes_\ZZ A^t(L)$ has a full sublattice that is a projective $\ZZ_\ell[G]$-module. This implies $\langle  \QQ_\ell\otimes_\ZZ A^t(L),-1\rangle$ belongs to the image of the natural map $K_1(\ZZ_\ell[G]) \to K_1(\QQ_\ell[G])$ and hence that the element $\partial_{G,\QQ}(\langle  \QQ\cdot A^t(L),-1\rangle)_\ell = \partial_{\ZZ_\ell[G],\QQ_\ell}(\langle  \QQ_\ell\otimes_\ZZ A^t(L),-1\rangle)$ vanishes as a consequence of the long exact sequence of relative $K$-theory (see (\ref{E:kcomm})). \end{proof}

\subsubsection{The proof of Proposition \ref{cons prop 2}} We abbreviate the connecting homomorphism $\partial_{\ZZ_{(p)}[G],\RR}$ to $\partial_p$ and use the natural scalar extension map
\[ \iota_p= \iota_{G,p}: K_0(\ZZ[G],\RR[G]) \to K_0(\ZZ_{(p)}[G],\RR[G]).\]

Then, as a first step, we note that there are equalities
\begin{equation}\label{first equalities} \iota_p\bigl(\chi_G^{\rm sgn}(A)\bigr) = \partial_p\bigl( \langle \RR\cdot A^t(L),(-1)^{a_p}\rangle\bigr)\,\,\text{ and }\,\, \iota_p\bigl(\chi_G^{\rm BSD}(A,V_{L}))= \partial_p\bigl( \langle \RR[G]^t, \iota_{A,L}^{\rm NT}\rangle\bigr).\end{equation}
The first of these follows directly from (\ref{sign p equality}) and the second from a routine comparison of the  definition of the automorphism $\iota_{A,L}^{\rm NT}$ with the explicit computation of $\chi_G^{\rm BSD}(A,V_{L})$ in terms of the non-abelian determinant of the representative of ${\rm SC}_{V_L,(p)}$ fixed in Proposition \ref{cons prop}(ii). %cf. \S\ref{na dets section}). 

We next claim that 
\begin{equation}\label{next equality} \iota_p\bigl(\chi_G^{\rm coh}(A,V_{L})\bigr) = \partial_p\bigl( \langle\RR[G], p^{\chi (\mathcal{L})}\rangle\bigr).\end{equation}
To show this we recall from (the proof of) Lemma \ref{zar lemma} that the complex $C := R\Gamma(X,\mathcal{L}_L)^\ast$ belongs to $D^{\rm perf}(\mathbb{F}_p[G])$ and is acyclic outside degrees $0$ and $1$. Since $G$ is a $p$-group, $C$ is therefore isomorphic in $D(\mathbb{F}_p[G])$ to a complex of the form $\mathbb{F}_p[G]^{n_0} \to \mathbb{F}_p[G]^{n_1}$, where the first term is placed in degree $0$ and $n_0$ and $n_1$ are suitable natural numbers. By using this representative, one computes that 
\[ \chi_G^{\rm coh}(A,V_L) := \chi_{G}(R\Gamma(X, \mathcal{L}_L)^*,0) = \partial_p\bigl( \langle \RR[G], p^{n_0-n_1}\rangle\bigr). \]
%(n_0-n_1)(\ZZ_{(p)}[G],p,\ZZ_{(p)}[G]) =
%
To deduce (\ref{next equality}) it is now enough to note that a computation of Euler characteristics in $K_0(\mathbb{F}_p) \cong \ZZ$ implies that 
\begin{align*} \chi(\mathcal{L}) :=\chi_{\mathbb{F}_p} (R\Gamma(X, \mathcal{L})^*) =&\,\,\chi_{\mathbb{F}_p}\bigl(R\Hom_{\mathbb{F}_p[G]}(\mathbb{F}_p,R\Gamma(X, \mathcal{L}_L)^\ast)\bigr)\\
  =&\,\,\chi_{\mathbb{F}_p}\bigl( R\Hom_{\mathbb{F}_p[G]}(\mathbb{F}_p,\mathbb{F}_p[G]^{n_0} \to \mathbb{F}_p[G]^{n_1})\bigr)\\
  =&\,\,\chi_{\mathbb{F}_p}\bigl(\mathbb{F}_p^{n_0} \to \mathbb{F}_p^{n_1}\bigr)\\
  =&\,\, n_0-n_1,\end{align*}
where the first aligned equality follows from the isomorphism (\ref{second KT descent}). 

Thus, if one defines an element of $K_1(\RR[G])$ by setting  
\[ \mathscr{L} := L^*_{U}(A_{L/K}, 1)\times \langle\RR[G]^t, \iota_{A,L}^{\rm NT}\rangle^{-1}\times\langle \RR[G], p^{\chi(\mathcal{L})}\rangle\times  \langle \RR\cdot A^t(L),(-1)^{a_p}\rangle^{-1} \]
then the equalities (\ref{first equalities}) and (\ref{next equality}) imply that 
\[ \partial_p(\mathscr{L}) = \iota_p\bigl(\partial_{G}( L^*_{U}(A_{L/K}, 1)) - \chi_G^{\rm BSD}(A,V_{L}) + \chi_G^{\rm coh}(A,V_{L}) - \chi_G^{\rm sgn}(A)\bigr). \]
%=\,\,&\partial_p\bigl( L^*_{U}(A_{L/K}, 1)\times ((-1)^{a_p},\RR\cdot A^t(L)) \times (\RR[G]^t, \iota_{A,L}^{\rm NT})\bigr). \end{align*}
%
In addition, an explicit computation of reduced norm combines with the definition (\ref{normalise def}) of each term $\mathscr{L}(A,\psi)$ to imply that   
\[ {\rm Nrd}_{\RR[G]}(\mathscr{L}) = {\sum}_{\psi \in {\rm Ir}(G)}\mathscr{L}(A,\psi)e_\psi \in \zeta(\CC[G])^\times.\]

This equality implies the conditions stated in Proposition \ref{cons prop 2}(i) are equivalent to asserting ${\rm Nrd}_{\RR[G]}(\mathscr{L})$ belongs to $\zeta(\QQ[G])^\times$. Hence, since $K_1(\QQ[G])$ is the full pre-image under $\partial_p$  of the subgroup $K_0(\ZZ_{(p)}[G],\QQ[G])$ of $K_0(\ZZ_{(p)}[G],\RR[G])$,  
these conditions are true if the equality in Conjecture \ref{conj:ebsd}(iii) is valid modulo $K_0(\ZZ[G],\QQ[G])$. Their validity thus follows directly from the assumed validity of Conjecture \ref{conj:ebsd}(i) and (ii) and the argument of \S\ref{sect:reform}.  

In a similar way, if $G$ is abelian, then the above computation shows that Conjecture \ref{conj:ebsd}(iii) implies ${\rm Nrd}_{\RR[G]}(\mathscr{L})$ belongs to $\ZZ_{(p)}[G]^\times$. In addition, since $\ZZ_{(p)}[G]$ is a local ring (as $G$ is a $p$-group), the latter containment is valid if and only if  ${\rm Nrd}_{\RR[G]}(\mathscr{L})$ belongs to $\ZZ_{(p)}[G]$ and its image under the projection $\ZZ_{(p)}[G] \to \ZZ_{(p)}$ belongs to $\ZZ_{(p)}^\times$. These conditions are in turn equivalent to requiring $\mathscr{L}(A,{\bf 1}_G)$ belongs to $\ZZ_{(p)}^\times$ and, also, for all $g\in G$ one has 
\[ {\sum}_{\psi\in {\rm Ir}(G)} \psi(g^{-1})\mathscr{L}(A,\psi) \in |G|\cdot \ZZ_{(p)}.\]
To deduce the result of Proposition \ref{cons prop 2}(ii) it is thus enough to note that, for each abelian subquotient $Q = H/J$ of $G$,  the arguments of Propositions \ref{bsd-schneider} and \ref{group independent} combine to imply that the validity of Conjecture \ref{conj:ebsd}(iii) for the data $(A,L/K)$ implies the validity modulo $\ker(\iota_{Q,p})$ of the equality in Conjecture \ref{conj:ebsd}(iii) for the data $(A_{L^H},L^J/L^H)$.  

Finally, to prove Proposition \ref{cons prop 2}(iii) we assume the validity of Conjecture \ref{conj:ebsd} and hence that the element $\mathscr{L}$ belongs to the image of $K_1(\ZZ_{(p)}[G])$ in $K_1(\RR[G])$. Thus, if we fix an embedding of fields $j: \RR \to \CC_p$, then the image of $\mathscr{L}$ under the induced map $K_1(\RR[G]) \to K_1(\CC_p[G])$ belongs to the image of the natural map $K_1(\ZZ_p[G])\to K_1(\CC_p[G])$. 

Given this containment, the equalities in claim (iii) follow from the general result of \cite[Th. 2.1]{bk} (with $\Lambda = \ZZ_p$) and the fact that, for each subgroup $H$ of $G$, the argument of Proposition \ref{group independent} implies ${\sum}_{\psi \in {\rm Ir}(H^{\rm ab})}\mathscr{L}(A_{L^H},\psi)e_\psi$ is equal to the image of ${\rm Nrd}_{\RR[G]}(\mathscr{L})$ under the upper composite map in the diagram
\[ \begin{CD} \zeta(\RR[G])^\times @> \subset >> \zeta(\CC[G])^\times @> \varrho_H >> \zeta(\CC[H])^\times @> \varrho'_H >> \zeta(\CC[H^{\rm ab}])^\times = \CC[H^{\rm ab}]^\times\\
& & @A {\rm Nrd}_{\CC[G]} A \cong A @A\cong A {\rm Nrd}_{\CC[H]} A\\
& & K_1(\CC[G]) @> \theta_{G,H}^1 >> K_1(\CC[H])\end{CD}\]
Here $\theta_{G,H}^1$ is the natural restriction of scalars map, $\varrho_H$ is defined by the requirement that the square commutes and $\varrho_H'$ is the natural projection map.

This completes the proof of Proposition \ref{cons prop 2}. 

\appendix

\section{Kummer-\'etale descent for coherent cohomology}
In this first appendix, we show that the coherent cohomology over a `separated'  formal fs log scheme can be computed via the \v{C}ech resolution with respect to an affine \emph{Kummer-\'etale} covering (not necessarily a Zariski open covering). Whilst this result seems to be well known to experts, we have not been able to locate a good reference for it in the literature.

\subsection{Fs log schemes and their fibre products}
The main purpose of this section is to review the construction of fibre products for fs log schemes, which we need for the sheaf theory on Kummer-\'etale sites and the construction of {\Cech} complexes. We will briefly recall some definitions of monoids and log schemes needed for the construction of fibre products. We do not give a complete review of basic definitions on monoids and log geometry but rather refer readers to  \cite{Kato:Fontaine-Illusie} and \cite{Niziol:LogK-1} for basic definitions in log geometry and to \cite{Ogus:LogGeom} for a more comprehensive reference.

Recall that a (always commutative) monoid $P$ is said to be \emph{fine} if it is finitely generated and the natural map $P\to P^{\rm gp}$ is injective (where $P^{\rm gp}$ is the commutative group obtained by adjoining the inverse of each element of $P$). A fine monoid $P$ is said to be \emph{saturated} if for any $\alpha\in P^{\rm gp}$, we have $\alpha^n\in P$ for some $n>0$ if and only if $\alpha\in P$. By a \emph{fs} monoid, we mean a fine and saturated monoid.

For each monoid $P$ we define a saturation $P^{\rm sat}:=\{\alpha\in P^{\rm gp};\ \alpha^n\in P \text{ for some }n>0\}$.

\begin{lemma}\label{lem:Gordon}
If $P$ is finitely generated, then the monoid $P^{\rm sat}$ is fs.
\end{lemma}
\begin{proof}
It suffices to show that $P^{\rm sat}$ is finitely generated, which is  a direct consequence of Gordon's Lemma (cf. \cite[Ch.~I, Th.~2.3.19]{Ogus:LogGeom}). \end{proof}

A log scheme $X^\sharp$ is called \emph{fs} (i.e., fine and saturated) if \'etale locally on the underlying scheme $X$, the log structure is generated by a map of monoids $P\to \mathcal{O}_X$ where $P$ is a fs monoid.
Our log schemes and formal schemes are always assumed to be fs (i.e., fine saturated).

Let $X^\sharp$ and $Y^\sharp$ be fs log schemes over $S^\sharp$ (with underlying schemes denoted as $X$, $Y$, and $S$). We want to construct a fs log scheme $X^\sharp\times_{S^\sharp} Y^\sharp$ satisfying the universal property of fibre product (cf. \cite[Ch.~III, Cor.~2.1.6]{Ogus:LogGeom}).

By replacing the formal log schemes with suitable \'etale coverings, we choose charts $P\to \mathcal{O}_X$, $Q\to \mathcal{O}_Y$ and $M\to \mathcal{O}_S$ defining the log structures (where $P$, $Q$ and $M$ are fs monoids, viewed as constant sheaves), such that there exist maps $M\to P, Q$ giving rise to the structure morphism $X^\sharp, Y^\sharp \to S^\sharp$. (The existence of such local charts follows from \cite[Lem.~2.10]{Kato:Fontaine-Illusie}.) 

The most natural candidate is to endow $X\times_S Y$ with the log structure associated to the chart $P\oplus_MQ\to \mathcal{O}_X\otimes_{\mathcal{O}_S}\mathcal{O}_Y$, where $P\oplus_MQ$ is the amalgamated sum of monoids. But this may not always work as $P\oplus_MQ$ may not be fine nor saturated.

Writing $P\oplus_M^{\rm sat} Q $ for the \emph{saturation} of $P\oplus_M Q $ we can define the following fs log scheme
\[
X^\sharp\times_{S^\sharp}Y^\sharp := (X\times_S Y)\times_{\Spec\ZZ[P\oplus_M Q]}\Spec\ZZ[P\oplus_M^{\rm sat}Q]
\]
with the log structure given by the chart $P\oplus_M^{\rm sat}Q \to \mathcal{O}_{X^\sharp\times_{S^\sharp}Y^\sharp}$ naturally extending $P\oplus_M Q \to \mathcal{O}_{X\times_S Y}$. By glueing this \'etale-local construction, we obtain the fibre products for any fs log schemes. We repeat this construction to obtain fibre products of formal fs log schemes.

Note that this notion of fibre product may not be compatible with fibre products of (formal) schemes without log structure, as we can see from the explicit \'etale-local construction. Instead, we have the following lemma, which is a consequence of Lemma~\ref{lem:Gordon}. (See \cite[Ch.~III, Cor.~2.1.6]{Ogus:LogGeom} for the proof.)
\begin{lemma}\label{lem:log-closed}
The underlying scheme for $X^\sharp\times_{S^\sharp}Y^\sharp $ is finite over  $X\times_S Y$. The same holds for formal fs log schemes.
\end{lemma}
%
%\begin{proof}
%
%The claim is \'etale-local on $X\times_SY$, so we may replace $X$, $Y$ and $S$ by some \'etale neighbourhood so that there exists suitable charts as above. Then it suffices to show that $\ZZ[P\oplus^{\rm sat}_MQ]$ is a finite algebra over $\ZZ[P\oplus_MQ]$, which follows from Gordon's lemma (cf.  Lemma~\ref{lem:Gordon}).
%\end{proof}
\begin{remark}{\em To give a concrete example in which the underlying scheme for $X^\sharp\times_{S^\sharp}Y^\sharp$ differs from $X\times_SY$ we fix a finite Galois Kummer-\'etale cover $\pi: X_L^\sharp \xrightarrow{\pi^\sharp}X^\sharp$ of group $G$. In this case one has $X_L^\sharp\times_{X^\sharp}X_L^\sharp \cong G\times X_L^\sharp$ whereas $X_L\times_X X_L \cong G\times X_L$ only if $\pi$ is unramified.} 
\end{remark}

The following corollary of  lemma~\ref{lem:log-closed} will be used later.

\begin{corollary}\label{cor:affineness}
Let $X^\sharp$ be a fs log scheme, such that the underlying scheme is separated. Let $U^\sharp$ and $U^{\prime\sharp}$ be fs log  schemes over $X^\sharp$, such that the underlying schemes $U$ and $U'$ are affine. Then $U^\sharp\times_{X^\sharp}U^{\prime\sharp}$ is also affine. The same holds for formal fs log schemes.
\end{corollary}

\begin{proof}

Under the hypotheses, the scheme $U\times_X U'$ is affine, which follows from the cartesian diagram below:
\[
\xymatrix{
U\times_{X}U' \ar@{^{(}->}[r] \ar[d] &
U\times U' \ar[d]\\
X\ar@{^{(}->}[r] ^-{\Delta_{X}}&
X\times X,
}\]
Now by lemma~\ref{lem:log-closed}, the underlying scheme of $U^\sharp\times_{X^\sharp}U^{\prime\sharp}$ is finite over an affine scheme $U\times_X U'$. This proves the corollary.
\end{proof}

\subsection{{\Cech}-to-derived functor spectral sequence for Kummer-\'etale cohomology}

For a log formal scheme $\gotx$, we write $\gotx^\sharp_\ket$ for the associated Kummer-\'etale site (as per \cite[Def.~2.13]{Niziol:LogK-1}).

We quickly recall the definition of \v{C}ech complex and \v{C}ech-to-derived functor spectral sequences in this setting.

\begin{definition}{\em
Let $\gotu^\sharp$ be an Kummer-\'etale covering of $\gotx^\sharp$ (i.e., the structure morphism $\gotu^\sharp\to\gotx^\sharp$ is Kummer-\'etale and surjective), and let $\calf$ be a sheaf of abelian groups on the Kummer-\'etale site $\gotx^\sharp_\ket$. Then we can form a \v{C}ech complex
\[
C^\bullet(\gotu^\sharp,\calf):= [\Gamma(\gotu^\sharp,\calf) \to \Gamma(\gotu^\sharp\times_{\gotx^\sharp}\gotu^\sharp,\calf) \to \Gamma(\gotu^\sharp\times_{\gotx^\sharp}\gotu^\sharp\times_{\gotx^\sharp}\gotu^\sharp,\calf)\to\cdots],
\]
with differentials defined in a standard way.

(The usual definition of Cech complexes for the case without log structure, \emph{cf.} \cite[Ch.~III, \S2]{milne:etale}, formally goes through.)  For any bounded-below complexes $\calf^\bullet$, we define the \v{C}ech complex $C^\bullet(\gotu^\sharp,\calf^\bullet)$ as the total complex of the double complex obtained from \v{C}ech complex of each term of $\calf^\bullet$.}
\end{definition}

Whilst the \v{C}ech complex $C^\bullet(\gotu^\sharp,\calf)$ does not necessarily represent $R\Gamma(\gotx^\sharp_\ket,\calf)$, there exists a natural `\Cech-to-derived functor spectral sequence'
\begin{equation}
E_1^{i,j}: H^j (\gotu^{\sharp}_{i,\ket}, \calf) \Rightarrow H^{i+j} (\gotx^\sharp_{\ket},\calf),
\end{equation}
where $\gotu^{\sharp}_i$ is the $(i+1)$-fold self fibre product of $\gotu^\sharp$ over $\gotx^\sharp$.
One way to read off this spectral sequence from the literature is via the technique of cohomological descent for (simplicial) topoi associated to the Kummer-\'etale sites $\gotu^\sharp_{\ket}$ and $\gotx^\sharp_{\ket}$ (cf. SGA~4$_\text{II}$, Exp.~Vbis. \cite{SGA4-2}).
Indeed, since it admits a local section, $\gotu^\sharp\to\gotx^\sharp$ is a `morphism of universal cohomological descent' by [loc. cit., Prop.~(3.3.1)] and so the above spectral sequence is just a special case of the descent spectral sequence from [loc. cit., Prop.~(2.5.5)]).

\begin{remark}\label{rmk:SpecSeq}{\em  The complex $(E_1^{i,0}, d^{i,0})$ coincides with $C^\bullet(\gotu^\sharp,\calf)$ and so the above spectral sequence implies $C^\bullet(\gotu^\sharp,\calf) = R\Gamma(\gotx^\sharp_\ket,\calf)$ if $E_1^{i,j}$ vanishes for all $j>0$ and $i\geqslant 0$.}
\end{remark}

\subsection{Coherent cohomology}  We first recall Kummer-\'etale descent theory for coherent sheaves on schemes and formal schemes.

Let $\gotx^\sharp$ be a log scheme over $\Z/p^n$ for some $n$ and $\calf$ a quasi-coherent $\mathcal{O}_\gotx$-module. Then, by Kato's unpublished result (cf. \cite[Prop.~2.19]{Niziol:LogK-1}) the presheaf $\gotu^\sharp\in\gotx^\sharp_\ket \rightsquigarrow \Gamma(\gotu, \calf_{\gotu})$ is a sheaf on $\gotx^\sharp_\ket $, where $\calf_{\gotu}$ denotes the pull back of  $\calf$ via the structure morphism $\gotu\to\gotx$ of the underlying schemes. We use the same notation $\calf$ to denote the Kummer-\'etale sheaf associated to a quasi-coherent sheaf $\calf$.

Now, if $\gotx^\sharp$ be a locally noetherian formal fs log scheme  over $\Spf\Zp$, we can associate, to a coherent $\mathcal{O}_\gotx$-module $\calf$,  a Kummer-\'etale $\Zp$-sheaf $\calf$ by extending the construction for coherent sheaves on schemes via projective limit.
(We restrict to coherent sheaves to avoid technicalities regarding completion.)

Now, we are interested in $C^\bullet(\gotu^\sharp,\calf)$ when $\calf$ is  a vector bundle on $\gotx$ (viewed as a Kummer-\'etale sheaf), while $\gotu^\sharp$ remains a Kummer-\'etale covering of $\gotx^\sharp$.

\begin{proposition}\label{prop:Cech}
Let $\gotx^\sharp$ be a noetherian formal fs log scheme over $\Spf R$ (for some noetherian adic ring $R$, with trivial log structure), and assume that $\gotx$ is separated.
Then for any coherent $\mathcal{O}_\gotx$-module $\calf$ there is a natural isomorphism $R\Gamma(\gotx^\sharp_\ket,\calf) \riso R\Gamma(\gotx,\calf)$.

Furthermore, for any Kummer-\'etale covering $\gotu^\sharp\to \gotx^\sharp$ where $\gotu$ is affine, the \Cech{} complex $C^\bullet(\gotu^\sharp,\calf)$ represents $R\Gamma(\gotx,\calf)$.

The same holds if we replace $\calf$ with a bounded-below complex $\calf^\bullet$ of coherent sheaves of $\mathcal{O}_\gotx$-modules, such that the differential maps $d^i:\calf^i\to\calf^{i+1}$ are additive morphisms of Kummer-\'etale sheaves. 
\end{proposition}

\begin{proof} By standard argument with hypercohomology spectral sequences, the claim for $\calf^\bullet$ can be reduced to $\calf$.

Let us first assume that $\gotx$ is affine. Then by \cite[Prop.~3.27]{Niziol:LogK-1} (and the theorem on formal functions), we have $R\Gamma(\gotx^\sharp_\ket,\mathcal{O}_\gotx) = \Gamma(\gotx,\mathcal{O}_\gotx)$. Now, by resolving $\calf$ with free $\mathcal{O}_\gotx$-modules, we obtain $R\Gamma(\gotx^\sharp_\ket,\calf) = \Gamma(\gotx,\calf).$

Choose a Kummer-\'etale covering $\gotu^\sharp\to\gotx^\sharp$ with $\gotu$ affine. Then Corollary~\ref{cor:affineness} implies
\[\gotu^\sharp_i:=\underbrace{\gotu^\sharp\times_{\gotx^\sharp}\cdots\times_{\gotx^\sharp}\gotu^\sharp}_{i+1\text{ times}}\]
 has an \emph{affine} underlying formal scheme. Therefore, by the \Cech-to-derived spectral sequence argument it follows that $C^\bullet(\gotu^\sharp,\calf)$ represents $R\Gamma(\gotx^\sharp_\ket,\calf)$ (cf. Remark~\ref{rmk:SpecSeq}). Now if we choose $\gotu^\sharp$ to be the disjoint union of finite affine open covering of $\gotx$ (with the natural log structure induced from $\gotx^\sharp$), then $C^\bullet(\gotu^\sharp,\calf)$ represents $R\Gamma(\gotx,\calf)$, as claimed.
\end{proof}

\begin{remark}{\em We apply Proposition \ref{prop:Cech} to the log de-Rham complex $\calf^\bullet$, where the maps $d^i:\calf^i\to\calf^{i+1}$ are not $\mathcal{O}_\gotx$-linear but are additive morphisms of Kummer-\'etale sheaves.}\end{remark}

\section{Lefschetz trace formula for rigid cohomology}

In this second appendix, our aim is to establish a slight extension of the Lefschetz trace formula for rigid cohomology that is proved in \cite[Th. 6.3]{ELS}. 

As before, we let $U$ be a \emph{smooth affine curve} over a finite field of characteristic $p>0$. Let $\Lambda_0$ be a finite extension of $\QQ_p$, and assume that $\mathcal{O}_U$ contains the residue field $k_0$ of $\Lambda_0$. Set $q_0=p^{r_0}:=\#(k_0)$

Let $\mathcal{O}^{\dagger} _{U/\Lambda_0}$ denote the global section of the overconvergent structure with coefficients $\Lambda_0$. To explain, let $X$ be a smooth compactification of $U$, and let $\mathfrak{X}_{\mathcal{O}_{\Lambda_0}}$ be a formal lift of $X$ over the valuation ring $\mathcal{O}_{\Lambda_0}$ of $\Lambda_0$. Let $\mathcal{X}_{\Lambda_0}$ denote its rigid generic fibre, which contains the tube $]U[_{\Lambda_0}$ of $U$ as an open subspace. 
Then $\mathcal{O}^\dagger_{U/\Lambda_0}$ is the ring of rigid analytic functions defined on some `strict neighbourhood' of $]U[_{\Lambda_0}$. (We often call such rigid analytic functions \emph{overconvergent} along $X\setminus U$.) Note that if $\Lambda'_0$ is a finite extension of $\Lambda_0$ whose residue field is contained in $\mathcal{O}_U$, then we have $\mathcal{O}^{\dagger}_{U/\Lambda_0'} \cong \mathcal{O}^{\dagger}_{U/\Lambda_0}\otimes_{\Lambda_0}\Lambda_0'$.

Since $\mathcal{O}_U$ contains the residue field $k_0$ of $\Lambda_0$, one can define an overconvergent $\Lambda_0$-$F$-isocrystal $\mathcal{E}$ over $U$; see \S\ref{ssec:frob}.  (See also \cite[\S7]{tsuzuki98} for the definition and the natural context where overconvergent $\Lambda_0$-$F$-isocrystals appear. In loc.~cit. it is called an \emph{overconvergent $\Lambda_0$-$F^{r_0}$-isocrystal} where $r_0:=[k_0:\FF_p]$, but let us suppress $r_0$ from the notation.) In our intended application however, we would naturally like to remove the assumption that the residue field of the coefficient field $\Lambda_0$ is not contained in $\mathcal{O}_U$. (\emph{Cf.} the proof of Theorem~\ref{order-leading}.)

Let $\Lambda$ be finite extension of $\QQ_p$. Unless the residue field of $\Lambda$ is contained in $\mathcal{O}_U$, the usual definition of `overconvergent $\Lambda$-$F$-isocrystals' over $U$ does not apply. Instead, let us consider a subextension $\Lambda_0\subset \Lambda$ whose residue field $k_0$ is contained in $\mathcal{O}_U$. (For example, we may choose $k_0$ to be the maximal subfield of the residue field $k_\Lambda$ of $\Lambda$ that can be embedded in $\mathcal{O}_U$.) Instead of defining `overconvergent $\Lambda$-$F$-isocrystals' over $U$, we will work with overconvergent $\Lambda_0$-$F$-isocrystals `equipped with $\Lambda$-action'; \emph{cf.} \S\ref{ssec:frob}. 
%Let $\Lambda$ be a finite extension of $\Lambda_0$. Assume that $\Lambda$ acts on our overconvergent $\Lambda_0$-$F$-isocrystal $\mathcal{E}$; i.e., $\Lambda$ acts on $\mathcal{E}$ compatibly with the $\Lambda_0$-action, and $\nabla_{\mathcal{E}}$ and $\varphi_{\mathcal{E}}$ are both $\Lambda$-linear. If $\mathcal{O}_U$ does not contain the residue field $k_\Lambda$ of $\Lambda$, then we cannot define $\mathcal{E}$ as an overconvergent $\Lambda$-$F$-isocrystal without extending the base field of $U$. 
The aim of this appendix is to extend the Lefschetz trace formula \cite[6.3]{ELS} for the rigid cohomology with coefficients in such overconvergent $\Lambda_0$-$F$-isocrystals with $\Lambda$-action when $\mathcal{O}_U$ does not contain the residue field of $\Lambda$.

\subsection{Overconvergent modules and duality}
Let $U$ and $\Lambda_0$ be as before and let $\Lambda$ be a finite extension of $\Lambda_0$. We set $U':=U\times_{\Spec k_0}\Spec k_\Lambda$ where $k_\Lambda$ is the residue field of $\Lambda$. %, and assume that $U'$ is also connected. (In the intended setting, this can be arranged by replacing $\Lambda_0$ by some finite unramified extension.) 

Set $X':=X\times_{\Spec k_0}\Spec k_\Lambda$, which is a smooth compactification  of $U'$. We also choose a formal $\mathcal{O}_{\Lambda}$-lift $\mathfrak{X}'_{\Lambda}$ of $X'$, and we subsequently obtain its rigid generic fibre $\mathcal{X}'_{\Lambda}$ and the tube $]U'[_{\Lambda}\subset \mathcal{X}'_{\Lambda}$. With this setting, we define $\mathcal{O}^\dagger_{U'/\Lambda}$ to be the ring of rigid analytic functions defined on some strict neighbourhood of $]U'[_{\Lambda}$. Then we have an isomorphism of $\Lambda$-algebras
\begin{equation}\label{eq:oc-ring}
\mathcal{O}^\dagger_{U'/\Lambda} \cong \mathcal{O}^\dagger_{U/\Lambda_0}\otimes_{\Lambda_0}\Lambda.
\end{equation}
By the very construction, $\mathcal{O}^\dagger_{U/\Lambda_0}$ and $\mathcal{O}^\dagger_{U'/\Lambda}$ are \emph{Fr\'echet algebras} over $\Lambda_0$ and $\Lambda$, respectively, so any finite locally free modules over them are $p$-adic Fr\'echet spaces.

Let $\Omega^{\dagger}_{U/\Lambda_0}$ and $\Omega^{\dagger}_{U'/\Lambda}$ denote the modules of overconvergent K\"ahler differentials. Then we also have
\begin{equation}\label{eq:oc-diff}
\Omega^{\dagger}_{U'/\Lambda} := \Omega^{\dagger}_{U/\Lambda_0} \otimes_{\Lambda_0}\Lambda.
\end{equation}

%For convenience, let us introduce the following terminology. By \emph{overconvergent vector bundle} $\mathcal{E}$, we mean a vector bundle defined on some strict neighbourhood of $]U'[_{\Lambda}\subset \mathcal{X}'_\Lambda$ (without specifying which strict neighbourhood it is defined over). 
Let $]U'[_{\Lambda}\subset \mathcal{V}'\subset \mathcal{X}'_\Lambda$ be a strict neighbourhood, and let $\mathcal{E}_{\mathcal{V}'}$ be a vector bundle  defined over $\mathcal{V}'$. We set 
\begin{equation}\label{eq:global-sections}
\mathcal{E}:=\varinjlim_{\mathcal{W'}} \Gamma(\mathcal{W}', \mathcal{E}_{\mathcal{V}'}), 
\end{equation}
where the direct limit is taken over all strict neighbourhoods $\mathcal{W}' $ contained in $ \mathcal{V}'$. Then $\mathcal E$ turns out to be a locally free $\mathcal{O}^{\dagger}_{U'/\Lambda}$-module; hence, a Fr\'echet $\Lambda$-space. (The local freeness claim can easily reduced to the case when $U$ is an open subscheme of $\mathbb{P}^1$, which is standard; \emph{cf.} \cite[\S{V}, Th\'eor\`eme~1]{gruson}.) Note that $\Omega^{\dagger}_{U'/\Lambda}$ can also be obtained from the line bundle $\mathcal{E}_{\mathcal{V'}}=\Omega_{\mathcal{X}'_{\Lambda}}$ of K\"ahler differentials over $\mathcal{V}'=\mathcal{X}'_\Lambda$. %(If one uses the formalism of \emph{dagger spaces} introduced by Gro\ss{}e-Kl\"onne, then we are identifying a vector bundle on some affinoid dagger space with its global sections.)

To the `overconvergent vector bundle' $\mathcal{E}_{\mathcal{V}'}$, we define another $\mathcal{O}^{\dagger}_{U'/\Lambda}$-module as follows:
\begin{definition}
For a sufficiently small strict neighbourhood $\mathcal{V}'$ of $]U'[_\Lambda$, we define \[
\mathcal{E}_c:= H^1_{]U'[_{\Lambda}}(\mathcal {V}', \mathcal{E}_{\mathcal{V}'}),
\]
where $H^1_{]U'[_{\Lambda}}(\mathcal {V}', \mathcal{E}_{\mathcal{V}'})$ is the first cohomology of the mapping fibre of 
\[
R\Gamma(\mathcal{V}',\mathcal{E}_{\mathcal{V}'})\to R\Gamma(\mathcal{V}'\cap ]X'\setminus U'[_{\Lambda},\mathcal{E}_{\mathcal{V}'}).
\]
Note that $\mathcal{E}_c$ does not depend on the choice of $\mathcal{V}'$; \emph{cf.} \cite[\S3.2]{ELS}.

If we set $\mathcal{E}_{\mathcal{V}'}:=\mathcal{O}_{\mathcal{V}'}$ (respectively, $\mathcal{E}_{\mathcal{V}'}:=\Omega_{\mathcal{V}'}$), then the corresponding $\mathcal{E}_c$ is denoted as $(\mathcal{O}^{\dagger}_{U'/\Lambda})_c$ (respectively, $(\Omega^{\dagger}_{U'/\Lambda})_c$). 
\end{definition}

By shrinking $\mathcal{V}'$ if necessary, we may assume that $\mathcal{V}'$ is affinoid. In that case $\mathcal{V}'\cap ]X'\setminus U'[_{\Lambda}$ is quasi-Stein, so we can deduce the following
\begin{itemize}
	\item $\mathcal{E}_c$ is a Fr\'echet $\Lambda$-space.
	\item From the same argument as in \cite[\S3.2]{ELS} we can deduce that 
\begin{equation}\label{eq:global-section-with-support}
\mathcal{E}_c \cong \mathcal{E} \otimes_{\mathcal{O}^{\dagger}_{U'/\Lambda}}(\mathcal{O}^{\dagger}_{U'/\Lambda})_c.
\end{equation}
In particular, $\mathcal{E}_c$ depends only on $\mathcal{E}$, not on the choice of strict neighbourhood $\mathcal{V}'$. 

\end{itemize}

\begin{lemma}\label{lem:duality}\ 

\begin{itemize}
\item[(i)] There is a canonical trace map 
\[\tr: (\Omega^{\dagger}_{U'/\Lambda})_c \to \Lambda ,\]
which factors through an isomorphism $H^2_{{\rm rig},c}(U'/\Lambda)\xrightarrow\sim \Lambda$.
\item[(ii)]%\label{lem:duality:pairing} 
 Let $\mathcal{E}_{\mathcal{V}'}$ be a vector bundle  on some strict neighbourhood $\mathcal{V}'$ of $]U'[_{\Lambda}$, and consider $\mathcal{E}$ as above. Then we have the following natural $\Lambda$-bilinear perfect pairing
\begin{align*}
\langle-,-\rangle^0:&\mathcal{E}^{\vee} \times (\mathcal{E}\otimes_{\mathcal{O}^\dagger_{U'/\Lambda}}(\Omega^{\dagger}_{U'/\Lambda})_c)\longrightarrow \Lambda \quad \langle u, m\otimes\omega_c\rangle^0:=\tr(u(m)\cdot \omega_c);\\
\langle-,-\rangle^1:&(\mathcal{E}^{\vee} \otimes_{\mathcal{O}^\dagger_{U'/\Lambda}}\Omega^{\dagger}_{U'/\Lambda})\times \mathcal{E}_c\longrightarrow \Lambda \quad \langle u\otimes\omega,m\otimes f_c\rangle^1:=\tr(u(m)\cdot (\omega\otimes f_c))
\end{align*}
where $u\in\mathcal{E}^{\vee}$, $m\in\mathcal{E}$, $f_c\in(\mathcal{O}^{\dagger}_{U'/\Lambda})_c$, $\omega_c\in(\Omega^{\dagger}_{U'/\Lambda})_c$ and $\omega\in\Omega^{\dagger}_{U'/\Lambda}$.
%\item[(iii)] Suppose that we furthermore have an integrable connection $\nabla\colon\mathcal{E} \to \mathcal{E}\otimes_{\mathcal{O}^\dagger_{U'/\Lambda}}\Omega^{\dagger}_{U'/\Lambda}$
\end{itemize}
\end{lemma}
\begin{proof}The first assertion is standard (\emph{cf.} the proof of \cite[Lem.~3.4(i)]{ELS}).
If $X' \cong \mathbb{P}^1$ and $U'\cong \mathbb{A}^1$, then the second and third assertions are proved in \cite[\S3.3]{ELS}. In general, one can find a finite covering $f:X'\to\mathbb{P}^1_{k_\Lambda}$ so that the preimage of $\mathbb{A}^1_{k_{\Lambda}}$ is $U'$. Then although $f$ may not admit a formal lift, one can find a strict neighbourhood $\mathcal{W}'$ of $]U'[_{\Lambda}$ contained in $\mathcal{V}'$, such that its image in $\mathbb{P}^1_{\Lambda}$ is a strict neighbourhood of $]\mathbb{A}^1[_{\Lambda}$; \emph{cf.} the proof of  Proposition~5.2.21 in \cite[p.~151]{LS07}. Since the claim can be checked after push-forward by finite covering, the claim is reduced  to the case when $U'=\mathbb{A}^1$, which is already handled.
\end{proof}

\subsection{Overconvergent $F$-isocrystals}\label{ssec:frob}
Let us choose a $\Lambda_0$-linear $q_0$-Frobenius operator 
\begin{equation}\label{eq:frob}
\varphi_{\Lambda_0}:\mathcal{O}^{\dagger}_{U/\Lambda_0}\to \mathcal{O}^{\dagger}_{U/\Lambda_0},
\end{equation}
which is possible by the approximation theorem; \emph{cf.} \cite[Th.~2.4.4]{vdput}.

Let us recall the following standard definition.
\begin{definition}\label{def:oc-Lambda}
An \emph{overconvergent $\Lambda_0$-$F$-isocrystal} over $U$ is a tuple $(\mathcal{E},\varphi_{\mathcal{E}},\nabla_{\mathcal{E}})$, where
\begin{itemize}
\item $\mathcal{E}$ is a finite locally free $\mathcal{O}^{\dagger}_{U/\Lambda_0}$-module. Note that such $\mathcal{E}$ necessarily comes from a vector bundle on some strict neighbourhood $\mathcal{V}_0$ of $]U[_{\Lambda_0}$ via \eqref{eq:global-sections}.
\item $\nabla_{\mathcal{E}}:\mathcal{E}\to\mathcal{E}\otimes\Omega^{\dagger}_{U/\Lambda_0}$ is  a continuous integrable connection on $\mathcal{E}$.
\item $\varphi_{\mathcal{E}}:\mathcal{E}\to \mathcal{E}$ is  a $\varphi_{\Lambda_0}$-semilinear horizontal endomorphism of $\mathcal{E}$.
\end{itemize}
If $\nabla_{\mathcal{E}}$ and $\varphi_{\mathcal{E}}$ are understood, then we simply use $\mathcal{E}$ to denote an overconvergent $\Lambda_0$-$F$-isocrystal.
\end{definition}

For a finite extension $\Lambda$ of $\Lambda_0$, we define a \emph{$\Lambda$-action} on  overconvergent $\Lambda_0$-$F$-isocrystal $\mathcal{E}$ to be a $\Lambda$-action on the underlying module $\mathcal{E}$ that is compatible with the $\Lambda_0$-action and commutes with $\nabla_{\mathcal{E}}$ and $\varphi_{\mathcal{E}}$. To be more explicit, an \emph{overconvergent $\Lambda_0$-$F$-isocrystal over $U$ with $\Lambda$-action} consists of the following data: 
\begin{itemize}
\item $\mathcal{E}$ is  a finite locally free module over $\mathcal{O}^\dagger_{U/\Lambda_0}\otimes_{\Lambda_0}\Lambda$.
\item $\nabla_{\mathcal{E}}:\mathcal{E}\to\mathcal{E}\otimes\Omega^{\dagger}_{U/\Lambda_0}$ is a \emph{$\Lambda$-linear} continuous integrable connection on $\mathcal{E}$.
\item $\varphi_{\mathcal{E}}:\mathcal{E}\to \mathcal{E}$ is  a \emph{$\Lambda$-linear} $\varphi_{\Lambda_0}$-semilinear horizontal endomorphism of $\mathcal{E}$.
\end{itemize}

Since we have $\mathcal{O}^\dagger_{U/\Lambda_0}\otimes_{\Lambda_0}\Lambda\cong \mathcal{O}^\dagger_{U'/\Lambda} $ \eqref{eq:oc-ring}, we may view $\mathcal{E}$ as coming from a vector bundle on some strict neighbourhood $\mathcal{V}'$ of $]U'[_{\Lambda}$. Also from \eqref{eq:oc-diff}, the connection $\nabla_{\mathcal{E}}$ is defined over some strict neighbourhood. On the other hand, $\varphi_{\mathcal{E}}$ can be described more naturally if we view $\mathcal{E}$ as a module over $\mathcal{O}^\dagger_{U/\Lambda_0}\otimes_{\Lambda_0}\Lambda$ (instead of $\mathcal{O}^\dagger_{U'/\Lambda} $). Indeed, $\varphi_{\mathcal{E}}$ is semilinear over 
\[
\varphi_{\Lambda_0}\otimes\Lambda: \mathcal{O}^\dagger_{U/\Lambda_0}\otimes_{\Lambda_0}\Lambda \to \mathcal{O}^\dagger_{U/\Lambda_0}\otimes_{\Lambda_0}\Lambda,
\]
the $\Lambda$-linear extension of $\varphi_{\Lambda_0}$, which cannot be naturally defined for  $\mathcal{O}^\dagger_{U'/\Lambda} $ without going through the isomorphism  \eqref{eq:oc-ring}.

\begin{lemma}\label{lem:duality-frob}\

\begin{itemize}
	\item[(i)] Let $\varphi_{\Lambda_0}$ also denote the endomorphism on $\Omega^{\dagger}_{U/\Lambda_0}$ induced by $\varphi_{\Lambda_0}$. Then we have
\[
\tr (\varphi_{\Lambda_0}(\omega)) = q_0\cdot\tr(\omega) \quad \forall \omega\in\Omega^{\dagger}_{U/\Lambda_0},
\]
where $q_0=\#(k_0)$ and $\tr$ is defined in Lemma~\ref{lem:duality}.
	\item[(ii)] %\label{lem:duality-frob:Fcrys}
	Let $\mathcal{E}$ be an overconvergent $\Lambda_0$-$F$-isocrystal over $U$ with $\Lambda$-action, and let $\langle-,-\rangle^0$ denote the duality pairing in Lemma~\ref{lem:duality}(ii). Then for any $u\in\mathcal E^\vee$, $m\in\mathcal E$ and $\omega_c\in(\Omega^\dagger_{U/\Lambda_0})_c$, we have
\[\langle \varphi_{\mathcal{E}}^\vee(u), \varphi_{\mathcal E}(m)\otimes \varphi_{\Lambda_0}(\omega_c)\rangle^0 = q_0\cdot \langle u,m\otimes\omega_c\rangle^0.\]
\end{itemize}
\end{lemma}
\begin{proof}
As in the proof of Lemma~\ref{lem:duality}, one can reduce the proof of this lemma to the case when $U\cong\mathbb{A}^1$. In that case, the first claim is proved in \cite[Prop.~4.2]{ELS}, and the second claim is proved in \cite[\S4.3]{ELS}.
\end{proof}

\begin{definition}\label{def:psi}
Let $\mathcal{E}$ be an overconvergent $\Lambda_0$-$F$-isocrystal over $U$ with $\Lambda$-action. Then we define the \emph{Dwork operators} $\psi_{\mathcal{E}^\vee}^i$ for $i=0,1$ to be $\Lambda$-linear endomorphisms
\begin{align*}
	\psi^0_{\mathcal{E}^\vee}&:\mathcal{E}^\vee\to\mathcal{E}^\vee\\
	\psi^1_{\mathcal{E}^\vee}&:\mathcal{E}^\vee \otimes_{\mathcal{O}^\dagger_{U'/\Lambda}}\Omega^{\dagger}_{U'/\Lambda}\to\mathcal{E}^\vee \otimes_{\mathcal{O}^\dagger_{U'/\Lambda}}\Omega^{\dagger}_{U'/\Lambda},
\end{align*}
which are respectively the adjoints of $\varphi_{\mathcal{E}}\otimes\varphi_{\Lambda_0}$ and $\varphi_{\mathcal{E}_c}$ with respect to the duality pairings $\langle-,-\rangle^i$ for $i=0,1$; \emph{cf.} Lemma~\ref{lem:duality-frob}(ii).
\end{definition}
Clearly, the Dwork operators $\psi^i_{\mathcal{E}^\vee}$ are $\varphi_{\Lambda_0}$-antilinear; i.e., for any $f\in\mathcal{O}^\dagger_{U/\Lambda_0}$ and $u\in\mathcal{E}^\vee$, we have $\psi^0_{\mathcal{E}^\vee}(\varphi_{\Lambda_0}(f)\cdot u) = f\cdot\psi^0_{\mathcal{E}^\vee}(u)$, and similarly for $\psi^1_{\mathcal{E}}$. 
Furthermore, by Lemma~\ref{lem:duality-frob} it follows that
\begin{equation}
	\psi^0_{\mathcal{E}^\vee} = q_0\cdot\varphi_{\mathcal{E}}^\vee.
\end{equation}

To make $\psi^1_{\mathcal{E}^\vee}$ more explicit, let us consider the case when $\mathcal{E} = \mathcal{O}^\dagger_{U/\Lambda_0}$ (with $\Lambda=\Lambda_0$), equipped with the Frobenius operator $\varphi_{\Lambda_0}$ and the usual connection. Then  the duality pairing $\langle-,-\rangle^1$ in Lemma~\ref{lem:duality}(ii) takes the following form
\[\Omega^\dagger_{U/\Lambda_0} \times (\mathcal{O}^\dagger_{U/\Lambda_0})_c \to \Lambda_0.\]
Let
\begin{equation}\label{eq:psi}
\psi_{\Lambda_0}:	\Omega^\dagger_{U/\Lambda_0} \to \Omega^\dagger_{U/\Lambda_0}
\end{equation}
denote the adjoint of $\varphi_{\Lambda_0}: (\mathcal{O}^\dagger_{U/\Lambda_0})_c \to (\mathcal{O}^\dagger_{U/\Lambda_0})_c$. Identifying $\Omega^\dagger_{U/\Lambda_0} $ with $\Lambda_0$-linear dual of $(\mathcal{O}^\dagger_{U/\Lambda_0})_c $, it follows that $\psi_{\Lambda_0} = q_0\cdot \varphi_{\Lambda_0}^\vee$. We then have
\[
\psi^1_{\mathcal{E}^\vee} = \psi^0_{\mathcal{E}^\vee}\otimes \psi_{\Lambda_0}.
\]

\subsection{Rigid cohomology with coefficients}
Let us recall the definition of rigid cohomology with and without compact support with coefficients in $\Lambda_0$-$F$-isocrystals with $\Lambda$-actions.

\begin{definition}
	Let $\mathcal{E}$ be an overconvergent $\Lambda_0$-$F$-isocrystal over $U$ with $\Lambda$-action. Suppose that $U$ is affine. Then we set
	\[R\Gamma_{\rm rig}(U/\Lambda_0,\mathcal{E} ^\vee):= \left[ \mathcal{E} ^\vee \xrightarrow{\nabla_{\mathcal{E} ^\vee}}\mathcal{E} ^\vee\otimes_{\mathcal{O}^\dagger_{U/\Lambda_0}}\Omega^\dagger_{U/\Lambda_0} \right], \]
	 which is a complex of Fr\'echet $\Lambda$-spaces concentrated in degrees $[0,1]$. 
\end{definition}
Note that this complex represents the rigid cohomology with coefficients in $\mathcal{E}$ viewed as an overconvergent $\Lambda_0$-$F$-crystal. Furthermore, the Dwork operator $\psi_{\mathcal{E}^\vee}^\bullet$ as in Definition~\ref{def:psi} acts on the complex $R\Gamma_{\rm rig}(U/\Lambda_0,\mathcal{E} ^\vee) $ as a \emph{nuclear} operator on each term; \emph{cf.}  \cite[Lem.~5.2]{ELS}.

 To define the compactly supported variant, let us recall that we have a derivation
 \[
 d:\xymatrix@1{(\mathcal{O}^\dagger_{U/\Lambda_0})_c  =   H^1_{]U[_{\Lambda_0}}(\mathcal{X}_{\Lambda_0}, \mathcal{O}_{\mathcal{X}_{\Lambda_0}})\ar[r]& H^1_{]U[_{\Lambda_0}}(\mathcal{X}_{\Lambda_0}, \Omega_{\mathcal{X}_{\Lambda_0}})=(\Omega^\dagger_{U/\Lambda_0})_c}
 \]
induced by the universal derivation $d: \mathcal{O}_{\mathcal{X}_{\Lambda_0}} \to \Omega_{\mathcal{X}_{\Lambda_0}} $.

\begin{definition}
For $\mathcal{E}$ as before, 	we set
\[R\Gamma_{{\rm rig},c}(U/\Lambda_0,\mathcal{E}):= \left[0\to \mathcal{E}\otimes_{\mathcal{O}^\dagger_{U/\Lambda_0}}(\mathcal{O}^\dagger_{U/\Lambda_0})_c \xrightarrow{\nabla_{\mathcal{E}}\otimes d}\mathcal{E}\otimes_{\mathcal{O}^\dagger_{U/\Lambda_0}}(\Omega^\dagger_{U/\Lambda_0})_c \right],\]
which is a complex of Fr\'echet $\Lambda$-spaces concentrated in degrees $[1,2]$. 
\end{definition}
Note that this complex represents the compactly supported rigid cohomology with coefficients in $\mathcal{E}$ viewed as an overconvergent $\Lambda_0$-$F$-crystal. Furthermore, $\varphi_{\mathcal{E}}$
 induces a natural $(\varphi_{\Lambda_0}\otimes\Lambda)$-semilinear operator $\varphi_{\mathcal{E}_c}^\bullet$ on the complex $R\Gamma_{{\rm rig},c}(U/\Lambda_0,\mathcal{E}) $, which is a \emph{nuclear} operator on each term; \emph{cf.}  \cite[Lem.~5.2]{ELS}.

%From Lemmas~\ref{lem:duality} and \ref{lem:duality-frob}, we can deduce the following proposition:
\begin{proposition}
The duality pairing $\langle-,-\rangle^\bullet$ defined in Lemma~\ref{lem:duality}(ii) induces a natural $\Lambda$-linear isomorphism
\[R\Gamma_{\rm rig}(U/\Lambda_0,\mathcal{E}^\vee)\cong R\Hom_{\Lambda}(R\Gamma_{{\rm rig},c}(U/\Lambda_0,\mathcal{E}),\Lambda[2]).\]
Furthermore, the Dwork operator $\psi_{\mathcal{E}^\vee}^\bullet$ corresponds to the $\Lambda$-linear dual of $\varphi_{\mathcal{E}_c}^\bullet$ via this isomorphism.
\end{proposition}
\begin{proof}
By repeating the proof of Lemma~\ref{lem:duality}(ii), the first claim can be reduced to the case when $U=\mathbb{A}^1$ and $X=\mathbb{P}^1$, which is proved in \cite[\S3.3]{ELS}. The second claim directly follows from Lemma~\ref{lem:duality-frob}.
\end{proof}

\subsection{$L$-functions and Lefschetz trace formula}

Let $\mathcal{E}$ be an overconvergent $\Lambda_0$-$F$-isocrystal over $U$ with $\Lambda$-action as before. Then given any closed point $x$ of $U$ with residue field $k(x)$, we obtain the fibre $\mathcal{E}_x$ at $x$, which is a finite-rank $W(k(x))\otimes_{W(k_0)}\Lambda$-module equipped with a $\Lambda$-linear $q_0$-Frobenius operator $\varphi_{\mathcal{E},x}$.

To such $\mathcal{E}$ we associate the $L$-functions as follows:
\begin{equation}
	Z_U(\mathcal{E},t) := {\prod}_{x\in |U|} {\rm det}_{\Lambda}(1-(t\cdot\varphi_{\mathcal{E},x})^{[k(x):k_0]}|\ \mathcal{E}_x)^{-1}.
\end{equation}

For any positive integer $r$, let $k_0^{(r)}$ be a degree-$r$ extension of $k_0$ and set 
\begin{equation}
	S_r(U,\mathcal{E}) := {\sum}_{x\in U(k_0^{(r)})} {\rm tr}_{\Lambda}(\varphi_{\mathcal{E},x}^r |\ \mathcal{E}_x),
\end{equation}
which is zero when $k$ is not contained in $k_0^{(r)}$.
%If we set $S(U,\mathcal{E}):= S_1(U,\mathcal{E}) $, then one can easily check that 
%\begin{equation}\label{eq:Sr}
%	S_r(U,\mathcal{E}) = S(U^{(r)},\mathcal{E}^{(r)}),
%\end{equation}
%where $U^{(r)} := U\times_{\Spec k_0} \Spec k_0 ^{(r)}$, and $\mathcal{E}^{(r)} $ is the pull back of $\mathcal{E}$ to $U^{(r)}$.

One can check that
\begin{equation}\label{eq:gen-series}
	Z_U(\mathcal{E},t) = \exp\left({\sum}_{r=1}^\infty S_r(U,\mathcal{E})\cdot t^r/r
	\right);
\end{equation}
\emph{cf.} \cite[2.3]{ELS}.

The main goal of the appendix is to prove the following slight generalisation of the Lefschetz trace formula \cite[Th.~6.3]{ELS}.

\begin{theorem}\label{thm:Lefschetz}
Let $\mathcal{E}$ be an overconvergent $\Lambda_0$-$F$-isocrystal over $U$ with $\Lambda$-action. Then we have
\begin{align*}
Z_U	(\mathcal{E},t) &= \left({\rm det}_\Lambda(1 - t\cdot\varphi^\bullet_{\mathcal{E}_c}|\ R\Gamma_{{\rm rig},c}(U,\mathcal{E})\right)^{-1}\\
&= \left({\rm det}_\Lambda(1 - t\cdot\psi^\bullet_{\mathcal{E}^\vee}|\ R\Gamma_{\rm rig}(U,\mathcal{E}^\vee)\right)^{-1},
\end{align*}
where for $\theta^\bullet= \varphi^\bullet_{\mathcal{E}_c}$ or $\psi^\bullet_{\mathcal{E}^\vee}$ we let ${\rm det}_{\Lambda}(1-t\theta^\bullet)\colon={\prod}_{i}{\rm det}_{\Lambda}(1-t\theta^i)^{(-1)^i}$  denote the alternating product Fredholm determinants of $\Lambda$-linear nuclear operators on each term.
\end{theorem}
Note that the Lefschetz trace formula proved in \cite[Th.~6.3]{ELS} applies only to overconvergent $\Lambda$-$F$-isocrystals. In particular, it does not apply directly to our setting unless $\mathcal{O}_U$ contains the residue field $k_\Lambda$ of $\Lambda$.
\begin{proof}
 By the Poincar\'e duality it suffices to show the first equality (i.e., the equality via the compactly supported rigid cohomology).
	By \eqref{eq:gen-series} it suffices to show
%	\begin{align*}
%	S_r(U,\mathcal{E}) &= {\rm tr}_{\Lambda}((\varphi_{\mathcal{E}_c}^\bullet)^r|\ R\Gamma_{{\rm rig},c}(U,\mathcal{E}))\\
%	&= {\rm tr}_{\Lambda}((\psi_{\mathcal{E}^\vee}^\bullet)^r|\ R\Gamma_{\rm rig}(U,\mathcal{E}^\vee)
%	\end{align*}
\begin{equation}\label{eq:ELS6.2}
S_r(U,\mathcal{E}) = {\rm tr}_{\Lambda}((\varphi_{\mathcal{E}_c}^\bullet)^r|\ R\Gamma_{{\rm rig},c}(U,\mathcal{E}))
\end{equation}
for any $r$. %By replacing $(U,\mathcal{E})$ with $(U^{(r)},\mathcal{E}^{(r)})$ (\emph{cf.} \eqref{eq:Sr}), it suffices to show the above equalities for $r=1$ for any affine smooth curves $U$ over any finite field $k$ containing $k_0$ (i.e., the generalisation of \cite[Th.~6.2]{ELS}). 
And to verify \eqref{eq:ELS6.2} it suffices to show that \begin{equation}\label{eq:ELS5.3}
	{\rm tr}_{\Lambda}(\varphi^\bullet_{\mathcal{E}_c}|\ R\Gamma_{{\rm rig},c}(U,\mathcal{E})) = 0 \quad \text{if }U(k_0^{(r)}) = \emptyset;
\end{equation}
i.e., the generalisation of \cite[Lem.~5.3]{ELS}. Indeed, if $k$ is not contained in $k_0^{(r)}$ then $S_r(U,\mathcal{E}) = 0$. If $k$ is contained in $k_0^{(r)}$, then we may apply the excision sequence for the compactly supported rigid cohomology to the following setting
\[
U(k_0^{(r)})^{\Gal(k_0^{(r)}/k)} \hookrightarrow U \hookleftarrow U\setminus U(k_0^{(r)})^{\Gal(k_0^{(r)}/k)}
\]
and conclude that \eqref{eq:ELS5.3} implies \eqref{eq:ELS6.2}.

To verify \eqref{eq:ELS5.3} we use the following lemma.
\begin{lemma}
Let $R$ be a Fr\'echet $\Lambda$-algebra equipped with a $\Lambda$-linear lift of $q_0^r$-Frobenius morphism $\varphi:R\to R$.
Let $M$ be a Fr\'echet $R$-module equipped with a nuclear $\varphi$-semilinear endomorphism $\varphi_M:M\to M$. Then for any $f\in R$ 	we have
\[{\rm tr}_\Lambda\left((f-\varphi(f))\cdot\varphi_M \right) = 0\]
\end{lemma}
\begin{proof}
(Compare with the proof of Lemma~5.3 in \cite{ELS}.) Let us consider the following morphism of exact sequences:
\[
\xymatrix{
0 \ar[r] & \ker(f) \ar[r] \ar[d]_-0 & M \ar[r]^-{f} \ar[d]_-{\varphi(f)\cdot\varphi_M} & M \ar[r] \ar[d]^-{f\cdot\varphi_M} & M/fM \ar[r] \ar[d]^-0 & 0\\
0 \ar[r] & \ker(f) \ar[r] & M \ar[r]^-{f}  & M \ar[r]  & M/fM \ar[r] & 0
}.
\]
Therefore it follows that ${\rm tr}_\Lambda(\varphi(f)\cdot\varphi_M) = {\rm tr}_\Lambda(f\cdot\varphi_M) $, hence the lemma.
\end{proof}

We will apply the above lemma to $R=\mathcal{O}^\dagger_{U/\Lambda_0}\otimes_{\Lambda_0}\Lambda$ and $r$th iterated Frobenius operators.
Since $U(k_0^{(r)})=\emptyset$, the graph of the $q_0^r$-Frobenius and the diagonal  do not intersect in $U\times U$. Therefore, there exists $f_j,g_j\in R$ such that
\[{\sum}_j f_jg_j = 1 \quad\text{and}\quad {\sum}_j \varphi^r(f_j)g_j=0,\]
so we have $1 = {\sum}_j(f_j-\varphi^r(f_j))\cdot g_j$.

Apply the above lemma when $M$ is one of the terms in $R\Gamma_{{\rm rig},c}(U/\Lambda_0,\mathcal{E})$ and $\varphi_M = g_j(\varphi^\bullet_{\mathcal{E}_c})^r$ for each $j$, we obtain
\begin{align*}
{\rm tr}_\Lambda\big(\varphi_{\mathcal{E}_c}) &=
	{\rm tr}_\Lambda({\sum}_j(f_j-\varphi^r(f_j))g_j(\varphi^\bullet_{\mathcal{E}_c})^r\big)\\
	&={\sum}_j {\rm tr}_\Lambda\big((f_j-\varphi^r(f_j))g_j(\varphi^\bullet_{\mathcal{E}_c})^r\big)=0,
\end{align*}
which proves \eqref{eq:ELS5.3}, hence the theorem.
\end{proof}

%%%%%%%%%%%%%%%%%%%%
%%% BIBLIOGRAPHY %%%
%%%%%%%%%%%%%%%%%%%%

\end{document}